\documentclass[12pt]{article}

\usepackage{amsmath}
\usepackage{amssymb}
\usepackage{fullpage}
\usepackage{amsthm}
\usepackage{graphicx}
\usepackage{placeins}
\usepackage{caption}
\usepackage{extpfeil}
\usepackage{enumerate}
\usepackage{caption}
\usepackage{subcaption}

\usepackage{multirow}%表格中的multirow

\usepackage{hyperref}
\usepackage{xcolor}
\usepackage{pbox}

\newcommand{\be}{\begin{equation}} 
\newcommand{\ee}{\end{equation}}
\newcommand{\bee}{\begin{equation*}}
\newcommand{\eee}{\end{equation*}}

\newcommand{\dive}{\mbox{div}}

\newcommand{\ctthe}{\cot\theta}
\newcommand{\sthe}{\sin\theta}
\newcommand{\cthe}{\cos\theta}

\newtheorem{thm}{Theorem}[section]

\newtheorem{rmk}{Remark}[section]
\newtheorem{lem}{Lemma}[section]

\newtheorem{cor}{Corollary}[section]\label{key}

\begin{document}

\title{Vanishing viscosity limit for homogeneous axisymmetric no-swirl solutions of stationary Navier-Stokes equations}
\author{Li Li\footnote{Department of Mathematics, Harbin Institute of Technology, Harbin 150080, China. Email: lilihit@126.com}, 
YanYan Li\footnote{Department of Mathematics, Rutgers University, 110 Frelinghuysen Road, Piscataway, NJ 08854, USA. Email: yyli@math.rutgers.edu}, 
Xukai Yan\footnote{School of Mathematics, Georgia Institute of Technology, 686 Cherry St NW, Atlanta, GA 30313, USA. Email: xukai.yan@math.gatech.edu}}
\date{}
\maketitle

\abstract{$(-1)$-homogeneous axisymmetric no-swirl solutions of three dimensional incompressible stationary Navier-Stokes equations which are smooth on the unit sphere minus the north and south poles have been classified. 
  In this paper we study the vanishing viscosity limit of sequences of these solutions. As the viscosity tends to zero, some  sequences of solutions $C^m_{loc}$ converge to solutions of Euler equations on the sphere minus the poles,  while for other sequences of solutions, transition layer behaviors occur. 
For every latitude circle, there are sequences which $C^m_{loc}$ converge respectively to different solutions of the Euler equations on the spherical caps above and below the latitude circle.  
We give detailed analysis of these convergence and transition layer behaviors. }
%%%%%%%%%%%%%%%%%%%%%%%%%%%%%%%%%%%%%%%%%%%%%%%%%%%%%%%%%%%%%%%%%%%%%%%%%%%%
%%%%%%%%%%%%%%%%  SECTION 1 %%%%%%%%%%%%%%%%%%%%%%%%%%%%%%%%%%%%%%%%%%%%%%%
%%%%%%%%%%%%%%%%%%%%%%%%%%%%%%%%%%%%%%%%%%%%%%%%%%%%%%%%%%%%%%%%%%%%%%%%%%%%

\section{Introduction}\label{sec:intro}

%\subsection{Reorganize Introduction}

%*******************Reorganize Introduction *************

We consider $(-1)$-homogeneous solutions of incompressible stationary Navier-Stokes equations in $\mathbb{R}^3$:
\begin{equation}\label{NS}
\left\{
\begin{split}
	& -\nu \triangle u + u\cdot \nabla u +\nabla p = 0, \\
	& \dive\textrm{ } u=0.
\end{split}
\right.
\end{equation}

The incompressible stationary Euler equations in $\mathbb{R}^3$ are given by:
\begin{equation}\label{Euler}
\left\{
\begin{split}
	& v\cdot \nabla v +\nabla q = 0, \\
	& \dive\textrm{ } v=0.
\end{split}
\right.
\end{equation}

Equations (\ref{NS}) and (\ref{Euler}) are invariant under the scaling $u(x)\to \lambda u(\lambda x)$ and $p(x)\to \lambda^2 p(\lambda x)$, $\lambda>0$. We study solutions which are invariant under the scaling. For such solutions $u$ is $(-1)$-homogeneous and $p$ is $(-2)$-homogeneous. We call them $(-1)$-homogeneous solutions according to the homogeneity of $u$.

Landau discovered in \cite{Landau} a three parameter family of explicit $(-1)$-homogeneous solutions of (\ref{NS}), which are axisymmetric with no swirl. Tian and Xin proved in \cite{TianXin} that all $(-1)$-homogeneous, axisymmetric nonzero solutions of (\ref{NS}) which are smooth on the unit sphere $\mathbb{S}^2$ are Landau solutions. They also gave in the paper explicit expressions of all $(-1)$-homogeneous axisymmetric solutions of (\ref{Euler}). \v{S}ver\'{a}k proved in \cite{Sverak} that all $(-1)$-homogeneous nonzero solutions which are smooth on $\mathbb{S}^2$ are Landau solutions. We studied in \cite{LLY1} and \cite{LLY2} $(-1)$-homogeneous axisymmetric solutions of (\ref{NS})  which are smooth on $\mathbb{S}^2$ minus the north and south poles. In particular, we classified in \cite{LLY2} all such solutions with no swirl. $(-1)$-homogeneous solutions of (\ref{NS}) and (\ref{Euler}) have been studied in \cite{CK},  \cite{G}, \cite{KP}, \cite{KPS}, \cite{KS}, \cite{Luo-Shvydkoy}, \cite{MT}, \cite{PP1}, \cite{PP2}, \cite{PP3}, \cite{Serrin}, \cite{Shvydkoy}, \cite{SL}, \cite{SQ}, \cite{W} and \cite{Y}.

%Serrin \cite{Serrin}, Luo and Shvydkoy \cite{Luo-Shvydkoy}, and Shvydkoy \cite{Shvydkoy}. See also references in \cite{LLY1} and \cite{LLY2}.%$(-1)$-homogeneous axisymmetric no-swirl solutions of (\ref{NS}).

In spherical coordinates $(r,\theta, \phi)$, where $r$ is the radial distance from the origin, $\theta$ is the angle between the radius vector and the positive $x_3$-axis, and $\phi$ is the meridian angle about the $x_3$-axis, a vector field $u$ is written as
%We study the above two equations in spherical coordinates $(r,\theta,\phi)$. A vector field $u$ can be written as
\[
%\label{u_polar}
	u = u_r \vec{e}_r+ u_\theta \vec{e}_{\theta} + u_\phi \vec{e}_{\phi},
\]
where
\[
	\vec{e}_r = \left(
	\begin{matrix}
		\sthe\cos\phi \\
		\sthe\sin\phi \\
		\cthe
	\end{matrix} \right),  \hspace{0.5cm}
	\vec{e}_{\theta} = \left(
	\begin{matrix}
		\cthe\cos\phi  \\
		\cthe\sin\phi   \\
		-\sthe	
	\end{matrix} \right), \hspace{0.5cm}
	\vec{e}_{\phi} = \left(
	\begin{matrix}
		-\sin\phi \\  \cos\phi \\ 0
	\end{matrix} \right).
\]

We use $N$ and $S$ to denote respectively the north and south poles of $\mathbb{S}^2$. 
A vector field $u$ is called axisymmetric if $u_r$, $u_{\theta}$ and $u_{\phi}$ depend only on $r$ and $\theta$, and is called {\it no-swirl} if $u_{\phi}=0$. For any $(-1)$-homogeneous axisymmetric no-swirl solution $(u,p)$ of (\ref{NS}), $u_r$ and $p$ (modulo a constant) can be expressed by $u_{\theta}$ and its derivatives as follows
\begin{equation}\label{eqNS_1}
\begin{split}
   & u_{r} =- \frac{d u_{ \theta}}{d \theta} -\ctthe u_{\theta},  \\
   & 2p=-\frac{d^2 u_{r}}{d\theta^2} - (\ctthe - u_{ \theta}) \frac{d u_{ r}}{d\theta} - u_{r}^2 -u_{\theta}^2.
   \end{split}
\end{equation}
Similarly, for any $(-1)$-homogeneous axisymmetric no-swirl solution $(v,q)$ of (\ref{Euler}), $v_r$ and $q$ can be expressed by $v_{\theta}$ and its derivatives as follows
\begin{equation}\label{eqEuler_1}
%\begin{split}
  v_r=- \frac{d v_{\theta}}{d \theta} -\ctthe v_{\theta},  \quad 
     2q=v_{\theta}\frac{d v_r}{d\theta}-v_r^2-v_{\theta}^2.
  % \end{split}
   \end{equation}

In this paper, we analyze the behavior of any sequence of $(-1)$-homogeneous axisymmetric no-swirl solutions $\{(u_{\nu_k}, p_{\nu_k})\}$ of (\ref{NS}), with vanishing viscosity $\nu_k\to 0$. We will show that in some cases there are subsequences converging to solutions of (\ref{Euler}) on $\mathbb{S}^2$ and in some other cases there are transition layer behaviors. There have been a large amount of research work on vanishing viscosity limit for incompressible Navier-Stokes equations. See for instance \cite{C}, \cite{CC}, \cite{Constantin-Vicol}, \cite{DM}, \cite{DN},  \cite{F}, \cite{Guo-Nguyen}, \cite{Iyer1}, \cite{Iyer2}, \cite{Maekawa}, \cite{Masmoudi},  \cite{SC1}, \cite{SC2}, \cite{Temam-Wang} and \cite{Wang}. %On the other hand, there has not been much work on vanishing viscosity limit for stationary incompressible Navier-Stokes equations.%, which satisfy that for any compact subset $K$ in $\mathbb{S}^2\setminus\{S,N\}$,
%\begin{equation}\label{eq_3}
%   \sup_{k}||u_{\nu_k, \theta}||_{L^2(K)}<\infty.
%\end{equation}

%**************************************

Based on our result in \cite{LLY2}, we have the following theorem.
\begin{thm}\label{thm1_0_0}
  (i) Let $0<\nu<1$, $(u_{\nu}, p_{\nu})$ be $(-1)$-homogeneous axisymmetric no-swirl solutions of (\ref{NS}) which are smooth on $\mathbb{S}^2\setminus\{S,N\}$.
  Then for any $0<\theta_1<\theta_2<\theta_3<\theta_4<\pi$, there exists some positive constant $C$, depending only on the $\{\theta_i\}$, such that
  \[
      \int_{\mathbb{S}^2\cap\{\theta_1<\theta<\theta_4\}}|u_{ \nu, \theta}|^2\le C\left(\int_{\mathbb{S}^2\cap\{\theta_2<\theta<\theta_3\}}|u_{\nu, \theta}|^2+\nu^2\right).
  \]
  
  \medskip
  
  (ii) Let $\nu_k\to 0^+$, $(u_{\nu_k}, p_{\nu_k})$ be $(-1)$-homogeneous axisymmetric no-swirl solutions of (\ref{NS}) which are smooth on $\mathbb{S}^2\setminus\{S,N\}$.
  If  $\displaystyle \sup_{k}\nu_k^{-2}\int_{\mathbb{S}^2\cap\{a<\theta<b\}}|u_{\nu_k, \theta}|^2<\infty$ for some $-1<a<b<1$,
 then there exists some solution $(\tilde{u}, \tilde{p})$ of (\ref{NS}) with $\nu=1$ which is smooth on $\mathbb{S}^2\setminus\{S,N\}$, such that, after passing to a subsequence, for any $\epsilon>0$, and any integer $m$, 
  \[
  \lim_{k\to \infty}||(\frac{u_{\nu_k}}{\nu_k}, \frac{p_{\nu_k}}{\nu_k^2})-(\tilde{u},\tilde{p})||_{C^m(\mathbb{S}^2\cap\{\epsilon<\theta<\pi-\epsilon\})}=0.
     %\lim_{k\to \infty}||\frac{1}{\nu_k}u_{\nu_k, \theta}-\tilde{u}_{\theta}||_{C^m(\mathbb{S}^2\cap\{\epsilon<\theta<\pi-\epsilon\})}=0.
  \]
 % \marginpar{convergence in strong norm instead of $L^2$?}
\end{thm}

%As in (\ref{NS}) and (\ref{Euler}), we work with variable $x:=\cos\theta$ and vectors $U:=u\sin\theta$. We use $"$ $'$ $"$ to denote the derivative with respect to $x$. It is known, see e.g. \cite{LLY1}, that  all (-1)-homogeneous axisymmetric no-swirl solutions of (\ref{NS}) on $\mathbb{S}^2\setminus\{S,N\}$ are given by  $u=U'_{\theta}(x)\vec{e}_r+\sin\theta U_{\theta}(x)\vec{e}_{\theta}$ where  $U_{\theta}$ satisfies
%\be \label{eq:NSE}
%	\nu (1-x^2)U_\theta' + 2\nu x U_\theta + \frac{1}{2}U_\theta^2 = P_c(x):=c_1(1-x)+c_2(1+x)+c_3(1-x^2), \quad -1<x<1.
%\ee
%for some $c=(c_1,c_2,c_3)$.

 %Similarly, let $V=v\sin\theta $, then all $(-1)$-homogeneous axisymmetric solutions of Euler equations (\ref{Euler}) are given by $v=V'_{\theta}\vec{e}_r+\sin\theta V_{\theta}\vec{e}_{\theta}+a\vec{e}_{\phi}$, where $a$ is a constant and $V_{\theta}$ satisfies, for some $c$,
%\be \label{eq:EE}
%	\frac{1}{2}V_\theta^2 = P_c(x).
%\ee
%where $P_c(x)$ is a second order polynomial of $x$ given in (\ref{eq:NSE}).
%The original Euler equation is equivalent to (\ref{eq:EE}), and
%\[
%	V_r = V_\theta' \sin\theta, \quad \quad  V_\phi=a, \quad \quad p
%	= -\frac{1}{2} \left( \frac{1}{2} (V_\theta^2)'' +\csc^2\theta V_\theta^2 + a^2\csc^2\theta\right).
%\]

%Let $(v^{\pm}(c), p^{\pm}(c))$ be solutions of (\ref{Euler}) given by $V_{\theta}^{\pm}(c)(x)=\pm \sqrt{2P_c(x)}$ respectively, and
%\[
%   J_0:=\{c\in \mathbb{R}^3\mid c=(c_1,c_2,c_3), c_1\ge 0, c_2\ge 0, c_3\ge -\frac{1}{2}(c_1+2\sqrt{c_1c_2}+c_2)\}.
%\]

As in \cite{LLY1} and \cite{LLY2}, we work with variable $x:=\cos\theta$ and vector $U:=u\sin\theta$. We use $"$ $'$ $"$ to denote the derivative with respect to $x$. 

For $\nu\ge 0$, let 
\begin{equation}\label{eqdef_1}
     \bar{c}_3(c_1,c_2; \nu)=-\frac{1}{2}(\sqrt{\nu^2+c_1}+\sqrt{\nu^2+c_2})(\sqrt{\nu^2+c_1}+\sqrt{\nu^2+c_2}+2\nu),
\end{equation}
and introduce
\begin{equation*}%\label{eqdef_2}
     J_\nu:=\{c\in \mathbb{R}^3\mid c=(c_1,c_2,c_3), c_1\ge -\nu^2, c_2\ge -\nu^2, c_3\ge \bar{c}_3(c_1,c_2;\nu)\}.
     \end{equation*}
%     Note that 
%     \[
%   J_0=\{c\in \mathbb{R}^3\mid c=(c_1,c_2,c_3), c_1\ge 0, c_2\ge 0, c_3\ge -\frac{1}{2}(c_1+2\sqrt{c_1c_2}+c_2)\}.
%\]
It is easy to see that $J_{\nu'}\subset J_{\nu}$ for any $0\le \nu'\le \nu$. We use $\mathring{J}_{\nu}$ to denote the interior of $J_{\nu}$. For $\nu>0$, it is known from Theorem 1.2 in \cite{LLY1} that  all (-1)-homogeneous axisymmetric no-swirl solutions of (\ref{NS}) which are smooth in  $\mathbb{S}^2\setminus\{S,N\}$ are given by  $u=U'_{\theta}(x)\vec{e}_r+\sin\theta U_{\theta}(x)\vec{e}_{\theta}$ where  $U_{\theta}$ satisfies
\be \label{eq:NSE}
	\nu (1-x^2)U_\theta' + 2\nu x U_\theta + \frac{1}{2}U_\theta^2 = P_c(x):=c_1(1-x)+c_2(1+x)+c_3(1-x^2), \quad -1<x<1,
\ee
for some $c=(c_1,c_2,c_3)\in J_{\nu}$.

Let $\tilde{U}_{\theta}:=\frac{U_{\theta}}{\nu}$, then $U_{\theta}$ is a solution of (\ref{eq:NSE}) if and only if $\tilde{U}_{\theta}$ is a solution of
\begin{equation}\label{eqNSE_1}
    (1-x^2)\tilde{U}_\theta' + 2 x \tilde{U}_\theta + \frac{1}{2}\tilde{U}_\theta^2 = P_{\frac{c}{\nu^2}}(x), \quad -1<x<1.
\end{equation}

 Similar to the above, let $V=v\sin\theta $, then all $(-1)$-homogeneous axisymmetric solutions of Euler equations (\ref{Euler}) are given by $v=V'_{\theta}\vec{e}_r+\sin\theta V_{\theta}\vec{e}_{\theta}+a\vec{e}_{\phi}$, where $a$ is a constant and $V_{\theta}$ satisfies, for some $c$,
\be \label{eq:EE}
	\frac{1}{2}V_\theta^2 = P_c(x), 
\ee
where $P_c(x)$ is the second order polynomial given in (\ref{eq:NSE}).  Introduce a subset of $\partial J_0$:
\[
   \partial'J_0:=\{(0,0,c_3)\mid c_3>0\}\cup \{(c_1,0,c_3)\mid c_1>0, c_3\ge -\frac{1}{2}c_1\}\cup\{(0,c_2,c_3)\mid c_2>0,c_3\ge -\frac{1}{2}c_2\}
\]

By Lemma \ref{lem7_1} in the Appendix, $P_c\ge 0$ on $[-1,1]$ if and only if $c\in J_0$; $P_c>0$ on [-1,1] if and only if $c\in \mathring J_0$; and $P_c>0$ on $(-1,1)$ if and only if $c\in \mathring J_0\cup \partial' J_0$.

%Let 
%\[
%   J_0:=\{c\in \mathbb{R}^3\mid c=(c_1,c_2,c_3), c_1\ge 0, c_2\ge 0, c_3\ge -\frac{1}{2}(c_1+2\sqrt{c_1c_2}+c_2)\},
%\]
%
%$$
%  P_c(x)=c_1(1-x)+c_2(1+x)+c_2(1-x^2), -\quad 1\le x\le 1.
%$$
%
%We use $\mathring{J}_0$ to denote the interior of $J_0$. 
%It is easy to see that $P_c\ge 0$ on $[-1,1]$ if and only if $c\in J_0$.

%%%%%%%%%%%%%%%%

For $c\in J_{0}$, let $v^{\pm}_c=v^{\pm}_{c, r}\vec{e}_r + v^{\pm}_{c, \theta} \vec{e}_\theta$, where
$$
v^{\pm}_{c, \theta} (r, \theta, \varphi)=\pm \frac { \sqrt{ 2P_c(\cos\theta) } }{r \sin\theta}, \quad 
v^{\pm}_{c, r}(r, \theta, \varphi)=\pm \frac{P_c'(\cos\theta)}{r\sqrt{2P_c(\cos\theta)}}, %(-2\cos\theta+c_2-c_1)
$$
and
$$
   q_c(r, \theta, \varphi)=-\frac{1}{2r^2}(P_c''(\cos\theta)+\frac{2P_c(\cos\theta)}{\sin^2\theta}).
$$

It is easy to see from the above (see also \cite{TianXin}) that $\{(v^{\pm}_c, q_c)\ |\ c\in \mathring J_0\cup \partial' J_0\}$ is the set of (-1)-homogeneous axisymmetric no-swirl
solutions of (\ref{Euler}) which are smooth in $\mathbb S^2 \setminus\{S, N\}$.

%********************
Next, we prove that if a sequence of $(-1)$-homogeneous axisymmetric no-swirl solutions $\{(u_{\nu_k}, p_{\nu_k})\}$ of (\ref{NS}) converges weakly in $L^2(\mathbb{S}^2\cap \{\theta_1<\theta<\theta_2\})$ to  $(v^{+}_{c},q_c)$ or $(v^{-}_{c},q_c)$ for some $c\in \mathring{J}_0$, then the convergence is $C^m_{loc}$ for any positive integer $m$. More precisely we have the following theorem.

% \begin{thm}\label{thm1_0_1}
%     For $0<\theta_1<\theta_2<\pi$ and $\nu_k\to 0^+$,  let $\{(u_{\nu_k}, p_{\nu_k})\}$ and $(v,q)$ be smooth $(-1)$-homogeneous solutions of (\ref{NS}) and (\ref{Euler}) respectively in the open cone in $\mathbb{R}^3$ generated by $\mathbb{S}^2\cap \{\theta_1<\theta<\theta_2\}$. Assume that $\{u_{\nu_k, \theta}\}$ weakly converges to $v$ in $L^2(\mathbb{S}^2\cap \{\theta_1<\theta<\theta_2\})$. Then for any $\epsilon>0$ and any positive integer $m$, there exists some constant $C$, depending only on $\theta_1,\theta_2, \epsilon, m$ and $\sup_{\nu_k}||u_{\nu_k, \theta}||_{L^2(\mathbb{S}^2\cap \{\theta_1<\theta<\theta_2\})}$, such that 
%     \[
%        ||(u_{\nu_k}, p_{\nu_k})-(v,q)||_{C^m(\mathbb{S}^2\cap \{\theta_1+\epsilon<\theta<\theta_2-\epsilon)}\le C\nu_k.
%     \]
%  \end{thm}
  
   \begin{thm}\label{thm1_0_1}
     For $0<\theta_1<\theta_2<\pi$ and $\nu_k\to 0^+$,  let $\{(u_{\nu_k}, p_{\nu_k})\}$  be smooth $(-1)$-homogeneous, axisymmetric, no-swirl solutions of (\ref{NS}) in the open cone in $\mathbb{R}^3$ generated by $\mathbb{S}^2\cap \{\theta_1<\theta<\theta_2\}$. Assume that $\{u_{\nu_k, \theta}\}$ weakly converges to $v=v^+_c$ or $v^-_c$ in $L^2(\mathbb{S}^2\cap \{\theta_1<\theta<\theta_2\})$ for some $c\in \mathring{J}_0$. Then for any $\epsilon>0$ and any positive integer $m$, there exists some constant $C$, depending only on $\theta_1,\theta_2, \epsilon, m$ and $\sup_{\nu_k}||u_{\nu_k, \theta}||_{L^2(\mathbb{S}^2\cap \{\theta_1<\theta<\theta_2\})}$, such that 
     \[
        ||(u_{\nu_k}, p_{\nu_k})-(v, q_c)||_{C^m(\mathbb{S}^2\cap \{\theta_1+\epsilon<\theta<\theta_2-\epsilon)}\le C\nu_k.
     \]
  \end{thm}
 
In the above theorem we have only analyzed axisymmetric no-swirl solutions $\{u_\nu, p_\nu\}$.  Concerning
general solutions we raise the following

\medskip

\noindent {\bf Question 1.}\   Let $\Omega\subset \mathbb{S}^2$ be an open set, and let $\{ (u_{\nu_k}, p_{\nu_k})\}$, $\nu_k\to 0^+$, and
$(v, q)$ be smooth $(-1)$-homogeneous solutions of (\ref{NS}) and (\ref{Euler})  respectively in the open cone
in $\mathbb{R}^3$ generated by
$\Omega$.   Assume that $u_{\nu_k}$ weakly  converges to $v$ in $L^2(\Omega)$ as ${\nu_k} \to 0^+$. Is it true that %for every
%compact subset $K$ of
%$\Omega$ and
for every  non-negative integer $m$,
$\{(u_{\nu_k}, p_{\nu_k})\}$ converges  to $(v,q)$  in $C_{loc}^m(\Omega)$?

%\bigskip
%
%We also raise the following analogous question for two dimensional stationary incompressible Navier-Stokes equations.
%
%\medskip
%
%
%
%\noindent {\bf Question 2.}\    Let $\Omega\subset \mathbb{R}^2$
% be an open set, and let $\{(u_{\nu_k}, p_{\nu_k})\}$, $\nu_k\to 0^+$, and
%$(v, q)$ be respectively smooth solutions of   (\ref{NS}) and (\ref{Euler}) in $\Omega$. % respectively.
%  Assume that $u_{\nu_k}$ weakly  converges to $v$ in $L^2(\Omega)$ as $\nu_k \to 0^+$. Is it true that %for every
% %$\Omega$ and
%for every  non-negative integer $m$,
%$\{(u_{\nu_k}, p_{\nu_k})\}$ converges  to $(v,q)$  in $C_{loc}^m(\Omega)$?

%********************

Given part (ii) of Theorem \ref{thm1_0_0}, we will only consider below the behavior of $(u_{\nu_k},p_{\nu_k})$ when $\nu_k^{-2}\int_{\mathbb{S}^2\cap\{\frac{\pi}{4}<\theta<\frac{\pi}{2}\}}|u_{\nu_k, \theta}|^2\to \infty$ as $k\to \infty$. For instance, Theorem \ref{thm1_0}  below gives asymptotic profiles of $\{(u_{\nu_k},p_{\nu_k})\}$ under the condition.

%****************************************

%%%%%%%%%%%%%%%%%%%%%%

%, and $(u^{\pm}_{\nu}(c), p^{\pm}_{\nu}(c))$ be solutions of (\ref{NS}) given by $U_{\nu, \theta}^{\pm}$ respectively.
\begin{thm}\label{thm1_0}
  %Let  $0<\nu<1$, $0<\theta_0<\pi$.
(i) %
There exist $(-1)$-homogeneous axisymmetric no-swirl solutions
  \newline $\{(u^{\pm}_{\nu}(c), p^{\pm}_{\nu}(c))\}_{0< \nu\le 1}$ of (\ref{NS}), belonging to $C^{0}(\mathring{J}_{\nu}\times (0,1], C^m(\mathbb{S}^2\setminus(B_{\epsilon}(S)\cup B_{\epsilon}(N))))$  for every integer $m\ge 0$,  such that for every compact subset $K\subset \mathring{J}_{0}$, and every  $\epsilon>0$, there exists some constant $C$ depending only on $\epsilon, K$ and $m$, such that
  \[
 ||(u^{\pm}_{\nu}(c), p^{\pm}_{\nu}(c))-(v^{\pm}_c, q_c)||_{C^m(\mathbb{S}^2\setminus\{B_{\epsilon}(S)\cup B_{\epsilon}(N)\})}\le C\nu, \quad  c\in K.
    % \lim_{\nu\to 0}||(u^{\pm}_{\nu}(c), p^{\pm}_{\nu}(c))-(v^{\pm}_c, q_c)||_{C^m(\mathbb{S}^2\setminus\{B_{\epsilon}(S)\cup B_{\epsilon}(N)\})}=0, \quad \textrm{uniformly for } c\in K.
       \]
  (ii) For every $0<\theta_0<\pi$, there exist $(-1)$-homogeneous axisymmetric no-swirl solutions  $\{(u_{\nu}(c, \theta_0), p_{\nu}(c, \theta_0))\}_{0<\nu\le 1}$ of (\ref{NS}), belonging to $C^{0}(\mathring{J}_{\nu}\times(0,1]\times(0,\pi), C^m(\mathbb{S}^2\setminus(B_{\epsilon}(S)\cup B_{\epsilon}(N))))$  for every integer $m\ge 0$, such that for every compact subset $K\subset \mathring{J}_0$, and every  $\epsilon>0$, there exists some constant $C$ depending on $\epsilon, K$ and $m$, such that
  \[
  \begin{split}
     & ||(u_{\nu}(c, \theta_0), p_{\nu}(c, \theta_0))-(v^{+}_c, q_c)||_{C^m(\mathbb{S}^2\cap \{\theta_0+\epsilon<\theta<\pi-\epsilon\})}\\
     & +||(u_{\nu}(c, \theta_0), p_{\nu}(c, \theta_0))-(v^{-}_c, q_c)||_{C^m(\mathbb{S}^2\cap\{\epsilon<\theta<\theta_0-\epsilon\})}\le C\nu, \quad c\in K.
     \end{split}%$(-1)$-
%    \begin{split}
%     & \lim_{\nu\to 0}\left(||(u_{\nu}(c, \theta_0), p_{\nu}(c, \theta_0))-(v^{+}_c, q_c)||_{C^m(\mathbb{S}^2\cap \{\theta_0+\epsilon<\theta<\pi-\epsilon\})}\right.\\
%     & \left.+||(u_{\nu}(c, \theta_0), p_{\nu}(c, \theta_0))-(v^{-}_c, q_c)||_{C^m(\mathbb{S}^2\cap\{\epsilon<\theta<\theta_0-\epsilon\})}\right)=0, \quad \textrm{uniformly for }c\in K.
%     \end{split}
  \]
\end{thm}
Notice that for every $c$ in $\mathring{J}_0$, $P_c>0$ on $[-1,1]$, and $v^{+}_c \ne v^{-}_c$ on $\mathbb{S}^2\cap\{\theta=\theta_0\}$. The limit functions in Theorem \ref{thm1_0} (ii) have jump discontinuities across the circle $\{\theta=\theta_0\}$.

In the following we give more detailed study on the behaviors of $\{(u_{\nu_k}, p_{\nu_k})\}$ which include that in regions where limit functions are not smooth and transition layer behaviors occur.
%*********************

%Define
%\begin{equation}\label{eqdef_3}
%     J_0=\{c\in \mathbb{R}^3\mid c_1\ge 0, c_2\ge 0, c_3\ge \bar{c}_3(c_1,c_2;0)\}.
%  \end{equation}

Define, for $\nu>0$ and $c\in J_{\nu}$,

\begin{equation}\label{eq_2}
\begin{array}{lll}
	&\tau_{1}(\nu,c_1):=2\nu-2\sqrt{\nu^2+c_1}, \quad \quad &\tau_{2}(\nu,c_1):=2\nu+2\sqrt{\nu^2+c_1}, \\
	&\tau_{1}'(\nu,c_2):=-2\nu-2\sqrt{\nu^2+c_2}, \quad \quad &\tau_{2}'(\nu,c_2):=-2\nu+2\sqrt{\nu^2+c_2}.
\end{array}
\end{equation}

By Theorem 1.1 and Theorem 1.3 in \cite{LLY2}, using the scaling in (\ref{eqNSE_1}), we have the following theorem.\\
\textbf{Theorem A} (\cite{LLY2}) \emph{For each $\nu>0$, there exist $U^{+}_{\nu, \theta}(c)(x)\in C^0(J_{\nu}\times [-1,1))$ and $U^{-}_{\nu, \theta}(c)(x)\in C^0(J_{\nu}\times (-1,1])$ such that for every $c\in J_{\nu}$,  $U^{\pm}_{\nu, \theta}(c)\in C^{\infty}(-1,1)$ satisfy (\ref{eq:NSE}) in $(-1,1)$, and $U^{-}_{\nu, \theta}(c)\le U_{\nu, \theta}\le U^+_{\nu, \theta}(c)$ for any solution $U_{\nu, \theta}$ of (\ref{eq:NSE}) in $(-1,1)$.
If $c_3>\bar{c}_3(c_1,c_2, \nu)$,  then $U^{-}_{\nu, \theta}(c)<U^{+}_{\nu, \theta}(c)$ in $(-1,1)$, and the graphs of all  solutions of (\ref{eq:NSE}) foliate the region $\{(x,y)\in \mathbb{R}^2 \mid -1\le x\le 1, U^{-}_{\nu, \theta}(c)\le y \le U^{+}_{\nu, \theta}(c)\}$.
Moreover,
\begin{equation*}%\label{eq_LLY2}
   \begin{split}
      & U_{\nu, \theta}^+(-1)=\tau_2(\nu,c_1), \quad U_{\nu, \theta}^+(1)=\tau_2'(\nu,c_2), \\
      & U_{\nu, \theta}^-(-1)=\tau_1(\nu,c_1), \quad U_{\nu, \theta}^-(1)=\tau_1'(\nu,c_2),
   \end{split}
\end{equation*}
and if $U_{\nu, \theta}$ is a solution other than $U_{\nu, \theta}^\pm$, then
\begin{equation*}%\label{eq_LLY3}
   U_{\nu, \theta}(-1)=\tau_1(\nu,c_1), \quad U_{\nu, \theta}(1)=\tau_2'(\nu,c_2).
\end{equation*}
If $c_3=\bar{c}_3(c_1,c_2,\nu)$, then
\begin{equation*}%\label{eq_LLY1}
   U_{\nu, \theta}^+(c)\equiv U^-_{\nu, \theta}(c)\equiv U^*_{\nu, \theta}(c_1,c_2):= (\nu+\sqrt{\nu^2+c_1})(1-x)+(-\nu-\sqrt{\nu^2+c_2})(1+x).
\end{equation*}
In particular, $U^*_{\nu, \theta}(c_1,c_2)(-1)=\tau_2(\nu,c_1)$ and  $U^*_{\nu, \theta}(c_1,c_2)(1)=\tau_1'(\nu,c_2)$.}

For $c_1,c_2\ge 0$, $c_1+c_2>0$, denote
\begin{equation*}%\label{eqdef_4}
    c_3^*(c_1,c_2)=\bar{c}_3(c_1,c_2; 0)=-\frac{1}{2}(c_1+2\sqrt{c_1c_2}+c_2)< 0,
     \end{equation*}
     %For $c\in J_0\setminus\{0\}$,
     \begin{equation}\label{eqP_1}
        P_{(c_1,c_2)}^*(x):=P_{(c_1,c_2,c_3^*(c_1,c_2))}(x)=%\left\{
       % \begin{split}
       %    & 0, \quad \textrm{ if }c_1=c_2=0,\\
           -c_3^*(c_1,c_2)\left(x-\frac{\sqrt{c_1}-\sqrt{c_2}}{\sqrt{c_1}+\sqrt{c_2}}\right)^2. %\quad \textrm{ if }c_1,c_2\ge 0 \textrm{ and }c_1+c_2>0
       % \end{split}
      %  \right.
     \end{equation}
  Then
     \begin{equation}\label{eqP_2}
          P_c(x)=P_{(c_1,c_2)}^*(x)+(c_3-c_3^*(c_1,c_2))(1-x^2).
          \end{equation}
%The following theorem has been established in \cite{LLY2}, though not explicitly stated there.

%\textbf{Theorem B} \emph{}

Clearly,
          $c_3^*(c_1,c_2)=\min\{c_3\in \mathbb{R}\mid P_c(x)\ge 0 \textrm{ on }[-1,1]\}$.
%and $P_c\ge P_{c_1,c_2}^*\ge 0$ on $[-1,1]$. Moreover,  $\min_{[-1,1]} P_c=0$  if and only if $c_3=c_3^*(c_1,c_2)$ or $c_1c_2=0$. %one of the following conditions happens:

In this paper we will call $U_{\nu, \theta}^+(c)$ and $U_{\nu, \theta}^-(c)$ the upper solution and lower solution of (\ref{eq:NSE}) respectively.

%For every $0<\theta_0<\pi$, let $S_{\theta_0}=\{(r,\theta, \phi)\mid \theta=\theta_0\}$.

%**************************************

Consider sequences $\{(u_{\nu_k},p_{\nu_k})\}$ satisfying (\ref{NS}) with $\nu_k\to 0^+$.  Then $U_{\nu_k, \theta}=u_{\nu_k, \theta}\sin\theta$ satisfies (\ref{eq:NSE}) for some $P_{c_k}$, $c_k\in J_{\nu_k}$. As mentioned ealier, we only consider below the case when $\nu_k^{-2}\int_{\mathbb{S}^2\cap\{a<\theta<b\}}|u_{\nu_k, \theta}|^2\to \infty$ for some $a,b\in (-1,1)$. By Lemma \ref{lem2_0}, this is equivalent to the condition that $\nu_k^{-2}|c_k|\to \infty$. If $\lim_{k\to \infty}\nu_k^{-2}|c_k|<\infty$, then $c_k\to 0$, and by Theorem \ref{thm1_0_0} (ii), $u_{\nu_k}\to 0$ in $C_{loc}^m(\mathbb{S}^2\setminus\{S,N\})$ for every $m$.  %So in this paper we only discuss the behaviors of solutions this case.

The behaviors of $\{U_{\nu_k, \theta}^{\pm}\}$ are different from other solutions. In most cases, $U_{\nu_k, \theta}^{\pm}$ converge to solutions of Euler equation  (\ref{eq:EE}) on all $[-1,1]$, while for other solutions, boundary layer behavior occurs.

We first present the convergence results of  $\{U_{\nu_k, \theta}^{\pm}\}$ on $[-1,1]$. If $\min_{[-1,1]}P_c>0$, we have, after passing to a subsequence,  the convergence of $\{U_{\nu_k, \theta}^{\pm}(c_k)\}$, $c_k\to c$, to the solution of the Euler equation $\pm \sqrt{2P_c}$ on $[-1,1]$. %It is easy to see that $\min_{[-1,1]}P_c>0$ if and only if $c_1,c_2>0$ and $c_3>c_3^*(c_1,c_2)$.

%$\nu_k\to 0^+$, $c_k\in J_{\nu_k}$, $\frac{|c_{\nu}|}{\nu^2}\to \infty$. Assume $\hat{c}_{\nu}:=\frac{c_{\nu}}{|c_{\nu}|}\to \hat{c}$ where $\hat{c}_1,\hat{c}_2>0$ and $\hat{c}_3>c_3^*(\hat{c}_1,\hat{c}_2)$.
\begin{thm}\label{thm1_1}
   Let $\nu_k\to 0^+$, $c_k\in J_{\nu_k}$, $\nu_k^{-2}|c_k|\to \infty$. Assume $\hat{c}_{k}:=|c_k|^{-1}c_k\to \hat{c}\in \mathring{J}_0$.    Then for any $\epsilon>0$ and any positive integer $m$, there exists some constant $C$, depending only on $\epsilon, m$  and $\hat{c}$, such that for large $k$, %, an
   \begin{equation}\label{eq1_1}
   \begin{split}
   & ||U^{+}_{\nu_k, \theta}(c_k)-\sqrt{2P_{c_k}}||_{L^{\infty}(-1,1)}+||U^-_{\nu_k, \theta}(c_k)+\sqrt{2P_{c_k}}||_{L^{\infty}(-1,1)}\le C\nu_k,\\
   &  ||U^+_{\nu_k, \theta}(c_k)-\sqrt{2P_{c_k}}||_{C^m[-1,1-\epsilon]}+||U^-_{\nu_k, \theta}(c_k)+\sqrt{2P_{c_k}}||_{C^m[-1+\epsilon,1]}\le C\nu_k.
   \end{split}
\end{equation}
% for each positive integer $m$.
\end{thm}

%\begin{thm}\label{thm1_1}
%   Let $\nu_k\to 0^+$, $c_k\in J_{\nu_k}$, $c_k\to c\ne 0$, $c_1,c_2>0$ and $c_3>c_3^*(c_1,c_2)$. %If $c_3=c_3^*(c_1,c_2)$ assume $c_k\in \hat{J}_{\nu_k}$.
%   Then for any $\epsilon>0$ and positive integer $m$, there exists some constant $C$, depending only on $\epsilon, m$  and $c$, such that for large $k$%, an upper bound $|c|$ and a positive lower bound of $c_1,c_2$ and $c_3-c_3^*(c_1,c_2)$, such that
%   %\in J_0\setminus\{0\}
%  % \begin{equation}\label{eq1_0}
%%      \limsup_{k\to \infty, 0<\beta<\alpha(c)}\nu^{-\alpha(c)} \left( ||\frac{1}{2}(U^{\pm}_{\nu_k, \theta})^2-P_{c_k}||_{L^{\infty}(-1,1)}+ \nu^{\beta} ||\frac{1}{2}(U^{\pm}_{\nu_k, \theta})^2-P_{c_k}||_{C^{\beta}(-1+\epsilon,1-\epsilon)}\right)<\infty.
%%   \end{equation}
%
% %  Moreover, if $c_1,c_2>0$, %$\min_{[-1,1]}P_c>0$, then
%   \begin{equation}\label{eq1_1}
%   \begin{split}
%   & ||U^{+}_{\nu_k, \theta}-\sqrt{2P_{c_k}}||_{L^{\infty}(-1,1)}+||U^-_{\nu_k, \theta}+\sqrt{2P_{c_k}}||_{L^{\infty}(-1,1)}\le C\nu_k,\\
%   &  ||U^+_{\nu_k, \theta}-\sqrt{2P_{c_k}}||_{C^m(-1,1-\epsilon)}+||U^-_{\nu_k, \theta}+\sqrt{2P_{c_k}}||_{C^m(-1+\epsilon,1)}\le C\nu_k.
%   \end{split}
%\end{equation}
%% for each positive integer $m$.
%\end{thm}

\begin{rmk}
The constant $C$ in Theorem \ref{thm1_1} depends only on $\epsilon, m$, and a positive lower bound of $dist(\hat{c}, \partial J_0)$. %$\hat{c}_1,\hat{c}_2$ and $\hat{c}_3-c_3^*(\hat{c}_1,\hat{c}_2)$. 
Similar statements can be made for Theorem \ref{thm1_2_1},  \ref{thm1_2_2}, \ref{thm1_3} and \ref{thm:BL:1}.
\end{rmk}

\begin{rmk}
In the second estimate in (\ref{eq1_1}), the $\epsilon$ could not be taken as $0$ in general.
\end{rmk}

\bigskip

%\begin{rmk}
%(1)If $c_1,c_2>0,c_3>c_3^*(c_1,c_2)$, then it is equivalent to $\min_{[-1,1]}P_c>0$.
%
%(2)If $c_3=c_3^*(c_1,c_2)<0$ or $c_1c_2=0, c_3>c_3^*(c_1,c_2)$,  then there exists some $\bar{x}\in [-1,1]$ such that $P_{c}(\bar{x})=\min_{[-1,1]}P_c=0$. If  $c_3=c_3^*(c_1,c_2)<0$, then $-1<\bar{x}<1$, or $\bar{x}\in \{-1,1\}$ and $P'_c(\bar{x})=0$, while $c_1c_2=0, c_3>c_3^*(c_1,c_2)$ means  $\bar{x}\in \{-1,1\}$ and $P^2_{c}(\pm 1)+(P'_c(\pm 1))^2\ne 0$.
%\end{rmk}

  In Theorem \ref{thm1_1}, $\hat{c}\in \mathring{J}_0$, %$\hat{c}_1,\hat{c}_2>0$ and $\hat{c}_3>c_3^*(\hat{c}_1,\hat{c}_2)$, 
   which is equivalent to $\min_{[-1,1]}P_{\hat{c}}>0$. If $\min_{[-1,1]}P_{\hat{c}}=0$, i.e. $c\in \partial J_0$, things are more delicate.

  As pointed out later in Section 3, we only need to consider  in Theorem \ref{thm1_1} the special case when $\nu_k\to 0$, $c_k\to c\in \mathring{J}_0$. % where $c_1,c_2>0$ and $c_3>c_3^*(c_1,c_2)$. 
  In the following, we will only state the results in the case $c_k\to c\ne 0$. 
The next two theorems are for $c_3=c_3^*(c_1,c_2)$, i.e. $P_c=P^*_{(c_1,c_2)}$. 

The following results are proved among other things. If $c_k\in J_0$, then $\{U^{\pm}_{\nu_k, \theta}(c_k)\}$ converge to the Euler equation solutions $\pm \sqrt{2P_c}$ in $L^{\infty}(-1,1)$. If $\bar{x}:=\frac{\sqrt{c_1}-\sqrt{c_2}}{\sqrt{c_1}+\sqrt{c_2}}=1$, i.e. $c_2=0$, then $\{U^{+}_{\nu_k, \theta}(c_k)\}$ converges to the Euler equation solution $\sqrt{2P_c}$ in $L^{\infty}(-1,1)$. On the other hand, if $\bar{x}\in [-1,1)$, i.e. $c_2>0$, then there exist examples $\{U^{+}_{\nu_k, \theta}(c_k)\}$ having no convergent subsequence in $L^{\infty}(1-\delta,1)$ for any $\delta>0$. In particular, it has no subsequence converging to a solution of the Euler equation in $L^{\infty}(-1,1)$. 
Similar results are proved for $\{U^{-}_{\nu_k, \theta}(c_k)\}$. If $\bar{x}=-1$, i.e. $c_1=0$, then $\{U^{-}_{\nu_k, \theta}(c_k)\}$ converges to the Euler equation solution $-\sqrt{2P_c}$ in $L^{\infty}(-1,1)$. On the other hand, if $\bar{x}\in (-1,1]$, i.e. $c_1>0$, then there exist examples $\{U^{-}_{\nu_k, \theta}(c_k)\}$ having no convergent subsequence in $L^{\infty}(-1, -1+\delta)$ for any $\delta>0$.  In particular, it has no subsequence converging to a solution of the Euler equation in $L^{\infty}(-1,1)$. 

%If $\bar{x}:=\frac{\sqrt{c_1}-\sqrt{c_2}}{\sqrt{c_1}+\sqrt{c_2}}\in [-1,1)$, i.e. $c_2>0$, then there exist examples with $\inf_{k}||\frac{1}{2}(U_{\nu_k, \theta}^{+})^2-P_{c_k}||_{L^{\infty}(-1,1)}>0$, and in particular, $U_{\nu_k, \theta}^{+}$ does not converge to a solution of the Euler equation. On the other hand, if $\bar{x}=1$, i.e. $c_2=0$, then $U_{\nu_k, \theta}^{+}$ converges to the Euler equation solution $\sqrt{2P_c}$ in $L^{\infty}(-1,1)$. %$||U_{\nu_k, \theta}^{+}-\sqrt{2P_{c}}||_{L^{\infty}[-1,1]}$ converges to zero.
%When $P_c$ is a parabola with $\min_{[-1,1]}P_c=P_c(\bar{x})=0$ for some
%If $\bar{x}$ in $(-1,1]$, i.e. $c_1>0$, then there exist examples with $\displaystyle \inf_{k}||\frac{1}{2}(U_{\nu_k, \theta}^{-})^2-P_{c_k}||_{L^{\infty}(-1,1)}>0$,  and in particular, $U_{\nu_k, \theta}^{-}$ does not converge to a solution of the Euler equation. On the other hand, if $\bar{x}=-1$, i.e. $c_1=1$, then $U_{\nu_k, \theta}^{-}$ converges to the Euler equation solution $-\sqrt{2P_c}$ in $L^{\infty}[-1,1]$. %$||U_{\nu_k, \theta}^{-}-\sqrt{2P_{c}}||_{L^{\infty}[-1,1]}$ converges to zero.
%Moreover, if $c_k\in J_0$, then $U_{\nu_k, \theta}^{\pm}$ converge to the Euler equation solutions $\pm \sqrt{2P_c}$ in $L^{\infty}(-1,1)$. % for any $c$ with $c_3=c_3^*(c_1,c_2)$, we have the convergence of $||U_{\nu_k, \theta}^{\pm}\mp \sqrt{2P_{c}}||_{L^{\infty}[-1,1]}$ to zero.

\begin{thm}\label{thm1_2}
   %Let $\nu_k\to 0^+$, $c_k\in J_{\nu_k}$, $c_k\to c\in J_0\setminus\{0\}$ and $c_3=c_3^*(c_1,c_2)$.
   For any $c\in \partial J_0$ with $c_3=c_3^*(c_1,c_2)$ and $c_2>0$,  there exist some sequences $c_k\in J_{\nu_k}$, $c_k\to c$, $\nu_k\to 0^+$, such that  for any $\epsilon>0$, $\displaystyle \inf_{k}||\frac{1}{2}(U^+_{\nu_k, \theta}(c_k))^2-P_{c_k}||_{L^{\infty}(1-\epsilon,1)}>0$. Similarly, for any $c\in \partial J_0$ with $c_3=c_3^*(c_1,c_2)$ and $c_1>0$, there exist some sequences $c_k\in J_{\nu_k}$, $c_k\to c$, and $\nu_k\to 0^+$, such that for any $\epsilon>0$, $\displaystyle \inf_{k}||\frac{1}{2}(U^-_{\nu_k, \theta}(c_k))^2-P_{c_k}||_{L^{\infty}(-1,-1+\epsilon)}>0$.
%
%   If $c_2=0,c_1>0$,  then there exists some constant $C$, depending only on $c$, such that
%   \[
%      ||\frac{1}{2}(U^+_{\nu_k, \theta})^2-P_{c_k}||_{L^{\infty}(-1,1)}\le C(|c_{k2}|+|2c_{k3}+c_{k1}|^2+\nu_k^{\frac{2}{3}}).
%   \]
%
%   If in addition, $c_k\in \hat{J}_{\nu_k}$, then (\ref{eq1_0}) holds.
\end{thm}

\begin{rmk}
   It is easy to see that the $\{U^+_{\nu_k, \theta}\}$ constructed in Theorem \ref{thm1_2} satisfies  \newline $\inf_{k}||U^+_{\nu_k, \theta}(c_k)-\sqrt{2P_{c_k}}||_{L^{\infty}(1-\epsilon,1)}>0$, and $\{U^+_{\nu_k, \theta}(c_k)\}$  has no convergent subsequence in $L^{\infty}(1-\epsilon, 1)$ for any $\epsilon>0$. Similar statements apply to $\{U^-_{\nu_k, \theta}(c_k)\}$.
\end{rmk}

\begin{thm}\label{thm1_2_1}
   Let $\nu_k\to 0^+$, $c_k\in J_{\nu_k}$, %$\nu_k^{-2}|c_k|\to \infty$
   $c_k\to c\ne 0$ and $c_3=c_3^*(c_1,c_2)$.
   
    (i) If $c_k\in J_0$, then $\lim_{k\to \infty}||U^{\pm}_{\nu_k, \theta}(c_k) \mp \sqrt{2P_{c}}||_{L^{\infty}(-1,1)}=0$, and for any  $0<\beta<2/3$,  there exists some constant $C$, depending only on $c$, $\epsilon$ and $\beta$,  such that for large $k$,%an upper bound $|c|$ and $\beta$,  such that
   \begin{equation*}%\label{eqthm1_2_1_3}
  ||\frac{1}{2}(U^{\pm}_{\nu_k, \theta}(c_k))^2-P_{c_k}||_{L^{\infty}(-1,1)}+\nu_k^{\beta} ||\frac{1}{2}(U^{\pm}_{\nu_k, \theta}(c_k))^2-P_{c_k}||_{C^{\beta}(-1+\epsilon,1-\epsilon)}\le C\nu_k^{2/3}.
   \end{equation*}
   Moreover, $\bar{x}:=\frac{\sqrt{c_1}-\sqrt{c_2}}{\sqrt{c_1}+\sqrt{c_2}}\in [-1,1]$, and for any $\epsilon>0$ and integer $m\ge 0$, there exists some constant $C$, depending only on $c$, $m$ and $\epsilon$, such that for large $k$,
   \[
      ||U^{+}_{\nu_k, \theta}(c_k)-\sqrt{2P_{c_k}}||_{C^{m}([-1,1-\epsilon]\setminus[\bar{x}-\epsilon, \bar{x}+\epsilon])}+  ||U^{-}_{\nu_k, \theta}(c_k)+\sqrt{2P_{c_k}}||_{C^{m}([-1+\epsilon,1]\setminus[\bar{x}-\epsilon, \bar{x}+\epsilon])}\le C\nu_k.
   \]

   (ii) If $c_2=0$, then $\lim_{k\to \infty}||U^+_{\nu_k, \theta}(c_k)-\sqrt{2P_{c}}||_{L^{\infty}(-1,1)}=0$, and there exists some constant $C$, depending only on $c$, such that for large $k$,%an upper bound $|c|$, such that
   \begin{equation}\label{eqthm1_2_1_1}
      ||\frac{1}{2}(U^+_{\nu_k, \theta}(c_k))^2-P_{c_k}||_{L^{\infty}(-1,1)}\le C(|c_{k2}|+|2c_{k3}+c_{k1}|^2+\nu_k^{2/3})=o(1).
   \end{equation}

   (iii) If $c_1=0$, then $\lim_{k\to \infty}||U^-_{\nu_k, \theta}(c_k)+\sqrt{2P_{c}}||_{L^{\infty}(-1,1)}=0$, and there exists some constant $C$, depending only on $c$, such that for large $k$, %an upper bound $|c|$, such that
   \begin{equation*}%\label{eqthm1_2_1_2}
      ||\frac{1}{2}(U^-_{\nu_k, \theta}(c_k))^2-P_{c_k}||_{L^{\infty}(-1,1)}\le C(|c_{k1}|+|2c_{k3}+c_{k2}|^2+\nu_k^{2/3})=o(1).
   \end{equation*}

\end{thm}

%Note that the right hand sides of (\ref{eqthm1_2_1_1}) and (\ref{eqthm1_2_1_2}) tend to zero as $k\to \infty$.

%Notice that in Theorem \ref{thm1_2}, the condition $c_3=c_3^*(c_1,c_2)$ is equivalent to that $P_c$ is a parabola with the minimum point at $-1$ or $1$, and $\min P_c=0$.
%\addtocounter{thm}{-1}
%\renewcommand{\thethm}{\thesection.\arabic{thm}'}%
%\begin{thm}\label{thm1_2'}
%   Let $c\in J_0\setminus\{0\}$. If $c_1>0$, $c_3=c_3^*$, then there exist some sequences $c_k\in J_{\nu_k}$, $c_k\to c$, $\nu_k\to 0$, and some constant $\epsilon>0$, such that $||\frac{1}{2}(U^-_{\nu_k, \theta})^2-P_{c_k}||_{L^{\infty}(-1,1)}>\epsilon$ for all $k$.
%
%   If $c_1=0,c_2>0$, $c_3=c_3^*$, then there exists some constant $C$, depending only on $c$, such that
%   \[
%      ||\frac{1}{2}(U^-_{\nu_k, \theta})^2-P_{c_k}||_{L^{\infty}(-1,1)}\le C(|c_{k1}|+|2c_{k3}+c_{k2}|^2+\nu_k^{2/3}).
%   \]
%\end{thm}
%\renewcommand{\thethm}{\thesection.\arabic{thm}}%

\bigskip

Next, we discuss the remaining cases when $\min_{[-1,1]}P_c=0=P_c(-1)$ and $P_c'(-1)> 0$ or $\min_{[-1,1]}P_c=0=P_c(1)$ and $P_c'(1)<0$. This is equivalent to $c_1=0$ and $c_3>c_3^*(c_1,c_2)$ or $c_2=0$ and $c_3>c_3^*(c_1,c_2)$. 
%case when  $\min_{[-1,1]}P_c=0=P_c(-1)$ (or $P_c(1)$), and $P_c'(-1)\ne 0$ (or $P_c'(1)\ne 0$). This is equivalent to $c_3>c_3^*(c_1,c_2)$ with $c_1=0$ (or $c_2=0$). 
In this case, $U_{\nu_k, \theta}^{\pm}(c_k)$ converge respectively to the Euler equation solutions $\pm \sqrt{2P_c}$ in $L^{\infty}(-1,1)$.%the derivative at the minimum point on $[-1,1]$ is not zero. This is equivalent to $c_3>c_3^*(c_1,c_2)$ with $c_1=0$ or $c_2=0$.

\begin{thm}\label{thm1_2_2}
   Let $\nu_k\to 0^+$, $c_k\in J_{\nu_k}$, $c_k\to c\ne 0$, $c_1c_2=0$ and $c_3>c_3^*(c_1,c_2)$. %If $c_3=c_3^*(c_1,c_2)$ assume $c_k\in \hat{J}_{\nu_k}$.
   Then
    \[
      \lim_{k\to \infty}||U^{\pm}_{\nu_k, \theta}(c_k) \mp \sqrt{2P_{c}}||_{L^{\infty}(-1,1)}=0.
   \]
   Moreover, for any $\epsilon>0$ and integer $m\ge 0$, there exists some constant $C>0$, depending only on $\epsilon$, $\beta$, and $c$, such that for large $k$, %an upper bound of $|c|$, such that
   \begin{equation*}%\label{eq1_0}
   \nu_k^{1/2}||\frac{1}{2}(U^{\pm}_{\nu_k, \theta}(c_k))^2-P_{c_k}||_{L^{\infty}(-1,1)}+ ||U^{\pm}_{\nu_k, \theta}(c_k) \mp \sqrt{2P_{c}}||_{C^{m}(-1+\epsilon,1-\epsilon)}\le C\nu_k.
  % ||\frac{1}{2}(U^{\pm}_{\nu_k, \theta})^2-P_{c_k}||_{L^{\infty}(-1,1)}+\nu_k^{\beta} ||\frac{1}{2}(U^{\pm}_{\nu_k, \theta})^2-P_{c_k}||_{C^{\beta}(-1+\epsilon,1-\epsilon)}\le C\nu_k^{1/2}.
   \end{equation*}
\end{thm}

\bigskip

Theorem \ref{thm1_1} - Theorem \ref{thm1_2_2} present some convergent results of $U_\theta^{\pm}$. To help the readers to have a general picture of these results, we summarize the $L^\infty$ convergent results of $U_{\nu_k,\theta}^{\pm}$ in Table \ref{tab:1}, where we always assume that $\nu_k\to 0^+$, $c_k\in J_{\nu_k}$, $c_k\to c\ne 0$. 

\begin{table}
\caption{Summary of convergent results for $U_{\nu_k,\theta}^{\pm}$}
\label{tab:1}
\renewcommand{\arraystretch}{1.3}
\begin{center}
\begin{tabular}{|c|c|p{6.3cm}|p{6.3cm}|}
\hline
\multicolumn{2}{|c|}{\multirow{2}{*}{Conditions}} & \multicolumn{2}{|c|}{Conclusions (True/False)} \\
\cline{3-4}
\multicolumn{2}{|c|}{} & $\displaystyle \lim_{k\to \infty}\|U^{+}_{\nu_k, \theta}(c_k) - \sqrt{2P_{c}}\|_{L^{\infty}(-1,1)}=0$ 
& $\displaystyle \lim_{k\to \infty}\|U^{-}_{\nu_k, \theta}(c_k) + \sqrt{2P_{c}}\|_{L^{\infty}(-1,1)}=0$ \\
\hline\hline
\multicolumn{2}{|c|}{$c\in \mathring{J}_0$} & True & True \\
\hline 
\multirow{5}{*}{$c_3=c_3^*$} & $c_2>0$ & {False: $\forall \epsilon>0, \exists$ non-convergent \newline sequence $\{U_{\nu_k,\theta}^+\}$ in $(1-\epsilon,1)$.} &  \\
\cline{2-4}
& $c_1>0$ & & False: $\forall \epsilon>0, \exists$ non-convergent \newline sequence $\{U_{\nu_k,\theta}^-\}$ in $(-1,-1+\epsilon)$. \\
\cline{2-4}
& $c_k\in J_0$ & True & True \\
\cline{2-4}
& $c_2=0$ & True &  \\
\cline{2-4}
& $c_1=0$ & & True  \\
\hline 
\multicolumn{2}{|c|}{$c_3>c_3^*$, $c_1c_2=0$} & True & True  \\
\hline
\end{tabular}
\end{center}
\end{table}

We now present  results for solutions $U_{\nu_k, \theta}$ of (\ref{eq:NSE}) other than $U^{\pm}_{\nu_k, \theta}(c_k)$.

For $c\in J_0\setminus\{0\}$,  define
\begin{equation}\label{eq_alpha}
   \alpha(c)=\left\{
      \begin{array}{ll}
         1, & \textrm{ if }c_1,c_2>0,c_3>c_3^*(c_1,c_2),\\
         \frac{2}{3}, & \textrm{ if }c_3=c_3^*(c_1,c_2)<0,\\
         \frac{1}{2}, & \textrm{ if }c_1c_2=0, c_3>c_3^*(c_1,c_2).
      \end{array}
   \right.
\end{equation}

\begin{thm}\label{thm1_3}
Let $\nu_k\to 0^+$, $c_k\in J_0$, $c_k\to c\ne 0$. Assume $U_{\nu_k, \theta}(c_k)\in C^1(-1,1)$ is a solution of (\ref{eq:NSE}) with $\nu_k$ and $c_k$, other than $U_{\nu_k, \theta}^\pm(c_k)$. Then there exists at most one $-1<x_k<1$ such that $U_{\nu_k, \theta}(x_k)=0$, and such $x_k$ must exist if $c_1,c_2>0$. %When $x_k$ exist suppose $x_k\to \hat{x}$ for some $\hat{x}\in [-1,1]$.

(i) If $U_{\nu_k, \theta}(x_k)=0$ for some $x_k\in (-1,1)$, %such $x_k\in (-1,1)$ exists, 
then for any $\epsilon>0$,
\begin{equation}\label{eq1_7_0}
   \lim_{k\to \infty}\left(||U_{\nu_k, \theta}+\sqrt{2P_{c}}||_{L^{\infty}(-1, x_k-\epsilon)}+||U_{\nu_k, \theta}-\sqrt{2P_{c}}||_{L^{\infty}(x_k+\epsilon, 1)}\right)=0,
\end{equation}
and for any $0<\beta<\alpha(c)$, there exists some constant $C>0$, depending only on $c$, $\epsilon$  and $\beta$, such that for large $k$,
\begin{equation*}%\label{eq1_7_1}
    %\limsup_{k\to \infty}\nu_k^{-\alpha(c)}
    ||\frac{1}{2}U^2_{\nu_k, \theta}-P_{c_k}||_{L^{\infty}((-1, x_k-\epsilon)\cup (x_k+\epsilon, 1))}+ \nu_k^{\beta} ||\frac{1}{2}U^2_{\nu_k, \theta}-P_{c_k}||_{C^{\beta}((-1+\epsilon, x_k-\epsilon)\cup (x_k+\epsilon, 1-\epsilon))}\le C\nu_k^{\alpha(c)}.% \limsup_{k\to \infty}\nu_k^{-\frac{2}{3}}||f_k-\sqrt{2P_{c_k}}||_{L^{\infty}(\bar{x}+\epsilon, 1)}<\infty.
\end{equation*}

(ii) If $U_{\nu_k, \theta}(x_k)=0$ for some $x_k\in (-1,1)$ satisfying $x_k\to -1$ and $c_1=0$, or $U_{\nu_k, \theta}\ne 0$ on $(-1,1)$ and  $c_2>0=c_1$, then
%such $x_k$ exists and  $x_k\to -1$ with $c_1=0$, or $x_k$ does not exist with $c_2>0=c_1$, then 
\begin{equation}\label{eq1_7_2}
   \lim_{k\to \infty}||U_{\nu_k, \theta}-\sqrt{2P_{c}}||_{L^{\infty}(-1, 1)}=0.
\end{equation}
   If $U_{\nu_k, \theta}(x_k)=0$ for some $x_k\in (-1,1)$ satisfying $x_k\to 1$ and $c_2=0$, or $U_{\nu_k, \theta}\ne 0$ on $(-1,1)$ and  $c_1>0=c_2$, then 
\begin{equation}\label{eq1_7_3}
   \lim_{k\to \infty}||U_{\nu_k, \theta}+\sqrt{2P_{c}}||_{L^{\infty}(-1, 1)}=0.
\end{equation}
If $U_{\nu_k, \theta}\ne 0$ on $(-1,1)$ and  $c_1=c_2=0$, then, after passing to a subsequence, either (\ref{eq1_7_2}) or (\ref{eq1_7_3}) occurs.

% If such $x_k$ does not exist, or  $x_k\to -1$ with $c_1=0$, or $x_k\to  1$ with $c_2=0$, %$x_k\to -1$ with $c_2>0=c_1$, or $x_k\to  1$ with $c_1>0=c_2$, or $x_k\to \pm 1$ with $c_1=c_2=0$, then
%   \begin{equation*}%\label{eq1_8_5}
%   %\lim_{k\to \infty}||U_{\nu_k, \theta}-\sqrt{2P_{c}}||_{L^{\infty}((-1, 1))}=0
%      \textrm{either } \lim_{k\to \infty}||U_{\nu_k, \theta}-\sqrt{2P_{c}}||_{L^{\infty}(-1, 1)}=0, \textrm{ or }\lim_{k\to \infty}||U_{\nu_k, \theta}+\sqrt{2P_{c}}||_{L^{\infty}(-1, 1)}=0
%      %\lim_{k\to \infty}||\frac{1}{2}U^2_{\nu_k, \theta}-P_{k}||_{L^{\infty}((-1, 1))}=0
%   \end{equation*}
%
   %(iii)If $x_k$ does not exist, then %or $c_1=0,c_2>0$ with $\hat{x}= -1$, or $c_1>0, c_2=0$ with $\hat{x}=1$, or $c_1=c_2=0$ with $\hat{x}=\pm 1$, then
%\begin{equation}\label{eq1_8_2}
%   \limsup_{k\to \infty, 0<\beta<\alpha(c)}\nu_k^{-\alpha(c)}\left(||\frac{1}{2}U^2_{\nu_k, \theta}-P_{c_k}||_{L^{\infty}((-1, 1))}+\nu^{\beta} ||\frac{1}{2}U^2_{\nu_k, \theta}-P_{c_k}||_{C^{\beta}(-1+\epsilon, 1-\epsilon)}\right)<\infty.
%   \end{equation}

(iii) If $c\in \mathring{J}_0$, then $U_{\nu_k, \theta}(x_k)=0$ for some $x_k\in (-1,1)$ and for any $\epsilon>0$ and any positive integer $m$, there exists some constant $C>0$, depending only on $\epsilon$, $m$ and $c$, such that for large $k$, %a lower bound of $c_1,c_2$ and $c_3-c_3^*(c_1,c_2)$, such that
%Suppose $x_k\to \hat{x}$ as $k\to \infty$, then for any $\epsilon>0$, . assume $c_1,c_2>0$, $c_3>c_3^*(c_1,c_2)$,
\begin{equation}\label{eq1_2_1}
   ||U_{\nu_k, \theta}+\sqrt{2P_{c_k}}||_{L^{\infty}(-1, x_k-\epsilon)}+||U_{\nu_k, \theta}-\sqrt{2P_{c_k}}||_{L^{\infty}(x_k+\epsilon, 1)}\le C\nu_k,
\end{equation}
\begin{equation}\label{eq1_2_2}
   ||U_{\nu_k, \theta}+\sqrt{2P_{c_k}}||_{C^m(-1+\epsilon, x_k-\epsilon)}+||U_{\nu_k, \theta}-\sqrt{2P_{c_k}}||_{C^m(x_k+\epsilon, 1-\epsilon)}\le C\nu_k.
\end{equation}
\end{thm}

%\marginpar{provide support for this}

\begin{rmk}
  For any $c\in J_0$ with $c_1,c_2>0$, $0<\nu<1$, and $-1\le \hat{x}\le 1$, there exists some  solution $U_{\theta}$ of (\ref{eq:NSE}), other than $U_{\nu, \theta}^{\pm}(c)$, such that $U_{\theta}(\hat{x})=0$. This can be seen from Theorem A, which asserts that the graphs of all solutions of (\ref{eq:NSE}) foliate the region $\{(x,y)\in \mathbb{R}^2 \mid -1\le x\le 1, U^{-}_{\nu, \theta}(c)\le y \le U^{+}_{\nu, \theta}(c)\}$ in $\mathbb{R}^2$.
\end{rmk}

%\begin{rmk}
%  For any $c\in J_0\setminus\{0\}$, and any $\hat{x}\in [-1,1]$, 
%  there exist $\nu_k\to 0^+$, $c_k\in J_{\nu_k}$, $c_k\to c$, $x_k\in (-1,1)$, $x_k\to \hat{x}$, and $U_{\nu_k, \theta}\in C^1(-1,1)$ being a solutions of (\ref{eq:NSE}) with $\nu_{k}$ and $c_k$ other than $U_{\nu_k, \theta}^{\pm}(c_k)$ such that $U_{\nu_k, \theta}(x_k)=0$.
%  For any $c\in \partial J_0\setminus\{0\}$ with $c_1c_2=0$, 
%  there exist $\nu_k\to 0^+$, $c_k\in J_{\nu_k}$, $c_k\to c$,  and $U_{\nu_k, \theta}\in C^1(-1,1)$ being a solutions of (\ref{eq:NSE}) with $\nu_{k}$ and $c_k$ other than $U_{\nu_k, \theta}^{\pm}(c_k)$ such that $U_{\nu_k, \theta}\ne 0$ on $(-1,1)$. These can be seen easily from the proof of Theorem \ref{thm1_3}.
%  \end{rmk}

%\begin{rmk}
%  In Theorem \ref{thm1_3} (ii), %if $x_k$ does not exists, then $U_{\nu_k, \theta}$ does not change sign. %, one of the two convergence (\ref{eq1_8_5}) holds depending on the sign of $U_{\nu_k, \theta}$. For the other cases,
%  if  $x_k\to -1$ with $c_1=0$, then $U_{\nu_k, \theta}$ converge to the positive solution $\sqrt{2P_c}$ of the Euler equation. If $x_k\to 1$ with $c_2=0$, then $U_{\nu_k, \theta}$ converge to the negative solution $-\sqrt{2P_c}$ of the Euler equation.
%\end{rmk}

The above theorem indicates the formation of boundary layers (if we view $x=\pm 1$ as boundaries) and interior layers. We give descriptions of boundary layers and interior layers in the following theorem.

For $c\in J_0$, define
\begin{equation*}
   \kappa(c):=\left\{
      \begin{array}{ll}
        1,& \textrm{ if }c_1,c_2>0,c_3>c^*_3(c_1,c_2),\\
        0, & \textrm{ otherwise. }
      \end{array}
   \right.
\end{equation*}
Let $-1<x<1$, $K>0$, define
\begin{equation*}
	\label{eq:BL:mid}
		\widetilde{U}_{\theta,x_k}(x) =
		\begin{cases}
			-\sqrt{2P_{c_k}(x)}, \qquad & -1\le x < x_k - K\nu_k|\ln \nu_k|(1-x_k^2), \\
			\sqrt{2P_{c_k}(x_k)} \tanh \Big( \frac{\sqrt{2P_{c_k}(x_k)} \cdot (x-x_k)}{2 (1-x_k^2)\nu_k} \Big), \quad & |x-x_k|\le K\nu_k|\ln \nu_k|(1-x_k^2), \\%x_k - K\nu_k|\ln \nu_k|(1-x_k^2) < x < x_k + K\nu_k|\ln \nu_k|(1-x_k^2),  \\
			\sqrt{2P_{c_k}(x)}, \qquad & x_k + K\nu_k|\ln \nu_k|(1-x_k^2) < x \le 1. \\
		\end{cases}
	\end{equation*}

%Let $a,b\in [-1,1]$, $K\in \mathbb{R}$, define
%\begin{equation}
%			f_{k,a}(x):=\sqrt{2P_{c_k}(a)} \tanh \Big( \frac{\sqrt{2P_{c_k}(a)} \cdot (x-a)}{2 (1-a^2)\nu_k} \Big)
%	\end{equation}
%\begin{equation}
%	\label{eq:BL:mid}
%		\widetilde{U}^1_{\theta,k,a,b}(x) =
%		\begin{cases}
%			-\sqrt{2P_{c_k}(x)}, \qquad & -1\le x < b - K\nu_k|\ln \nu_k|, \\
%			f_{k,a}(x), \quad & b - K\nu_k|\ln \nu_k| < x < b + K\nu_k|\ln \nu_k|,  \\
%			\sqrt{2P_{c_k}(x)}, \qquad & b + K\nu_k|\ln \nu_k| < x \le 1, \\
%		\end{cases}
%	\end{equation}
%
%\begin{equation}
%	\label{eq:BL:left}
%		\widetilde{U}^2_{\theta,k,a,b}(x) =
%		\begin{cases}
%			- \sqrt{2P_{c_k}}, \qquad & -1\le x<b - K\nu_k|\ln \nu_k|(1+b), \\
%			f_{k,a}(x), \quad & b- K\nu_k|\ln \nu_k|(1+b) < x < b + K\nu_k|\ln \nu_k|(1+b),  \\
%			\sqrt{2P_{c_k}}, \qquad & b + K\nu_k|\ln \nu_k|(1+b) < x \le 1, \\
%		\end{cases}
%	\end{equation}
%
%\begin{equation}
%	\label{eq:BL:right}
%		\widetilde{U}^3_{\theta,k,a,b}(x) =
%		\begin{cases}
%			- \sqrt{2P_{c_k}}, \qquad & -1\le x<b - K\nu_k|\ln \nu_k|(1-b), \\
%			f_{k,a}(x), \quad & b- K\nu_k|\ln \nu_k|(1-b) < x < b + K\nu_k|\ln \nu_k|(1-b),  \\
%			\sqrt{2P_{c_k}}, \qquad & b + K\nu_k|\ln \nu_k|(1-b) < x \le 1, \\
%		\end{cases}
%	\end{equation}

\begin{thm}\label{thm:BL:1}
	Let $\nu_k\to 0^+$, $c_k\in J_0$, $c_k\to c\ne 0$. Assume $U_{\nu_k, \theta}\in C^1(-1,1)$ is a solution of (\ref{eq:NSE}) with $\nu_k$ and $c_k$, other than $U_{\nu_k, \theta}^\pm$. In addition, assume that there exists $x_k\in(-1,1)$ such that $U_{\nu_k, \theta}(x_k) = 0$ and $x_k\to \hat{x}\in[-1,1]$, $P_c(\hat{x})\not=0$. Then $\{U_{\nu_k, \theta}\}$ develops a layer near $x_k$.
	Moreover, there exist some positive constants $K$ and $C$, depending only on $c$, such that for large $k$, \begin{equation}\label{eq:BL:rate}
		\|U_{\nu_k, \theta} - \widetilde{U}_{\theta,x_k}\|_{L^\infty(-1,1)}
%		\le C \nu_k^\alpha+C\nu_k|\ln \nu_k|^2
		\le C\nu_k^{\alpha(c)}|\ln \nu_k|^{2\kappa(c)}.
		%\begin{cases}
%			C\nu_k|\ln \nu_k|^2, & \text{if (A1) holds}, \\
%			C \nu_k^{2/3}, & \text{if (A2) holds}, \\
%			C \nu_k^{1/2}, & \text{if (A3') holds}, \\
%		\end{cases}
	\end{equation}
\end{thm}

\begin{rmk}
	The solutions $U_{\nu_k, \theta}$ of (\ref{eq:NSE}) with $\nu_k$ asymptotically behave like $\widetilde{U}_{\theta,x_k}$ as $\nu_k\to 0^+$. Hence an interior layer appears when $\hat{x}\in(-1,1)$, and a boundary layer appears when $\hat{x}=\pm 1$ if we view $x=\pm 1$ as boundaries.
\end{rmk}
\begin{rmk}\label{rmk:scale}
	The length scale of the transition layers is $\nu_k$ for interior layers, and is $o(\nu_k)$ for boundary layers. Moreover, for any $\epsilon_k=o(\nu_k)$, there exists $\{U_{\nu_k, \theta}\}$ having boundary layer length scale as $\epsilon_k$.
\end{rmk}

The organization of the paper is as follows. Theorem \ref{thm1_0_0} is proved at the beginning of Section \ref{sec2}. In the remaining part of Section \ref{sec2} we present some preliminary results and prove Theorem \ref{thm1_0_1} at the end of Section \ref{sec2}.  In Section \ref{sec3} we prove Theorem \ref{thm1_1} and Theorem \ref{thm1_3} (iii). In Section \ref{sec4} we prove Theorem \ref{thm1_2} and Theorem \ref{thm1_2_1}. In Section \ref{sec5} we prove Theorem \ref{thm1_2_2}. Theorem \ref{thm1_3} (i) and (ii) are proved at the end of Section \ref{sec4} and the end of Section \ref{sec5}. Theorem \ref{thm1_0} is proved at the end of  Section \ref{sec5}. In Section \ref{sec6} we prove Theorem \ref{thm:BL:1}. In Section \ref{sec:illu} we give illustrations on transition layer behaviors. In the appendix, we present some elementary properties of second order polynomials which we have used.

\noindent
{\bf Acknowledgment}. 
The work of the first named author is partially supported by NSFC grants No. 11871177. The work of the second named author is partially supported by NSF grants DMS-1501004.   The work of the third named author is partially supported by AMS-Simons Travel Grant and AWM-NSF Travel Grant.

\section{Preliminary}\label{sec2}

%***************

We first prove Theorem \ref{thm1_0_0}.

\begin{lem}\label{lem2_0}
  For $0<\nu\le 1$, let $U_{\nu, \theta}$ satisfies (\ref{eq:NSE}) in $(-1,1)$ for some $c\in J_{\nu}$. Then there exists some universal constant $C>0$, such that for any $-1<r<s<1$, 
  \begin{equation}\label{eq2_0_0}
     ((s-r)^4|c|-(s-r)\nu^2)/C\le \int_{r}^{s}U^2_{\nu, \theta}\le C(|c|+\nu^2/\min\{1-s,1+r\}).
  \end{equation}
\end{lem}
\begin{proof}
   Throughout the proof, $C$ denotes some universal constant which may change value from line to line.
   
   For all $0<r<s<1$, there exist $a\in [-s,-r]$ and $b\in[r,s]$ such that 
   \[
      |U_{\nu, \theta}(a)|\le \frac{1}{\sqrt{s-r}}\left(\int_{-s}^{-r}U_{\nu, \theta}^2\right)^{1/2}, \quad  |U_{\nu, \theta}(b)|\le \frac{1}{\sqrt{s-r}}\left(\int_{r}^{s}U_{\nu, \theta}^2\right)^{1/2}.
   \]
   By (\ref{eq:NSE}) and the above,
   \[
      \begin{split}
       \int_{-r}^{r}U_{\nu, \theta}^2 & \le \int_{a}^{b}U_{\nu, \theta}^2 =2\int_{a}^{b}\left(P_c-\nu(1-x^2)U'_{\nu, \theta}-2\nu xU_{\nu, \theta}\right)dx\\
                                      & \le C|c|+C\nu(|U_{\nu, \theta}(a)|+|U_{\nu, \theta}(b)|)+C\nu^2+\frac{1}{4}\int_{a}^{b} U^2_{\nu, \theta}(x)dx\\
                                      & \le C(|c|+\nu^2/(s-r))+\frac{1}{2}\int_{-s}^{s}U_{\nu, \theta}^2 .%C(|c|+\nu^2)+\frac{1}{4}\int_{a}^{b} U^2_{\nu, \theta}(x)dx+\frac{C}{\sqrt{\epsilon}}\nu\left(\left(\int_{r}^{s}U^2_{\nu, \theta}dx\right)^{1/2}+|c|^{1/2}\right)
      % \le \frac{C}{\epsilon}\left(\left(\int_{r}^{s}U^2_{\nu, \theta}dx\right)^{1/2}+|c|^{1/2}\right)
       \end{split}
   \]
   By Lemma 1 in \cite{Giaquinta}, 
   \begin{equation*}
      \int_{-r}^{r}U_{\nu, \theta}^2\le C(|c|+\nu^2/(s-r)), \quad \forall 0<r<s<1.
   \end{equation*}
   The second inequality in (\ref{eq2_0_0}) follows from the above.
   
   Next, we prove the first inequality in (\ref{eq2_0_0}).  Rewrite $P_c=\hat{c}_1+\hat{c}_2x+\hat{c}_3x^2$. Then $|c|\le C|\hat{c}|$ where $\hat{c}=(\hat{c}_1,\hat{c}_2,\hat{c}_3)$.    For $-1<r<s< 1$, let $\delta =(s-r)/9$. Then there exist $a\in [r, r+\delta]$ and $b_i\in [r+2i\delta, r+(2i+1)\delta]$, $i=1,2,3$, such that
   \begin{equation}\label{eq2_0_1}
      |U_{\nu, \theta}(a)|\le \frac{1}{\sqrt{\delta}}\left(\int_{r}^{r+\delta}U_{\nu, \theta}^2\right)^{1/2}, \quad  |U_{\nu, \theta}(b_i)|\le \frac{1}{\sqrt{\delta}}\left(\int_{r+2i\delta}^{r+(2i+1)\delta}U_{\nu, \theta}^2\right)^{1/2}.
   \end{equation}
  For each $i=1,2,3$, we have 
   \[
        \int_{a}^{b_i}P_c(x)dx  =\int_{a}^{b_i}\left(\nu(1-x^2)U'_{\theta}+2\nu xU_{\theta}+\frac{1}{2}U^2_{\theta}\right)=:\beta_i.
   \]
   Let $\beta=(\beta_1,\beta_2,\beta_3)$, write the above  as $A\hat{c}^t=\beta^t$, where $\hat{c}^t$ and $\beta^t$ denote the transpose of $\hat{c}$ and $\beta$ respectively, and
   \[
      A= \left(
	\begin{matrix}
		b_1-a & (b_1^2-a^2)/2 & (b_1^3-a^3)/3\\
		b_2-a & (b_2^2-a^2)/2 & (b_2^3-a^3)/3\\
		b_3-a & (b_3^2-a^2)/2 & (b_3^3-a^3)/3
	\end{matrix} \right).
   \]
  By (\ref{eq2_0_1}),  we have, after an integration by parts,
   \begin{equation}\label{eq2_0_3}
    % \begin{split}
      |\beta_i|   \le C\nu \left( |U_{\nu, \theta}(a)|+ |U_{\nu, \theta}(b_i)|\right)+C\nu^2+C\int_{a}^{b_i}U_{\nu, \theta}^2dx
                    \le C\nu^2+\frac{C}{\delta}\int_{r}^{s}U_{\nu, \theta}^2.
     % \end{split} 
   \end{equation}
   %Let $f_i=b_i-a$, $g_i=b_i^2+ab_i+a^2$, $\lambda_1=\frac{g_1f_2-g_2f_1}{f_2-f_1}$, $\lambda_2=\frac{g_2-g_1}{f_2-f_1}$. 
   By computation, we have that $A$ is invertible and 
   \[
      A^{-1}=\left(
	\begin{matrix}
		-\frac{2a^2+ab_3+ab_2-b_2b_3}{(b_1-a)(b_2-b_1)(b_3-b_1)} & \frac{2a^2+ab_1+ab_3-b_1b_3}{(b_2-a)(b_2-b_1)(b_3-b_2)} & -\frac{2a^2+ab_1+ab_2-b_1b_2}{(b_3-a)(b_3-b_1)(b_3-b_2)}\\
		-\frac{2(b_2+b_3+a)}{(b_1-a)(b_2-b_1)(b_3-b_1)} & \frac{2(b_1+b_3+a)}{(b_2-a)(b_2-b_1)(b_3-b_2)} & -\frac{2(b_1+b_2+a)}{(b_3-a)(b_3-b_1)(b_3-b_2)}\\
		\frac{3}{(b_1-a)(b_2-b_1)(b_3-b_1)} & -\frac{3}{(b_2-a)(b_2-b_1)(b_3-b_2)} & \frac{3}{(b_3-a)(b_3-b_1)(b_3-b_2)}
	\end{matrix} \right).
   \]
   Clearly, $\delta\le b_i-a, b_j-b_i\le 9\delta$ for every $i<j$. So we have $||A^{-1}||\le C\delta^{-3}$. %\frac{C}{\delta^3}$. 
  % By the choice of $a$ and $b_i$ we make, we have $\delta\le b_i-a\le 9\delta$, for $i\ne j$. %$|f_i-f_j|=|b_i-b_j|\in [\delta, 9\delta]$ for $i\ne j$, and $g_i-g_j=(f_i-f_j)(b_i+b_j+a)$. Moreover,  
   %\[
   %  \begin{split}
 %    \lambda_1 %& =\frac{g_1f_2-g_2f_1}{f_2-f_1}=\frac{g_1(f_2-f_1)-(g_2-g_1)f_1}{f_2-f_1}\\
  %                      =g_1-f_1(b_1+b_2+a), 
    %  \end{split}
  % \]
  % and $\lambda_2=b_1+b_2+a$.  So $|\lambda_1|, |\lambda_2|\le C$. With these estimates, 
  %So we have that $|(A^{-1})_{ij}|\le \frac{C}{\delta^3}$, for all $1\le i, j\le 3$.
   Then, using %(\ref{eq2_0_2}) 
   (\ref{eq2_0_3}), we have 
   \[
      |c|\le C|\hat{c}|=C|A^{-1}\beta|\le ||A^{-1}|||\beta|\le C\delta^{-3}\left(\nu^2+\frac{1}{\delta}\int_{r}^{s}U_{\nu, \theta}^2\right).
   \]
   %Notice $\delta=\frac{s-r}{100}$, t
   The first inequality of (\ref{eq2_0_0}) follows from the above. The lemma is proved.
   \end{proof}
%*****************************

\noindent{\emph{Proof of Theorem \ref{thm1_0_0}}}:
%\begin{thm}%\label{thm1_0_0}
%  (i) Let $0<\nu<1$, $0<\theta_1<\theta<2<\theta_3<\theta_4<\pi$, $(u_{\nu}, p_{\nu})$ be $(-1)$-homogeneous axisymmetric no-swirl solutions of (\ref{NS}) on $\mathbb{S}^2\cap\{\theta_1<\theta<\theta_4\}$. Then there exists some positive constant $C$, depending only on $\theta_i$, $1\le i\le 4$, such that
%  \[
%      \int_{\mathbb{S}^2\cap\{\theta_1<\theta<\theta_4\}}|u_{\theta}(x)|^2d\sigma(x)\le C\left(\int_{\mathbb{S}^2\cap\{\theta_2<\theta<\theta_3\}}|u_{\theta}(x)|^2d\sigma(x)+\nu^2\right).
%  \]
%  (ii)Let $\nu_k\to 0^+$, $(u_{\nu_k}, p_{\nu_k})$ be $(-1)$-homogeneous axisymmetric no-swirl solutions of (\ref{NS})on $\mathbb{S}^2\setminus\{S,N\}$. If $\nu_k^{-2}\int_{\mathbb{S}^2\cap\{\frac{\pi}{4}<\theta<\frac{\pi}{2}\}}|u_{\nu_k, \theta}(x)|^2d\sigma(x)$ is bounded, then there exists some solution $(\tilde{u}, \tilde{p})$ of (\ref{NS}), such that for any $\epsilon>0$,
%  \begin{equation}\label{eq1_0_0_7}
%     ||u_{\nu_k, \theta}-\nu_k\tilde{u}_{\theta}||_{L^{2}(\mathbb{S}^2\cap\{\epsilon<\theta<\pi-\epsilon\})}=o(\nu_k)
%  \end{equation}
%\end{thm}
%\begin{proof}
   (i) We use  $C$ to denote a positive constant depending only on  $\{\theta_i\}$, which may vary from line to line. Let $r_i=\cos\theta_i$, $1\le i\le 4$, $x=\cos\theta$, $U_{\nu, \theta}=u_{\nu, \theta}\sin\theta$. Then $U_{\nu, \theta}$ satisfies (\ref{eq:NSE}) on $(r_4,r_1)$ for some $c\in J_{\nu}$. By Lemma \ref{lem2_0}, %and the fact that %We will first prove that for any $r_1\le r<s\le r_4$
%$U_{\nu, \theta}=u_{\nu, \theta}\sin\theta$, 
 we have
   \begin{equation}\label{eq1_0_0_5}
     \begin{split}
       & \int_{\mathbb{S}^2\cap \{\theta_1<\theta<\theta_4\}}|u_{\nu, \theta}|^2\le C\int_{r_4}^{r_1}U^2_{\nu, \theta}(x)dx \le C(|c|+\nu^2)\\
       & \le C\left(\int_{r_3}^{r_2}U^2_{\nu, \theta}(x)dx+\nu^2\right)\le C\left(\int_{\mathbb{S}^2\cap\{\theta_2<\theta<\theta_3\}}|u_{\nu, \theta}|^2+\nu^2\right).
      \end{split}
   \end{equation}
   Part (i) is proved.

   (ii) Let $U_{\nu_k, \theta}=u_{\nu_k, \theta}\sin\theta$, $r=\cos(\pi-\epsilon), s=\cos\epsilon$. Since $(u_{\nu_k}, p_{\nu_k})$ are $(-1)$-homogeneous axisymmetric no-swirl solutions of (\ref{NS}) on $\mathbb{S}^2\setminus\{S,N\}$, there exists $c_k\in J_{\nu_k}$, such that $U_{\nu_k, \theta}$ satisfies (\ref{eq:NSE}) with the right hand side to be $P_{c_k}$. By Lemma \ref{lem2_0},  using the boundedness of $\nu_k^{-2}\int_{\mathbb{S}^2\cap\{a<\theta<b\}}|u_{\nu_k, \theta}|^2$, $\{\nu_k^{-2}|c_k|\}$ is bounded for some $a,b\in (-1,1)$. %So there exists some solution $\tilde{U}_{\theta}$ of the Navier-Stokes equation (\ref{eq:NSE}), such that
 %  \begin{equation}\label{eq1_0_0_6}
  %    ||U_{\nu_k, \theta}(c_k)-\nu_k\tilde{U}_{\theta}||_{C^{\infty}_{loc}(-1,1)}=o(\nu_k)
  %     \end{equation}
      % as $k\to \infty$. Indeed, 
      Notice that  $\tilde{U}_{\theta, k}:=U_{\nu_k, \theta}(c_k)/\nu_k$ is a solution to (\ref{eqNSE_1}) with $P_{c_k\nu_k^{-2}}$, and after passing to a subsequence, $\tilde{c}_k:=c_k\nu_k^{-2}\to \tilde{c}$ for some $\tilde{c}$. By Lemma 2.2 in \cite{LLY2}, $\{||\tilde{U}_{\theta, k}||_{L^{\infty}(-1,1)}\}$ is bounded. It follows from standard ODE theories that there exists some smooth solution $\tilde{U}_{\theta}$ of (\ref{eqNSE_1}) with $c\nu^{-2}=\tilde{c}$ that $\tilde{U}_{\theta, k}\to \tilde{U}_{\theta}$ in $C^m([-1+\epsilon, 1-\epsilon])$ for any $\epsilon>0$ and any positive integer $m$. Part (ii) is proved with $\tilde{u}_{\theta}=\tilde{U}_{\theta}/\sin\theta$ together with (\ref{eqNS_1}).
      \qed

Let $\nu>0$, $c\in \mathbb{R}^3$,  %, $c=(c_1,c_2,c_3)\in \mathbb{R}^3$, $P_c$ be defined as in (\ref{eq:NSE}), 
  and $f_{\nu}$ be a solution of the equation
\begin{equation}\label{eq_1}
 \nu (1-x^2)f'_{\nu}+2\nu x f_{\nu}+\frac{1}{2}f^2_{\nu}=P_c(x). %\quad -1\le x\le 1.
\end{equation}

\begin{lem}\label{lem2_1}
	For $0<\nu\le 1$ and $c\in \mathbb{R}^3$, let $f_{\nu}$ be a solution of (\ref{eq_1}) in $C^1(-1,1)$. Then %there exists some constant $M>0$, depending only on $c$, such that %an upper bound of $|c|$, such that
 $|f_{\nu}|\le 5\sqrt{1+|c|}$ in $(-1,1)$.
\end{lem}
\begin{proof}
 By Theorem 1.2 and Theorem 1.3 in \cite{LLY2}, we have $c\in J_{\nu}$, $f_{\nu}(-1)=\tau_1$ or $\tau_2$, and $f_{\nu}(1)=\tau_1'$ or $\tau_2'$, where $\tau_1,\tau_2,\tau_1'$ and $\tau_2'$ are defined as in (\ref{eq_2}). Thus $|f(\pm 1)|<5\sqrt{1+|c|}$.

	%Fix a large $M$ satisfying
%	\[
%		M>3+2\max\{\sqrt{1+c_1}, \sqrt{1+c_2}\}%>\max \{|\tau_1|, |\tau_2|, |\tau_1'|, |\tau_2' |\}, \quad \quad \text{for any }0<\nu\leq 1,
%	\] and
%	\[
%		\frac{1}{2}M^2-2M>1+6|c|.%\sup_{-1<x<1} [c_1(1-x)+c_2(1+x)+c_3(1-x^2)].
%	\]
%	
	Suppose that there exists a point $x_0\in(-1,1)$ such that $f_\nu(x_0)>5\sqrt{1+|c|}$, then by (\ref{eq_1}),
	\[
		\nu (1-x_0^2) f_{\nu}'(x_0) \le 6|c|-\frac{1}{2}f_{\nu}^2(x_0)+2\nu f_{\nu}(x_0)\le 6|c|+4\nu^2-\frac{1}{4}f^2_{\nu}(x_0)<0.%P_c(x_0)-\frac{1}{2}M^2+2M\le 6|c|-\frac{1}{2}M^2+2M<-1.%< c_1(1-x_0)+c_2(1+x_0)+c_3(1-x_0^2) - (\frac{1}{2}M^2-2M)<-1.
	\]
	So $f_{\nu}'(x_0)<0$.  It follows that $f_{\nu}(x)>5\sqrt{1+|c|}$ for any $-1<x<x_0$. This contradicts the fact that $f_{\nu}(-1)<5\sqrt{1+|c|}$.
	 We have proved that $f_{\nu}\le 5\sqrt{1+|c|}$ on $(-1,1)$. Similarly, we can prove that $f_{\nu}\ge -5\sqrt{1+|c|}$ on $(-1,1)$.
	%On the other hand, if there exists a point $x_1\in(-1,1)$ such that $f_{\nu}(x_1)<-M$, then similar arguments implies that $f_{\nu}(x)<-M$ for any $x_1<x<1$. This contradicts the fact that $f_{\nu}(1)=\tau_1'$ or $\tau_2'$. So $|f_{\nu}|<M$ in $(-1,1)$.
\end{proof}

%%%%%%Previous Lemma 2_5

%\begin{lem}\label{lem2_5} %%%%%%
%   Let $0< \nu \le 1$, $-1\le a<b\le 1$, $f_{\nu}\in C^1(a,b)\cap C^0[a,b]$ be a solution of (\ref{eq_1}) in $(a,b)$ and $|f_{\nu}| \le M $ for some constant $M>0$ on $(a,b)$. Then there exists some $x_{\nu}\in [a,b]$,   such that
%   \[
%       \left|\frac{1}{2}f^2_{\nu}(x_{\nu})-P_{c}(x_{\nu})\right|<\frac{6M\nu}{b-a}.
%   \]
%\end{lem}
%\begin{proof}
%    For convenience write $h_{\nu}=\frac{1}{2}f^2_{\nu}-P_{c}$, then either $h_{\nu}$ has an zero point on $(a,b)$, or $h_{\nu}$ does not change sign on $(a,b)$. In the second case, take integral of (\ref{eq_1}) on $(a,b)$, we have
%    \[
%      \begin{split}
%       \int_{a}^{b}|h_{\nu}(x)|dx & =\left|\int_{a}^{b}(\nu(1-x^2)f'_{\nu}(x)+2\nu xf_{\nu}(x))dx\right|\\
%          & =\nu\left|(1-b^2)f_{\nu}(b)-(1-a^2)f_{\nu}(a)+\int_{a}^{b}4x f_{\nu}\right|\\
%          & \le 6M\nu,
%       \end{split}
%    \]
%   % for some $C$ depending only on $C_1$.
%    So there exists some $x_{\nu}\in (a,b)$, such that
%    \[
%       \left|\frac{1}{2}f^2_{\nu}(x_{\nu})-P_{c}(x_{\nu})\right|<\frac{6M\nu}{b-a}.
%   \]
%\end{proof}

\begin{lem}\label{lem2_5}
	Let $0< \nu \le 1$, $-1\le a<b\le 1$, $f_{\nu}\in C^1(a,b)$ be a solution of (\ref{eq_1}) in $(a,b)$ and $|f_{\nu}| \le M$ on $(a,b)$ for some constant $M>0$ . Then %there exists some $x_{\nu}\in [a,b]$, such that
	\begin{equation}\label{eq2_5_0}
		\inf_{(a,b)}\left|\frac{1}{2}f^2_{\nu}-P_{c}\right|<10M\nu/(b-a).
	\end{equation}
	Moreover, if $0 \le a <b\le 1$, we have%then there exists some $x_{\nu}\in [a,b]$, such that
	\begin{equation}\label{eq:lem2_5:1}
		\inf_{(a,b)}\left|\frac{1}{2}f^2_{\nu}-P_{c}\right|\le 8M\nu(1-a)/(b-a),
	\end{equation}
	and if $-1 \le a <b\le 0$, we have %then there exists some $x_{\nu}\in [a,b]$, such that
	\begin{equation}\label{eq:lem2_5:2}
		\inf_{(a,b)}\left|\frac{1}{2}f^2_{\nu}-P_{c}\right|\le 8M\nu(b+1)/(b-a).
	\end{equation}
\end{lem}
\begin{proof}
Shrinking $(a,b)$ slightly, we may assume without loss of generality that $f_{\nu}$ is also in $C^0[a,b]$.
	For convenience we write $h_{\nu}=\frac{1}{2}f^2_{\nu}-P_{c}$, and we only need to consider that $h_{\nu}$ does not change sign on $(a,b)$. Integrating (\ref{eq_1}) over $(a,b)$, we have
	\[
	\begin{split}
		\int_{a}^{b}|h_{\nu}(x)|dx & =\left|\int_{a}^{b}(\nu(1-x^2)f'_{\nu}(x)+2\nu xf_{\nu}(x))dx\right|\\
		& =\nu\left|(1-b^2)f_{\nu}(b)-(1-a^2)f_{\nu}(a)+\int_{a}^{b}4x f_{\nu}\right|\le 10M\nu.
		%& \le 6M\nu,
	\end{split}
	\]
This implies (\ref{eq2_5_0}).
	%So there exists some $x_{\nu}\in (a,b)$, such that
%	\[
%		\left|\frac{1}{2}f^2_{\nu}(x_{\nu})-P_{c}(x_{\nu})\right|<\frac{6M\nu}{b-a}.
%	\]

	If $0 \le a <b\le 1$, we have
	\[
	\begin{split}
		\int_{a}^{b}|h_{\nu}(x)|dx & =\nu\left|(1-b^2)f_{\nu}(b)-(1-a^2)f_{\nu}(a)+\int_{a}^{b}4x f_{\nu}\right|\\
		& \le M \nu \big( 2(1-b)+2(1-a)+4(1-a) \big) 
		 \le 8M\nu(1-a).
	\end{split}
	\]
	This gives (\ref{eq:lem2_5:1}). Estimate (\ref{eq:lem2_5:2}) can be proved similarly.
\end{proof}

\begin{lem}\label{lem2_2}
   For $0< \nu \le 1$, $c\in \mathbb{R}^3$, $-1\le a< b \le 1$, let $f_{\nu}\in C^1(a,b)\cap C^0[a,b]$ be a solution of (\ref{eq_1}) in $(a,b)$ satisfying, for some positive constants $\mu$ and $\delta$,  that $f_{\nu}(a)\ge \mu$ and $P_c\ge \delta$ in $(a,b)$. Then  for all $0<\nu\le 1$,
   \[
       f_{\nu}(x)\ge \min\{\mu, \sqrt{\delta}, \delta/(4\nu)\}, \quad a\le x\le b.
   \]
\end{lem}
\begin{proof}
    If for some $0<\lambda<\mu$, there exists some $x\in (a,b]$ such that $f_{\nu}(x)\le \lambda$, then let $x_{\nu}$ be the first point greater than $a$ such that $f_{\nu}(x_{\nu})=\lambda$. Then $f'_{\nu}(x_{\nu})\le 0$. By equation (\ref{eq_1}) we have that
   \[
       2\nu\lambda+\lambda^2/2\ge 2\nu x_{\nu}\lambda+\lambda^2/2\ge P_c(x_\nu)\ge \delta.
   \]
   So either $4\nu\lambda\ge \delta$ or $\lambda^2\ge \delta$. %The proof is finished.
\end{proof}

\addtocounter{lem}{-1}
\renewcommand{\thelem}{\thesection.\arabic{lem}'}%
\begin{lem}\label{lem2_2'}
	For $0< \nu \le 1$, $c\in\mathbb{R}^3$, $-1\le a< b \le 1$, let $f_{\nu}\in C^1(a,b)\cap C^0[a,b]$ be a solution of (\ref{eq_1}) in $(a,b)$, satisfying, for some positive constants $\mu$ and $\delta$, that $f_{\nu}(b)\le -\mu$ and $P_c\ge \delta$ in $(a,b)$. Then  for all $0<\nu\le 1$,
   \[
       f_{\nu}(x)\le -\min\{\mu, \sqrt{\delta}, \delta/(4\nu)\}, \quad a\le x\le b.
   \]
\end{lem}
\renewcommand{\thelem}{\thesection.\arabic{lem}}%
\begin{proof}
   Let $g_{\nu}(x):=-f_{\nu}(a+b-x)$ for $x\in [a,b]$. Then $g_{\nu}$ is a solution of (\ref{eq_1}) with the same $P_c$ and $g_{\nu}(a)\ge \mu$. The lemma follows from Lemma \ref{lem2_2}, applied to $g_{\nu}$. %Using Lemma \ref{lem2_2}, the lemma is proved.
\end{proof}

\begin{cor}\label{cor2_3}
	Let $0< \nu \le 1$, $c\in \mathbb{R}^3$, $-1\le a<b\le 1$, $f_{\nu}\in C^1(a,b)$ be a solution of (\ref{eq_1}) in $(a,b)$ satisfying $0\le f_{\nu} \le M$ on the interval for some positive constant $M$. If $P_c\ge \delta>0$ in $(a,b)$ for some constant $\delta$, then
	\begin{equation}\label{eq:lem2_3}
		f_{\nu}(x)\ge \min\{\sqrt{\delta},\delta/(4\nu)\}, \qquad x\in(a+\epsilon, b)
	\end{equation}
	holds for any $\epsilon$ satisfying $20M\nu/\delta<\epsilon<b-a$. %, $0<\nu \le 1$.
	If we further assume that $-1<a<-1/2$, then (\ref{eq:lem2_3}) holds for any $32M\nu(a+1)/\delta<\epsilon<\min\{a+1,b-a\}$. %, $0<\nu \le 1$.
\end{cor}
\begin{proof}
Shrinking $(a,b)$ slightly we may assume without loss of generality that $f_{\nu}$ is also in $C^0[a,b]$.
	By Lemma \ref{lem2_5}, there is some $x_{\nu}\in [a,a+\epsilon]$ such that
	\[
		\left|\frac{1}{2}f^2_\nu(x_{\nu})-P_c(x_{\nu})\right|\le 10M\nu/\epsilon.
	\]
	For $20M\nu/\delta<\epsilon<b-a$, we have 
		$f^2_\nu(x_{\nu})\ge 2P_c(x_{\nu})-20M\nu/\epsilon \ge \delta$.
	So $f_{\nu}(x_{\nu})\ge \sqrt{\delta}$. Then applying Lemma \ref{lem2_2}, we have (\ref{eq:lem2_3}) for any $20M\nu/\delta<\epsilon<b-a$, $0<\nu \le 1$.

	If $-1<a<-1/2$, then $2a+1<0$. By Lemma \ref{lem2_5} and (\ref{eq:lem2_5:2}), for any $0<\epsilon<\min(a+1,b-a)$, there exists $x_\nu\in (a,a+\epsilon)$ such that
	\[
		\left| \frac{1}{2} f^2_\nu (x_{\nu}) - P_c(x_{\nu}) \right| \le 8M\nu (a+\epsilon+1)/\epsilon < 16M\nu(a+1)/\epsilon.
	\]
	For $32M\nu(a+1)/\delta<\epsilon<\min(a+1,b-a)$, we have
		$f^2_\nu(x_{\nu})\ge 2P_c(x_{\nu})-32M\nu(a+1)/\epsilon\ge \delta$. 
	So $f_{\nu}(x_{\nu})\ge \sqrt{\delta}$. By Lemma \ref{lem2_2}, (\ref{eq:lem2_3}) holds for any $32M\nu(a+1)/\delta<\epsilon<\min(a+1,b-a)$, $0<\nu \le 1$. %The lemma is proved.
\end{proof}

\addtocounter{cor}{-1}
\renewcommand{\thecor}{\thesection.\arabic{cor}'}%
\begin{cor}\label{cor2_3'}
	Let $0< \nu \le 1$, $c\in \mathbb{R}^3$, $-1\le a<b\le 1$, $f_{\nu}\in C^1(a,b)$ be a solution of (\ref{eq_1}) in $(a,b)$ and $-M\le f_{\nu} \le 0$ for some positive constant $M$ on $(a,b)$. If $P_c\ge \delta>0$ in $(a,b)$, then
	\begin{equation}\label{eq:lem2_3'}
		f_{\nu}\le -\min\{\sqrt{\delta},\delta/(4\nu)\}, \qquad x\in(a,b-\epsilon)
	\end{equation}
	holds for any $20M\nu/\delta<\epsilon<b-a$.
	If we further assume that $1/2<b<1$, then (\ref{eq:lem2_3'}) holds for any $32M\nu(1-b)/\delta<\epsilon<\min(1-b,b-a)$.
\end{cor}
\renewcommand{\thecor}{\thesection.\arabic{cor}}%

\begin{rmk}
	Under the conditions of Corollary \ref{cor2_3} (or Corollary \ref{cor2_3'}), for any small $\epsilon>0$ fixed, there exists $\nu_0>0$, depending only on $\epsilon, M$ and $\delta$, such that (\ref{eq:lem2_3}) (or (\ref{eq:lem2_3'})) holds for all $0<\nu<\nu_0$.
\end{rmk}

%%%%%%%%%%%%%%%%%%%%%%%% LEMMA 2.4 %%%%%%%%%%%%%%%%%%%%%%%%%%%%%%%%%

\begin{lem}\label{lem2_4}
    Let $0< \nu \le 1$, $-1\le a< b \le 1$, $c\in \mathbb{R}^3\setminus\{0\}$, suppose $f_{\nu}\in C^1(a,b)\cap C^0[a,b]$ is a solution of (\ref{eq_1}) in $(a,b)$, and $P_c\ge 0$ in $(a,b)$. Then there exists at most one $x_{\nu}\in (a,b)$ such that $f_{\nu}(x_{\nu})=0$. Moreover, if $f_{\nu}(a)>0$, then $f_{\nu}>0$ on $(a,b)$, and if $f_{\nu}(b)<0$, then $f_{\nu}<0$ on $(a,b)$.
\end{lem}
\begin{proof}
We first prove that there does not exist $\bar{x}\in (a,b)$ and $\epsilon>0$ such that $f_{\nu}(\bar{x})=0$, $f'_{\nu}(\bar{x})\le 0$ and $f_{\nu}>0$ in $(\bar{x}-\epsilon, \bar{x})$. If such $\bar{x}$ and $\epsilon$ exist ,then
\[
      0\ge \nu(1-\bar{x}^2)f'_{\nu}(\bar{x})+2\nu\bar{x}f_{\nu}(\bar{x})+\frac{1}{2}f^2_{\nu}(\bar{x})=P_c(\bar{x})\ge 0.
   \]
   So $P_c(\bar{x})=0$,  and $f'_{\nu}(\bar{x})=0$. Since $c\ne 0$ and $P_c\ge 0$ in $(a,b)$, $P_c\equiv\lambda(x-\bar{x})^2$ for some $\lambda>0$. So $P'_c(\bar{x})=0$ and $P''_c(\bar{x})>0$.
 It is easy to see that $f\in C^3(a,b)$. Take a derivative of equation (\ref{eq_1}) at $\bar{x}$, using the fact $f_{\nu}(\bar{x})=f'_{\nu}(\bar{x})=0$, we have
      $\nu(1-\bar{x}^2)f''_{\nu}(\bar{x})=P'_c(\bar{x})=0$.
   So $f''_{\nu}(\bar{x})=0$. 
   Now we have $f_{\nu}(\bar{x})=f'_{\nu}(\bar{x})=f''_{\nu}(\bar{x})=0$ and $f'''_{\nu}(\bar{x})>0$ which imply that $f_{\nu}(x)<0$ for $x<\bar{x}$ and close to $x_{\nu}$, violating  $f_{\nu}>0$ in $(\bar{x}-\epsilon, \bar{x})$, a contradiction. 
   Similarly,  there does not exist $\bar{x}\in (a,b)$ and $\epsilon>0$ such that $f_{\nu}(\bar{x})=0$, $f'_{\nu}(\bar{x})\le 0$ and $f_{\nu}<0$ in $(\bar{x}, \bar{x}+\epsilon)$.

  % Next, take the second derivative of (\ref{eq_1}), and using $f_{\nu}(\bar{x})=f'_{\nu}(\bar{x})=f''_{\nu}(\bar{x})=0$, we have
%   \[
%        \nu(1-\bar{x}^2)f'''_{\nu}(\bar{x})=P''_c(\bar{x})> 0.
%   \]
%   So $f'''_{\nu}(\bar{x})> 0$. Then there exists some $\epsilon>0$ such that for all $|x-\bar{x}|<\epsilon$, $f'''_{\nu}(x)>0$.
%
%   Now we have that for any $x\in (\bar{x}-\epsilon,\bar{x})$, there exists some $\theta\in [x,\bar{x}]$, such that
%   \[
%     \begin{split}
%      0<f_{\nu}(x)&=f_{\nu}(\bar{x})+f'_{\nu}(\bar{x})(x-\bar{x})+\frac{1}{2}f''_{\nu}(\bar{x})(x-\bar{x})^2+\frac{1}{6}f'''_{\nu}(\theta)(x-\bar{x})^3\\
%                    & =\frac{1}{6}f'''_{\nu}(\theta)(x-\bar{x})^3<0,
%      \end{split}
%   \]
%   contradiction.

   Now we prove that there exists at most one $x_{\nu}\in (a,b)$ such that $f_{\nu}(x_{\nu})=0$. Clearly $f_{\nu}$ is not identically equal to zero on $(a,b)$. If $f_{\nu}$ has more than one zero point in $(a,b)$, then there exist some $x_\nu<y_{\nu}$ in $(-1,1)$ such that $f_{\nu}(x_{\nu})=f_\nu(y_\nu)=0$, and either $f_\nu<0$ in $(x_\nu,y_\nu)$ or $f_\nu>0$  in $(x_\nu,y_\nu)$. If $f_\nu<0$ in $(x_\nu,y_\nu)$, then $f_{\nu}(x_{\nu})=0$ and $f'_{\nu}(x_{\nu})\le 0$. If $f_\nu>0$ in $(x_\nu,y_\nu)$, then $f_{\nu}(y_{\nu})=0$ and $f'_{\nu}(y_{\nu})\le 0$. We have proved in the above that neither could occur, a contradiction.

Next, we prove that if $f_\nu(a)>0$, then $f_{\nu}>0$ on $(a,b)$. If $f_{\nu}$ is not positive on the whole interval $(a,b)$, then let $\bar{x}\in (a,b)$ be the first point greater than $a$ such that $f_{\nu}(\bar{x})=0$. %If $\bar{x}=b$, then $f_\nu>0$ in $(a,b)$, otherwise $\bar{x}\in (a,b)$.
    We have $f'_{\nu}(\bar{x})\le 0$, and $f_{\nu}>0$ in $(a,\bar{x})$, a contradiction. 
   Similarly, we have that if $f_{\nu}(b)<0$, then $f_{\nu}<0$ on $(a,b)$.
\end{proof}

\begin{cor}\label{cor2_1}
   Let $0< \nu \le 1$, $c\in\mathbb{R}^3\setminus\{0\}$. Assume that $P_c\ge 0$ in $(-1,1)$, then
   \[
      U_{\nu, \theta}^+(x)>0, \quad U_{\nu, \theta}^-(x)<0, \quad -1< x< 1.
   \]
\end{cor}
\begin{proof}
   Since $U_{\nu, \theta}^+(-1)=2\nu+2\sqrt{\nu^2+c_1}>0$ and $U_{\nu, \theta}^-(1)=-2\nu-2\sqrt{\nu^2+c_2}<0$, the corollary follows from Lemma \ref{lem2_4}.
\end{proof}

\begin{lem}\label{lem2_6}
   Let $0< \nu \le 1$, $c\in \mathbb{R}^3$, $-1\le a<b\le 1$, $f_{\nu}\in C^1(a,b)\cap C^0[a,b]$ be a solution of (\ref{eq_1}) in $(a,b)$. If there exist some $m, M>0$ such that
   \begin{equation*}%\label{eq2_6_1}
        m\le |f_{\nu}(x)|\le M, \quad \forall a\le x\le b,
   \end{equation*}
   then %there exists some $C>0$, depending only on $c$, such that
   \begin{equation}\label{eq2_6_2}
       ||\frac{1}{2}f^2_{\nu}-P_c||_{L^{\infty}(a, b)}\le \max\left\{\left(2M+\sqrt{6}|c|/m\right)\nu, \left|\frac{1}{2}f^2_{\nu}(a)-P_c(a)\right|, \left|\frac{1}{2}f^2_{\nu}(b)-P_c(b)\right| \right\}.%C\left(\frac{1}{m}+M\right)\nu+|\frac{1}{2}f^2_{\nu}(a)-2P_c(a)|+|\frac{1}{2}f^2_{\nu}(b)-2P_c(b)|.
   \end{equation}
 Moreover, for any $0<\epsilon<(b-a)/2$, %there exists some $C>0$, depending only on $c$, such that
   \begin{equation}\label{eq2_6_3}
       ||\frac{1}{2}f^2_{\nu}-P_c||_{L^{\infty}(a+\epsilon,b-\epsilon)}\le \nu \cdot \max\left\{2M+\sqrt{6}|c|/m, 10M/\epsilon\right\}.%C\left(\frac{1}{m}+\frac{M}{\epsilon}\right)\nu.
   \end{equation}
   \end{lem}
\begin{proof}
%Let $C$  denote a universal constant, depending only on $c$ unless stated otherwise, which may vary from line to line.

 We first prove (\ref{eq2_6_2}).  
%To prove (\ref{eq2_6_2}) for some constant $C$, we only need to prove
%\begin{equation}\label{eq2_6_4}
%       ||\frac{1}{2}f^2_{\nu}-2P_c||_{L^{\infty}(a, b)}\le C\nu+|\frac{1}{2}f^2_{\nu}(a)-2P_c(a)|+|\frac{1}{2}f^2_{\nu}(b)-2P_c(b)|.
%       \end{equation}
%       Then since $f_{\nu}\ge \frac{1}{M}$, (\ref{eq2_6_2}) is obtained from the above.
  Since $\frac{1}{2}f^2_{\nu}-P_c$ is continuous on $[a,b]$, there exists some $z_{\nu}\in [a,b]$, such that
   \[
       \left|\frac{1}{2}f^2_{\nu}(z_\nu)-P_c(z_\nu)\right|=\max_{[a,b]}\left|\frac{1}{2}f^2_{\nu}-P_c\right|.
   \]
   If $z_{\nu}=a$ or $b$, then  (\ref{eq2_6_2}) is proved. Otherwise we have
   \[
      f_\nu(z_\nu)f'_\nu(z_\nu)-P'_c(z_\nu)=0.
   \]
   Since $|f_{\nu}(z_{\nu})|\ge m$, we have
   \[
        |f'_\nu(z_\nu)|=|P'_c(z_\nu)|/|f_\nu(z_\nu)|\le (|c_1|+|c_2|+2|c_3|)/m\le \sqrt{6}|c|/m.
   \]
   Then by (\ref{eq_1}), we have
   \[
      \left|\frac{1}{2}f^2_{\nu}(z_\nu)-P_c(z_\nu)\right|=\nu|(1-z^2_\nu)f'_\nu(z_\nu)+2z_\nu f_\nu(z_\nu)|\le \left(2M+\sqrt{6}|c|/m\right)\nu. %C\left(\frac{1}{m}+M\right)\nu.
   \]
So  (\ref{eq2_6_2}) is proved.

Next, we prove (\ref{eq2_6_3}).
   %For any $a\le \alpha<\beta\le b$, integrate the equation (\ref{eq_1}) on $(\alpha,\beta)$, using $f_{\nu}\le M$, we have that
%   \begin{equation}\label{eq2_6_5}
%       \begin{split}
%          \left| \int_{\alpha}^{\beta}\left(\frac{1}{2}f^2_{\nu}-P_c\right) dx\right| & =\nu\left|\int_{\alpha}^{\beta}(1-x^2)f'_{\nu}+2x f_{\nu}\right|\\
%          & =\nu\left|(1-\beta^2)f_{\nu}(\beta)-(1-\alpha^2)f_{\nu}(\alpha)+\int_{\alpha}^{\beta}4x f_{\nu}\right|\\
%          & \le C\nu
%       \end{split}
%   \end{equation}
%   where $C=C(M)$ is some constant depending on  $M$.
%
%   Consequently, we have
%   \[
%      \min_{[\alpha, \beta]}\left|\frac{1}{2}f^2_{\nu}-P_c\right|\le \frac{C(M))\nu}{\alpha-\beta}.
%   \]
%   In particular, for any $\epsilon>0$, let $\alpha=a+\epsilon/4$, $\beta= a+\epsilon/2$,
%   \[
%       \min_{[a+\epsilon/4, a+\epsilon/2]}\left|\frac{1}{2}f^2_{\nu}-P_c\right|\le \frac{4C(M)\nu}{\epsilon},
%   \]
%
%
%   Similarly, there exists some constant $C$,
%   such that
%   \[
%       \min_{[b-\epsilon/2, b-\epsilon/4]}\left|\frac{1}{2}f^2_{\nu}-P_c\right|\le \frac{4C(M)\nu}{\epsilon}.
%   \]
  By Lemma \ref{lem2_5}, for any $\epsilon>0$, there exist some $x_\nu \in [a, a+\epsilon]$, and $y_{\nu}\in [b-\epsilon, b]$, satisfying
   \[
       \left|\frac{1}{2}f^2_{\nu}(x_\nu)-P_c(x_\nu)\right|\le 10 M\nu/\epsilon, \quad \left|\frac{1}{2}f^2_{\nu}(y_\nu)-P_c(y_\nu)\right|\le 10M\nu/\epsilon.
   \]
Apply (\ref{eq2_6_2}) on $(x_{\nu}, y_{\nu})$, we have
   \[
       ||\frac{1}{2}f^2_{\nu}-P_c||_{L^{\infty}(x_{\nu},y_{\nu})}\le \max\left\{\left(2M+\sqrt{6}|c|/m\right)\nu, 10M\nu/\epsilon\right\}. %C\left(\frac{1}{m}+M+\frac{M}{\epsilon}\right)\nu\le C\left(\frac{1}{m}+\frac{M}{\epsilon}\right)\nu.
   \]
Notice $(a+\epsilon,b-\epsilon)\subset (x_{\nu},y_{\nu})$,  the lemma is proved.
\end{proof}

\begin{cor}\label{cor2_2}
   Let $0< \nu \le 1$, $c\in \mathbb{R}^3$, $-1\le a<b\le 1$, $P_c\ge 0$ in $(a,b)$, $f_{\nu}\in C^1(a,b)\cap C^0[a,b]$ be a solution of (\ref{eq_1}) in $(a,b)$. If there exist some $m, M>0$ such that
   \begin{equation*}%\label{eqcor2_2_1}
        m\le f_{\nu}(x)\le M, \quad \forall a<x<b, \quad 0<\nu\le 1,
   \end{equation*}
   then there exists a universal constant $C>0$,  such that
   \begin{equation*}%\label{eqcor2_2_2}
   \begin{split}
       ||f_{\nu}-\sqrt{2P_c}||_{L^{\infty}(a, b)}\le  & \frac{C}{m}\max\left\{\left(M+|c|/m\right)\nu,(M+\sqrt{|c|})|f_{\nu}(a)-\sqrt{2P_c(a)}|,\right. \\
       & \left.(M+\sqrt{|c|})|f_{\nu}(b)-\sqrt{2P_c(b)}| \right\}.
       \end{split}%\frac{C}{m}\left(\frac{1}{m}+M\right)\nu+\frac{CM}{m}|f_{\nu}(a)-\sqrt{2P_c(a)}|+\frac{CM}{m}|f_{\nu}(b)-\sqrt{2P_c(b)}|.
   \end{equation*}
 Moreover, for any $0<\epsilon<(b-a)/2$, %there exists some $C>0$, depending only on $c$, such that
   \begin{equation*}%\label{eqcor2_2_3}
       ||f_{\nu}-\sqrt{2P_c}||_{L^{\infty}(a+\epsilon,b-\epsilon)}\le \frac{C\nu}{m}\max\left\{M+|c|/m,M/\epsilon\right\}.
   \end{equation*}
   \end{cor}
   \begin{proof}
      Since $f_{\nu}\ge m>0$  and $P_c\ge 0$ in $(a,b)$, we have $m\le f_\nu+\sqrt{2P_c}\le M+\sqrt{10|c|}$ in $(a,b)$. So the corollary follows from Lemma \ref{lem2_6}.
   \end{proof}

\addtocounter{cor}{-1}
\renewcommand{\thecor}{\thesection.\arabic{cor}'}%
\begin{cor}\label{cor2_2'}
	Let $0< \nu \le 1$, $-1\le a<b\le 1$, $P_c\ge 0$ in $(a,b)$, $f_{\nu}\in C^1(a,b)\cap C^0[a,b]$ be a solution of (\ref{eq_1}) in $(a,b)$. If there exist some $m, M>0$ such that
   \begin{equation*}%\label{eqcor2_2_1'}
        -M\le f_{\nu}(x)\le -m, \quad \forall a<x<b, \quad 0<\nu\le 1,
   \end{equation*}
   then there exists a universal constant $C>0$, such that
   \begin{equation*}%\label{eqcor2_2_2'}
   \begin{split}
       ||f_{\nu}+\sqrt{2P_c}||_{L^{\infty}(a, b)} \le & \frac{C}{m}\max\left\{\left(M+|c|/m\right)\nu, (M+\sqrt{|c|})|f_{\nu}(a)+\sqrt{2P_c(a)}|,\right.\\
       & \left.(M+\sqrt{|c|})|f_{\nu}(b)+\sqrt{2P_c(b)}| \right\}.
       \end{split}% \frac{C}{m}\left(\frac{1}{m}+M\right)\nu+\frac{CM}{m}|f_{\nu}(a)+\sqrt{2P_c(a)}|+\frac{CM}{m}|f_{\nu}(b)+\sqrt{2P_c(b)}|.
   \end{equation*}
 Moreover, for any $0<\epsilon<(b-a)/2$, there exists a  universal constant $C>0$, such that % , depending only on $c$, such that
   \begin{equation*}%\label{eqcor2_2_3'}
       ||f_{\nu}+\sqrt{2P_c}||_{L^{\infty}(a+\epsilon,b-\epsilon)}\le \frac{C\nu}{m}\max\left\{M+|c|/m, M/\epsilon\right\}.%\frac{C}{m}\left(\frac{1}{m}+\frac{M}{\epsilon}\right)\nu
   \end{equation*}
\end{cor}
\renewcommand{\thecor}{\thesection.\arabic{cor}}%

%**********************Combine Lemma \ref{lem2_10original} and \ref{lem2_11original}************************

 \begin{lem}\label{lem2_10}
      Let $0< \nu \le 1$, $c\in \mathbb{R}^3$, $-1\le a<b\le 1$, $\alpha\ge 0$, $f_{\nu}\in C^1(a,b)\cap C^0[a,b]$ be a solution of (\ref{eq_1}) in $(a,b)$, satisfying $f_{\nu}(a)>0$ or $f_{\nu}(b)<0$. Suppose there exists some $\bar{x}\in \mathbb{R}$ such that $P_c(\bar{x})=\min_{[a,b]}P_c\ge -C_1(b-a)^{\alpha}$, $\mathrm{dist}(\bar{x}, [a,b])\le C_1(b-a)$, and $P_c(x)\le P_c(\bar{x})+C_1|x-\bar{x}|^{\alpha}$ for $a\le x\le b$, and %Suppose $\min_{[a,b]}P_c(x)=P_c(\bar{x})$ for some $\bar{x}\in [a,b]$, satisfying $c_3\le  0$, %$min_{[a,b]}P_c=P_c(\bar{x})$ for some $\bar{x}\in [a,b]$,  and $P_c$ not identically equal to $0$ on $[a,b]$. If $P'_c(\bar{x})=0$,$P_c(\bar{x})\ge -C_1(b-a)^{\alpha}$, $\alpha\le 1$,%for some $\alpha>0$ and $\delta>0$, and
      \[
          |\frac{1}{2}f^2_{\nu}(a)-P_c(a)|+|\frac{1}{2}f^2_{\nu}(b)-P_c(b)|\le C_1(b-a)^{\alpha},
      \]
      for some positive constants $C_1$. Then there exists some constant $C$, depending only on $C_1$ and an upper bound of $|c|$, such that for $\nu<\sqrt{2C_1}/4(b-a)^{\alpha/2}$,
      \[
         ||\frac{1}{2}f^2_{\nu}-P_c||_{L^{\infty}(a,b)}\le C(b-a)^{\alpha}+C\nu(b-a)^{-\alpha/2}.
      \]
   \end{lem}
   \begin{proof}
   We only prove for the case $f_{\nu}(a)>0$, the case $f_{\nu}(b)<0$ can be proved similarly.
   For convenience denote $h:=\frac{1}{2}f_{\nu}^2-P_c$ and $\delta=b-a$. Let $C$ be a positive constant, depending only on $C_1$ and 
   an upper bound of $|c|$, which may vary from line to line.

     Suppose $\max_{[a,b]}|h|=|h(\tilde{z})|$ for some $\tilde{z}\in [a,b]$. If $\tilde{z}=a$ or $b$, then we are done. Suppose $\tilde{z}\in (a,b)$. Then
        $0=h'(\tilde{z})=f_{\nu}(\tilde{z})f'_{\nu}(\tilde{z})-P'_c(\tilde{z})$. 
     So
       \begin{equation}\label{eq2_10_1}
          |f'_{\nu}(\tilde{z})|=|P_c'(\tilde{z})|/|f_{\nu}(\tilde{z})|.
       \end{equation}
    If  $P_{c}(\bar{x})> 2C_1\delta^{\alpha}$, then since $|h(a)|\le C_1\delta^{\alpha}$,  we have
       \[
          f_{\nu}^2(a)\ge 2P_{c}(a)-2C_1\delta^{\alpha}\ge 2P_c(\bar{x})-2C_1\delta^{\alpha}\ge  2C_1\delta^{\alpha}.
       \]
       Since $f_{\nu}(a)>0$, we have $f_{\nu}(a)\ge \sqrt{2C_1}\delta^{\alpha/2}$. Then by Lemma \ref{lem2_2}, we have that for $\nu<\sqrt{2C_1}/4\delta^{\alpha/2}$,
       \[
           f_{\nu}(x)\ge \min\{\sqrt{2C_1}\delta^{\alpha/2},C_1\delta^{\alpha}/(2\nu)\}\ge \sqrt{2C_1}\delta^{\alpha/2}, \quad a<x<b.
       \]
       With this, we deduce from (\ref{eq2_10_1}) that
       \begin{equation}\label{eq2_10_2}
          |f'_{\nu}(\tilde{z})|\le C\delta^{-\alpha/2}.
       \end{equation}
       By (\ref{eq_1}) and Lemma \ref{lem2_1} we have the desired estimate
       \begin{equation}\label{eq2_10_3}
          |h(\tilde{z})|\le C\nu |f'_{\nu}(\tilde{z})|+C\nu|f_{\nu}(\tilde{z})|\le C\nu\delta^{-\alpha/2}.
          \end{equation}
       If $P_{c}(\bar{x})\le 2C_1\delta^{\alpha}$, then using the hypothesis $P_c(\bar{x})\ge -C_1(b-a)^{\alpha}$, $\mathrm{dist}(\bar{x}, [a,b])\le C_1(b-a)$, and $P_c(x)\le P_c(\bar{x})+C_1|x-\bar{x}|^{\alpha}$, we have
       \[
           -C_1\delta^{\alpha}\le P_c(\bar{x})\le P_c(\tilde{z})\le P_c(\bar{x})+C|\tilde{z}-\bar{x}|^{\alpha}\le C\delta^{\alpha}.%+C\delta^{\alpha}\le C\delta^{\alpha}.
           \]
            So
       \[
           \frac{1}{2}f^2_{\nu}(\tilde{z})\ge |h(\tilde{z})|-|P_{c}(\tilde{z})| \ge |h(\tilde{z})|-C\delta^{\alpha}.
       \]
        If $|h(\tilde{z})|\le 2C\delta^{\alpha}$, then we are done. Otherwise we have $|f_{\nu}(\tilde{z})|\ge \sqrt{2C}\delta^{\alpha/2}$. With this,  we deduce (\ref{eq2_10_2}) using (\ref{eq2_10_1}), and obtain (\ref{eq2_10_3}) as above. The lemma is proved.
   \end{proof}

\begin{lem}\label{lem2_12}
   Let $0< \nu \le 1$, $-1\le a<b\le 1$, $c\in \mathbb{R}^3$, $k\ge 0$ be an integer, assume $P_c\ge \delta>0$ on $(a,b)$, $f_\nu\in C^k[a,b]$ is a solution to (\ref{eq_1}). Suppose there exists some $ M>0$ such that $f_\nu\le M$ on $(a,b)$, and
    \begin{equation}\label{eq2_12_2}
      \left|\frac{d^i}{dx^i}(f_{\nu}-\sqrt{2P_c})(a)\right|+
      \left|\frac{d^i}{dx^i}(f_{\nu}-\sqrt{2P_c})(b)\right|\le M\nu, \quad \forall 0\le i\le k.
   \end{equation}
     Then
     there exists some $C>0$, depending only on  $\delta$, $k$, $M$ and an upper bound of $|c|$, such that
   \begin{equation}\label{eq2_12_1}
       ||f_{\nu}-\sqrt{2P_c}||_{C^k(a,b)}\le C\nu.
   \end{equation}
\end{lem}
\begin{proof}

%***************************

Throughout the proof, $C$ and $\nu_0$ denote  various positive constants, depending only on  $\delta$, $k$, $M$ and an upper bound of $|c|$. $C$ will be chosen first and will be large, and $\nu_0$ will be small, and its choice may depend on the largeness of $C$. We will only need to prove (\ref{eq2_12_1}) for $\nu<\nu_0$, since it is obvious for $\nu\ge \nu_0$.

 %Denote a large and a small constants by $C$ and $\nu_0$ respectively, depending only on  $\delta$, $k$, $M$ and an upper bound of $|c|$, which may vary from line to line. $C$ will be chosen first, then $\nu_0$ will be chosen later. %We prove the theorem by induction.

   For convenience, write $Q=\sqrt{2P_c}$. Denote
   \begin{equation}\label{eq2_12_9}
      h_0(x):=\frac{1}{\nu}(f_{\nu}(x)-Q(x)),\quad  h_i(x):=\frac{d^i}{d x^i}h_0(x),\quad  \forall i\ge 1.
      \end{equation}
   Rewrite (\ref{eq_1}) as
   \begin{equation}\label{eq2_12_7}
     \nu h'_0(x)=\frac{1}{1-x^2}F(x,h_0(x)), 
   \end{equation}
   %\[
%      \nu(1-x^2) g'_\nu+2x\nu g_\nu+\frac{1}{2}\nu g^2_{\nu}+(1-x^2)Q'+2xQ+g_\nu Q=0.
%   \]
   where
   \[
      F(x,h_0):=-\{2x\nu h_0+\frac{1}{2}\nu h^2_0+(1-x^2)Q'(x)+2xQ(x)+h_0 Q(x)\}.
      \]
%It can be seen that $F$ is smooth in $x$ and $g_{\nu}$ on $[a,b]$.

   \noindent\textbf{Claim}: For all $n\ge 2$, and for $x\in [a,b]$,
   \begin{equation}\label{eq2_12_8}
       \nu h_n(x)=\frac{1}{1-x^2}[2(n-1)\nu x+F_{h_0}(x,h_0)]h_{n-1}+\frac{1}{1-x^2}F_n(x,h_0,...,h_{n-2}),%:=H_n.
   \end{equation}
   where $F_n(x,h_0, h_1,...,h_{n-2})$ satisfies that for any compact subset $K\subset [a,b]\times \mathbb{R}^{n-1}$ and for any integer $m\ge 0$
   \[
      ||F_n||_{C^m(K)}\le C',
   \]
   for some $C'$ depending only on $n,m$ and $K$.
   %smooth in all $x,g_{\nu}, h_1,..., h_{n-2}$ on $[a,b]$, such that

   \noindent\emph{Proof of the Claim}: We prove it by induction. Differentiating (\ref{eq2_12_7}) leads to (\ref{eq2_12_8}) for $n=2$, with $F_2(x,h_0)=F_x(x,h_0)$.
 Now suppose that (\ref{eq2_12_8}) is true for some $n\ge 2$, and we will prove (\ref{eq2_12_8}) for $n+1$. Differentiating (\ref{eq2_12_8}), we have
\[
   \nu h_{n+1}=\frac{1}{1-x^2}(2n\nu x+F_{h_0})h_{n}+\frac{1}{1-x^2}F_{n+1}(x,h_0,...,h_{n-1}),
\]
where
\[
  \begin{split}
   F_{n+1}(x,h_0,...,h_{n-1})& :=2(n-1)\nu+F_{h_0h_0}(x,h_0)+F_{xh_0}(x,h_0)+\partial_xF_n(x,h_0,...,h_{n-2})\\
   & +\sum_{i=0}^{n-2}\partial_{h_i}F_n(x_0,h_0,...,h_{n-2})h_{i+1}.
   \end{split}
\]
%\[
%    \nu h_{n}=\frac{1}{1-x^2}\left((2(n-1)\nu x+F_{g_{\nu}})h_{n-1}+F_n\right).
%\]
%Take derivative both sides with respect to $x$, we have
%\[
%   \begin{split}
%   \nu h_{n+1} & =\frac{2x}{(1-x^2)^2}\left((2(n-1)\nu x+F_{g_{\nu}})h_{n-1}+F_n\right)\\
%       & +\frac{1}{1-x^2}\left((2(n-1)\nu x+F_{g_{\nu}})h_{n}+(2(n-1)\nu x+F_{g_{\nu}})'h_{n-1}+F'_n\right)\\
%       & =\frac{1}{1-x^2}2x\nu h_n+\frac{1}{1-x^2}(2(n-1)\nu x+F_{g_{\nu}})h_{n}+\frac{1}{1-x^2}\left((2(n-1)\nu x +F_{g_{\nu}})'h_{n-1} \right. \\
%       & \left. +F'_n(x,h_0,...,h_{n-2})\right)\\
%       & =: \frac{1}{1-x^2}(2n\nu x+F_{g_{\nu}})h_{n}+\frac{1}{1-x^2}F_{n+1}(x,h_0,...,h_{n-1}).
%   \end{split}
%\]
The claim is proved.

   %Since $|(f_{\nu}-Q)(a)|\le C\nu$, and $Q(a)>\sqrt{\delta}$, so for small $\nu$ we have $f_{\nu}(a)>\mu$ for some $\mu>0$. Then by Lemma \ref{lem2_2}, there exists some $m>0$, such that $f_{\nu}\ge m$ on $(a,b)$. Then
   We prove the lemma by induction on $k$. By (\ref{eq2_12_2}) with $i=0$,  $f_{\nu}(a)\ge \sqrt{2P_c(a)}-M\nu\ge \sqrt{2\delta}-M\nu\ge \sqrt{\delta}$ for $\nu\le \nu_0$. %for $\nu<\frac{\sqrt{2}-1}{M}$, we have $f_{\nu}(a)\ge \sqrt{\delta}$.
     By Lemma \ref{lem2_2},  $f_{\nu} \ge \sqrt{\delta}$ in $(a,b)$ for $\nu\le \nu_0$ on $[a,b]$. Then by Corollary \ref{cor2_2},  we have $|h_0|\le C$ in $[a,b]$.
  Let $z_\nu \in [a,b]$ such that $|h_1(z_{\nu})|=\max_{[a,b]}|h_1|$. By (\ref{eq2_12_2}), $|h_1(a)|, |h_1(b)|\le M$. If $z_{\nu}=a$ or $b$, the lemma holds for $k=1$. If $z_\nu\ne a$ or $b$, then by (\ref{eq2_12_8}),
   \[
      0=\nu  h'_1(z_\nu)=\nu h_2(z_{\nu})=\frac{1}{1-z_{\nu}^2}\left\{[2\nu z_{\nu}+F_{h_0}(z_\nu, h_0(z_\nu)) ] h_1(z_\nu)+F_2(z_\nu,h_0(z_\nu))\right\}=0,
   \]
   and, by the boundedness of $h_0$ and the property of $F_2$, $|F_2(x,h_0(x))|\le C$ on $[a,b]$.
  Since $|h_0|\le C$ and $Q=\sqrt{2P_c}\ge \sqrt{2\delta}$, we have, for $\nu \le \nu_0$, that%there exists some $\nu_0>0$, such that for $\nu<\nu_0$,
   \begin{equation*}%\label{eq2_12_4}
       |2\nu x+F_{h_0}(x,h_0(x))|=|\nu h_0(x)+Q(x)+4\nu x|\ge Q-C\nu\ge  \sqrt{\delta}>0, \quad a<x<b.
   \end{equation*}
   So  we have
   \[
       \max_{[a,b]}|h_1|=|h_1(z_\nu)|=\frac{|F_2(z_\nu, h_0(z_\nu))|}{|2\nu z_{\nu}+F_{h_0}(z_\nu, h_0(z_\nu))|}\le C,
   \]
   and the lemma holds for $k=1$.

   Next, assume the lemma holds for all $1\le k \le n$ for some $n$, and we prove it for $k=n+1$.
  By the induction hypothesis, $|h_k|\le C$ in $[a,b]$, for all $1\le k \le n$.
  Let $z_{\nu}\in [a,b]$ such that $|h_{n+1}(z_{\nu})|=\max_{[a,b]}|h_{n+1}|$. By (\ref{eq2_12_2}), $|h_{n+1}(a)|, |h_{n+1}(b)|\le M$. If $z_{\nu}=a$ or $b$, the lemma holds for $k=n+1$. Otherwise by (\ref{eq2_12_8}),
   \begin{equation*}%label{eq2_12_5}
   \begin{split}
      & 0=\nu  h'_{n+1}(z_\nu) =\nu h_{n+2}(z_{\nu,n})\\
      & =\frac{1}{1-z_{\nu,n}^2}\{[2(n+1)\nu z_{\nu,n}+F_{h_0}(z_{\nu},h_0(z_{\nu}))]h_{n+1}(z_{\nu})+F_{n+2}(z_{\nu},h_0(z_{\nu}),...,h_{n}(z_{\nu}))\},
     \end{split}
   \end{equation*}
   and, by the induction hypothesis and the property of $F_{n+2}$, $|F_{n+2}(x,h_0,...,h_{n})|\le C$ on $[a,b]$. 
   As above, for $\nu\le \nu_0$,
   \[
       |2(n+1)\nu z_{\nu}+F_{h_0}(z_{\nu},h_0(z_{\nu}))|>\sqrt{\delta}, \quad a<x<b,
        \]
 and therefore
   \[
      \max_{[a,b]}|h_{n+1}|=|h_{n+1}(z_{\nu})|=\frac{|F_{n+2}(z_{\nu})|}{|2(n+1)\nu z_{\nu}+F_{h_0}(z_{\nu})|}\le C.
   \]
So the lemma holds for $k=n+1$. The lemma is proved.
  \end{proof}

\addtocounter{lem}{-1}
\renewcommand{\thelem}{\thesection.\arabic{lem}'}%
\begin{lem}\label{lem2_12'}
	Let $0< \nu \le 1$, $-1\le a<b\le 1$, $c\in \mathbb{R}^3$, $k\ge 0$ be an integer, assume $P_c\ge \delta>0$ on $(a,b)$, $f_\nu\in C^k[a,b]$ is a solution to (\ref{eq_1}). Suppose there exists some $ M>0$ such that $f_\nu>- M$ on $(a,b)$, and 
	    \begin{equation*}%\label{eq2_12_2'}
     \left|\frac{d^i}{dx^i}(f_{\nu}+\sqrt{2P_c})(a)\right|+
      \left|\frac{d^i}{dx^i}(f_{\nu}+\sqrt{2P_c})(b)\right|\le M\nu, \quad \forall 0\le i \le k.
   \end{equation*}
     Then
     %for any $0<\epsilon<\frac{b-a}{2}$ and positive integer $k$,
     there exist some $C>0$ and $\nu_0>0$, depending only on  $\delta$, $k$, $M$ and an upper bound of $|c|$, such that
   \begin{equation*}%\label{eq2_12_1'}
       ||f_{\nu}+\sqrt{2P_c}||_{C^k(a,b)}\le C\nu, \quad \forall 0<\nu<\nu_0.
   \end{equation*}

\end{lem}
\renewcommand{\thelem}{\thesection.\arabic{lem}}%

\begin{lem}\label{lem2_7}
   Let $0< \nu \le 1$, $-1\le a<b\le 1$, $c\in \mathbb{R}^3$, $k\ge 0$ be an integer, assume $P_c\ge \delta>0$ on $[a,b]$, $f_\nu\in C^k[a,b]$ is a solution to (\ref{eq_1}), $f_\nu(a)>0$, and there exists some $M>0$ such that $f_{\nu}\le M$ in $[a,b]$. 
      Then $f_{\nu}>0$ in $[a,b]$. Moreover,  for any $0<\epsilon<(b-a)/2$, there exists some $C>0$, depending only on $\epsilon$, $\delta$, $M$, $k$ and an upper bound of $|c|$, such that
   \begin{equation}\label{eq2_7_1}
       ||f_{\nu}-\sqrt{2P_c}||_{C^k(a+\epsilon,b-\epsilon)}\le C\nu.
   \end{equation}
\end{lem}

\begin{proof}
Let $C$ be a constant, depending only on $\epsilon$, $\delta$, $k$, $M$ and an upper bound of $|c|$, which may vary from line to line.

Since $P_c\ge \delta>0$ on $[a,b]$ and $f_{\nu}(a)>0$, we have 
 $f_{\nu}(x)\ge \min\{f_{\nu}(a), \sqrt{\delta}, \delta/(4\nu)\}>0$ in $[a,b]$ by Lemma \ref{lem2_2}. The positivity of $f_{\nu}$ on $[a,b]$ can also be deduced from Lemma \ref{lem2_4}. Applying Lemma \ref{lem2_5} on $[a,a+\epsilon/2]$ and $[b-\epsilon/2,b]$ respectively, there exist some $x_{\nu}\in [a,a+\epsilon/2]$ and $y_{\nu}\in [b-\epsilon/2,b]$, such that
    $\left|\frac{1}{2}f^2_{\nu}(x_{\nu})-P_c(x_{\nu})\right|+\left|\frac{1}{2}f^2_{\nu}(y_{\nu})-P_c(y_{\nu})\right|\le C\nu$. 
Using $f_{\nu}>0$ and $P_c\ge \delta$, we have
\begin{equation*}%\label{eq2_7_2}
  \left|f_{\nu}(x_{\nu})-\sqrt{2P_c(x_{\nu})}\right|+\left|f_{\nu}(y_{\nu})-\sqrt{2P_c(y_{\nu})}\right|\le C\nu.
\end{equation*}
Since $P_c\ge \delta$, there exists $\nu_0>0$, depending only on $\epsilon, \delta, M$ and $c$, such that $f_{\nu}(x_{\nu})\ge \sqrt{2P_c(x_{\nu})}-C\nu\ge \sqrt{2\delta}-C\nu\ge \sqrt{\delta}$ for $\nu\le \nu_0$. Note that for $\nu\ge \nu_0$, (\ref{eq2_7_1}) is obvious. So we only need to consider $\nu\le \nu_0$. Applying Lemma \ref{lem2_2} on $[x_{\nu}, y_{\nu}]$, we have $f_{\nu}\ge 1/C$ on $[x_{\nu}, y_{\nu}]$. Since we also have $f_{\nu}\le M$ on $[x_{\nu}, y_{\nu}]$, by Corollary \ref{cor2_2}, we have that
\[
    ||f_{\nu}-\sqrt{2P_c}||_{L^{\infty}(a+\epsilon/2, b-\epsilon/2)}\le ||f_{\nu}-\sqrt{2P_c}||_{L^{\infty}(x_{\nu},y_{\nu})}\le C\nu.
\]

For convenience, denote $h_i$, $i\ge 0$, as in (\ref{eq2_12_9}). We have proved that $|h_0|\le C$ in $[a+\epsilon/2, b-\epsilon/2]$.
    So for any $0<\epsilon<(b-a)/2$,
   \begin{equation*}%\label{eq2_7_9}
      \big|\int_{a+\epsilon/2}^{a+\epsilon}h_1(x) dx\big|\le C, %=|h_0(a+\epsilon)-h_0(a+\epsilon/2)|\le C, 
      \quad \big|\int_{b-\epsilon}^{b-\epsilon/2}h_1(x) dx\big|\le C. %=|h_0(b-\epsilon/2)-h_0(b-\epsilon)|\le C.
   \end{equation*}
   Thus there exist some $x_\nu\in [a+\epsilon/2,a+\epsilon]$, and  $y_\nu\in [b-\epsilon, b-\epsilon/2]$, such that $|h_1(x_\nu)|\le C$, $|h_1(y_{\nu})|\le C$. Apply Lemma \ref{lem2_12} on $[x_\nu, y_{\nu}]$, we have $|h_1(x)|\le C, \quad x_{\nu}<x<y_{\nu}$.
      So the lemma holds for $k=1$.

   Next, assume the lemma holds for all $1\le k \le n$ for some $n$, and we prove it for $k=n+1$.
  By the induction hypothesis, for any $\epsilon>0$,  $|h_k|\le C$  in $(a+\epsilon/2, b-\epsilon/2)$ for all $1\le k \le n$. It follows that
 \[
    \big|\int_{a+\epsilon/2}^{a+\epsilon}h_{n+1}(x) dx\big|\le C,  \quad \big|\int_{b-\epsilon}^{b-\epsilon/2}h_{n+1}(x) dx\big|\le C.
 \]
 So there exist some $x_\nu\in [a+\epsilon/2,a+\epsilon]$, and  $y_\nu\in [b-\epsilon, b-\epsilon/2]$, such that $|h_{n+1}(x_\nu)|\le C, \quad |h_{n+1}(y_{\nu})|\le C$.  Apply Lemma \ref{lem2_12} on $[x_\nu, y_{\nu}]$, we have $|h_{n+1}(x)|\le C, \quad x_{\nu}<x<y_{\nu}$.
      The lemma holds for $k=n+1$. The lemma is proved.
\end{proof}

 \addtocounter{lem}{-1}
\renewcommand{\thelem}{\thesection.\arabic{lem}'}%
\begin{lem}\label{lem2_7'}

Let $0< \nu \le 1$, $-1\le a<b\le 1$, $c\in \mathbb{R}^3$, $k\ge 0$ be an integer, assume $P_c\ge \delta>0$ on $[a,b]$, $f_\nu\in C^k[a,b]$ is a solution to (\ref{eq_1}), $f_\nu(b)<0$, and there exists some $M>0$ such that $f_{\nu}\ge -M$ in $[a,b]$. % satisfies (\ref{eqcor2_2_1}).
      Then $f_{\nu}<0$ in $[a,b]$. Moreover,  for any $0<\epsilon<\frac{b-a}{2}$, there exists some $C>0$, depending only on $\epsilon$, $\delta$, $k$ and an upper bound of $|c|$, such that
   \begin{equation*}%\label{eq2_7_1'}
       ||f_{\nu}+\sqrt{2P_c}||_{C^k(a+\epsilon,b-\epsilon)}\le C\nu.
   \end{equation*}
   \end{lem}
\renewcommand{\thelem}{\thesection.\arabic{lem}}%

\noindent\emph{Proof of Theorem \ref{thm1_0_1}}:

Writing $U_{\nu_k, \theta}=u_{\nu_k, \theta}\sin\theta$, $V_{\theta}=v_{\theta}\sin\theta$ and $x=\cos\theta$ as usual.  Then $U_{\nu_k, \theta}$ satisfies the equation 
\begin{equation}\label{eq1_0_1_0}
   \nu_k (1-x^2)U_{\nu_k, \theta}'+2\nu_k x U_{\nu_k, \theta}+\frac{1}{2}U_{\nu_k, \theta}^2=P_{c_k}(x), \textrm{ in }(r,s), 
\end{equation}
and $U_{\nu_k, \theta} \rightharpoonup V_{\theta}$ in $L^2(r,s)$, where $r=\cos\theta_2$ and $s=\cos\theta_1$. We know that $-1<r<s<1$. 
For any $\epsilon>0$ and any positive integer $m$, let $C$ denote some positive constant depending only on $\theta_1,\theta_2, \epsilon$, $m$ and $\sup_{\nu_k}||u_{\nu_k, \theta}||_{L^2(\mathbb{S}^2\cap\{\theta_1<\theta<\theta_2\})}$ whose value may vary from line to line.

As in the proof of Lemma \ref{lem2_0}, with $\delta:=(s-r)/9$, for each $k$ there exist $a_k\in [r,r+\delta]$, and $b_{ki}\in [r+2i\delta, r+(2i+1)\delta]$, $i=1,2,3$, such that 
\[
   |U_{\nu_k, \theta}(a_k)|+\sum_{i=1}^3|U_{\nu_k, \theta}(b_{ki})|\le C.
\]
Arguing as in the proof of Lemma \ref{lem2_0}, we have, for $i=1,2,3$, that $\left|\int_{a_k}^{b_{ki}}P_{c_k}\right|\le C$, and in turn
\begin{equation}\label{eq1_0_1_1}
    |c_k|\le C.
\end{equation}
Passing to a subsequence,  $c_k\to c$, we have 
   $\frac{1}{2}V_{\theta}^2=P_c \textrm{ on }(r,s)$. 
%By the assumption, $V_{\theta}$ is smooth on $(r,s)$. Thus
Since $c\in \mathring{J}_0$, we have $P_c>0$ on $(r,s)$. So there exists some $\delta>0$, such that for large $k$, %Then by the positivity of $P_c$ on $(r,s)$ and the boundedness of $\{|c_k|\}$, we have 
\begin{equation}\label{eq1_0_1_2}
   P_{c_k}\ge 1/C \textrm{ on }[r+\delta/8, s-\delta/8].
   \end{equation}
Next, since $\int_{r}^{s}|U_{\nu_k, \theta}|^2\le C$, there exists some $a_k\in [r+\epsilon/16, r+\epsilon/4]$ and $b_k\in  [s-\epsilon/4, s-\epsilon/8]$, such that  $|U_{\nu_k, \theta}(a_k)|+|U_{\nu_k, \theta}(b_k)|\le C$. If $|U_{\nu_k, \theta}(\alpha_k)|=\max_{[a_k,b_k]}|U_{\nu_k, \theta}|>\max\{|U_{\nu_k, \theta}(a_k)|, |U_{\nu_k, \theta}(b_k)|\}$ for some $\alpha_k\in (a_k,b_k)$, then $U_{\nu_k, \theta}'(\alpha_k)=0$ and, by (\ref{eq1_0_1_0}) and (\ref{eq1_0_1_1}), we have 
\[
   \frac{1}{2}U^2_{\nu_k, \theta}(\alpha_k)\le |P_{c_k}(\alpha_k)|+|2\nu_k\alpha_k U_{\nu_k, \theta}(\alpha_k)|\le C+C|U_{\nu_k, \theta}(\alpha_k)|.
\]
It follows that $|U_{\nu_k, \theta}(\alpha_k)|\le C$. Hence 
\begin{equation}\label{eq1_0_1_3}
    |U_{\nu_k, \theta}|\le C \textrm{ on }[r+\delta/4, s-\delta/4].
\end{equation}
We know that either $V_{\theta}=\sqrt{2P_c}$ on $(r,s)$ or $V_{\theta}=-\sqrt{2P_c}$ on $(r,s)$. 

\noindent \textbf{Claim}: If $V_{\theta}=\sqrt{2P_c}$, then $U_{\nu_k, \theta}\ge 1/C$ on $[r+\epsilon/4, s-\epsilon/4]$. If $V_{\theta}=-\sqrt{2P_c}$, then $U_{\nu_k, \theta}\le -1/C$ on $[r+\epsilon/4, s-\epsilon/4]$.

\noindent\emph{Proof of the Claim}: We only treat the case when $V_{\theta}=\sqrt{2P_c}$. The other case can be treated similarly. 
 By (\ref{eq1_0_1_1}),  (\ref{eq1_0_1_2}),  and the weak convergence of $U_{\nu_k, \theta}$ to $V_{\theta}$, we have 
\[
   \int_{r+\epsilon/8}^{r+\epsilon/4} U_{\nu_k, \theta}V_{\theta}\to \int_{r+\epsilon/8}^{r+\epsilon/4}V_{\theta}^2=\int_{r+\epsilon/8}^{r+\epsilon/4}2P_c\ge 1/C.
\]
So there exists some $a_k\in [r+\epsilon/8, r+\epsilon/4]$, such that $U_{\nu_k, \theta}(a_k)\ge 1/C$. Applying Lemma \ref{lem2_2} on $[a_k, s-\epsilon/8]$, we have $U_{\nu_k, \theta}\ge 1/C$ on $[a_k, s-\epsilon/8]$. Thus $U_{\nu_k, \theta}\ge 1/C$ on $[r+\epsilon/4, s-\epsilon/4]$. Note that if $V_{\theta}=-\sqrt{2P_c}$, we will argue similarly and use Lemma \ref{lem2_2'} instead of Lemma \ref{lem2_2}. The claim is proved.

By the claim and (\ref{eq1_0_1_3}), we either have $1/C\le U_{\nu_k, \theta}\le C$ on $[r+\epsilon/4, s-\epsilon/4]$, or $-C\le U_{\nu_k, \theta}\le -1/C$ on $[r+\epsilon/4, s-\epsilon/4]$. We can, in view of (\ref{eq1_0_1_2}), apply Lemma \ref{lem2_7} and Lemma \ref{lem2_7'}, to obtain 
\[
   || U_{\nu_k, \theta}-V_{\theta}||_{C^m([r+\epsilon, s-\epsilon])}\le C\nu_k.
\]
Notice $x=\cos\theta$, $u_{\nu_k, \theta}=U_{\nu_k, \theta}/\sin\theta$ and $v_{\theta}=V_{\theta}/\sin\theta$, we have proved that 
\[
   || u_{\nu_k, \theta}-v_{\theta}||_{C^m(\mathbb{S}^2\cap \{\theta_1+\epsilon<\theta<\theta_2-\epsilon\})}\le C\nu_k.
\]
The conclusion of the theorem then follows from the above, in view of formulas (\ref{eqNS_1}) and (\ref{eqEuler_1}).
%\[
%\begin{split}
%   & u_{\nu_k, r} =- \frac{d u_{\nu_k, \theta}}{d \theta} -\ctthe u_{\nu_k, \theta}, u_{\nu_k,\phi}=0, \\
%   & 2p_k=-\frac{d^2 u_{\nu_k, r}}{d\theta^2} - (\ctthe - u_{\nu_k, \theta}) \frac{d u_{\nu_k, r}}{d\theta} - u_{\nu_k, r}^2 -u_{\nu_k, \theta}^2,\\
%   &v_r=- \frac{d v_{\theta}}{d \theta} -\ctthe v_{\theta}, v_{\phi}=0, \\
%    & q=v_{\theta}\frac{d v_r}{d\theta}-v_r^2-v_{\theta}^2
%   \end{split}
%   \]
The theorem is proved.
\qed

%Consider sequences $\{\nu_k\}$ and $\{c_k\}$ satisfying the following condition
%\begin{equation}\label{eq5_1}
%     c_k\in J_{\nu_k}, c_k\to c\in J_0, \textrm{ and }\nu_k\to 0 \textrm{ as } k \to \infty.
%\end{equation}

\bigskip

\section{$c\in \mathring{J}_0$}\label{sec3}

\noindent{\textit{Proof of Theorem \ref{thm1_1}}}:

We only need to prove the theorem in the special case that $c_k\to c\ne 0$ and $\nu_k\to 0$, where $c_1,c_2>0,c_3>c_3^*(c_1,c_2)$, which is equivalent  to $\min_{[-1,1]}P_c>0$. Indeed, let  $\hat{\nu}_k=nu_k/\sqrt{|c_k|}$. By the assumption $\hat{\nu}_k\to 0$, $\hat{c}_k\to \hat{c}\ne 0$. It is easy to see that $U^{+}_{\theta, \hat{\nu}_k}(\hat{c}_k)=U^{+}_{\nu_k, \theta}(c_k)/\sqrt{|c_k|}$.
   %\begin{equation*}%\label{eqA_2}
%      U^{+}_{\theta, \hat{\nu}_k}(\hat{c}_k)=\frac{U^{+}_{\nu_k, \theta}(c_k)}{\sqrt{|c_k|}}.
%   \end{equation*}
   The desired estimate (\ref{eq1_1}) for $U^{+}_{\nu_k, \theta}(c_k)$ can be easily deduced from the estimate of $U^{+}_{\theta, \hat{\nu}_k}(\hat{c}_k)$.

We prove the estimates in (\ref{eq1_1}) for $\{U_{\nu_k, \theta}^+\}$, the proof for $\{U_{\nu_k, \theta}^-\}$ is similar. 
In the following, $C$ denotes various constant depending only on $c$. %which is independent of $k$.
By Lemma \ref{lem2_1}, $||U_{\nu_k, \theta}^+||_{L^{\infty}(-1,1)}\le C$ for all $k$. By Theorem A, the convergence of $\{c_k\}$ to $c$ and the fact that $\min_{[-1,1]}P_c>0$ and $c_1,c_2>0$, we have, for large $k$,
   $\min_{[-1,1]}P_{c_k}\ge \frac{1}{2}\min_{[-1,1]}P_c>0$,
   $U^+_{\nu_k, \theta}(-1)=\tau_2(\nu_k, (c_k)_1)\ge \sqrt{2(c_k)_1}\ge \sqrt{c_1}>0$, and
   $\left|U^+_{\nu_k, \theta}(\pm 1)-\sqrt{2P_{c_k}(\pm 1)}\right|\le C\nu_k$.
An application of Lemma \ref{lem2_2} gives, for large $k$, that
   $U^+_{\nu_k, \theta}\ge 1/C \textrm{ on }[-1,1]$.
 The first estimate (\ref{eq1_1}) then follows from Corollary \ref{cor2_2}.
%Since $c_k\to c$ and $\nu_k\to 0^+$, by Lemma \ref{lem2_1}, there exists some constant $M$ depending only on an upper bound of $|c|$, such that for sufficiently large $k$, $|U^+_{\nu_k, \theta}|\le M$. Since $c_k\to c$, for $k$ sufficiently large, there exists some positive constant $\delta$, depending only on  $|c_1|$, such that  $U^+_{\nu_k, \theta}(-1)\ge \sqrt{2c_{k,1}}\ge \sqrt{c_1}\ge \sqrt{\delta}>0$ and  $P_{c_k}\ge \delta$ in $(-1,1)$. Then by Lemma \ref{lem2_2},  for all large $k$,
%   \[
%       U^+_{\nu_k, \theta}(x)\ge \min\{\sqrt{\delta}, \delta/4\}, \quad -1<x<1.
%   \]
% By Theorem 1.3 in \cite{LLY2}, $U^{+}_{\theta}(-1)=\tau_2$ and $U_{\theta}^+(1)=\tau_2'$ where $\tau_2,\tau_2'$ are defined by (\ref{eq_2}). So
% \[
%      \left|\frac{1}{2}(U^+_{\nu_k, \theta})^2(\pm 1)-P_{c_k}(\pm 1)\right|\le C\nu_k,
% \]
% for some $C$ depending only on an upper bound of $|c|$. Using Corollary \ref{cor2_2}, we obtain that $U^+_{\nu_k, \theta}$ satisfies  (\ref{eq1_1}).

To prove the second estimate in (\ref{eq1_1}),  we first prove the following lemma.

\begin{lem}\label{lem3_1}
   Let $0<\nu\le 1$, $c\in \mathbb{R}^3$, $U_{\nu, \theta}^+$ be the upper solution of (\ref{eq_1}). If $P_c(-1)=2c_1\ge \delta>0$,
    %$  \min_{[-1,1]}P_c(x)\ge \delta>0$,
then for each non-negative integer $m$, there exists some constant $C$ depending only on $\delta$, $m$ and an upper bound of $|c|$, such that
\[
   \left|\frac{d^m}{dx^m}(U^+_{\nu, \theta}-\sqrt{2P_c})(-1)\right|\le C\nu.
\]
\end{lem}
\begin{proof}
Denote $C$ to be a constant, depending only on $\delta$, $m$ and an upper bound of $|c|$, which may vary from line to line. 
 We first prove that for every $m\ge 0$, %there exists some constant $C$ such that
\begin{equation}\label{eq3_1_1}
   \left|\frac{d^m}{dx^m}U^+_{\nu, \theta}(-1)\right|\le C. % \quad \forall 0<\nu<1.
\end{equation}
It can be checked that $U^+_{\nu, \theta}$ is a solution of (\ref{eq_1}) if and only if $\nu U^+_{\nu, \theta}$ is a solution of (\ref{eqNSE_1}). Then by Lemma 2.3 in \cite{LLY2}, we have that $U^+_{\nu, \theta}\in C^{\infty}[-1,0]$. For $m\ge 0$, differentiating (\ref{eq_1}) $(m+1)$ times and sending $x$ to $-1$ lead to
\[
  \begin{split}
   & \nu\sum_{i=1}^{m+1}{m+1 \choose i}\frac{d^i}{dx^i}(1-x^2)\frac{d^{m+2-i}}{dx^{m+2-i}}U^+_{\nu, \theta}(x)
+2\nu \sum_{i=0}^{1}{m+1 \choose i}\frac{d^i}{dx^i}(x)\frac{d^{m+1-i}}{dx^{m+1-i}}U^+_{\nu, \theta}(x)\\
   &   +\frac{1}{2} \sum_{i=0}^{m+1}{m+1 \choose i}\frac{d^i}{dx^i}U^+_{\nu, \theta}(x)\frac{d^{m+1-i}}{dx^{m+2-i}}U^+_{\nu, \theta}(x)=\frac{d^{m+1}}{dx^{m+1}}P_c(x),\quad \textrm{ at }x=-1.
\end{split}
\]
Notice that the $i=1$ term in the first sum and the $i=0$ term in the second sum cancel out, we rewrite the above equation as
\[
  \begin{split}
     & U^+_{\nu, \theta}(x)\frac{d^{m+1}}{dx^{m+1}}U^+_{\nu, \theta}(x)  =\frac{d^{m+1}}{dx^{m+1}}P_c(x)-\frac{1}{2} \sum_{i=1}^{m}{m+1 \choose i}\frac{d^i}{dx^i}U^+_{\nu, \theta}(x)\frac{d^{m+1-i}}{dx^{m+2-i}}U^+_{\nu, \theta}(x)\\
   & -\nu\sum_{i=2}^{m+1}{m+1 \choose i}\frac{d^i}{dx^i}(1-x^2)\frac{d^{m+2-i}}{dx^{m+2-i}}U^+_{\nu, \theta}(x)
  -2\nu (m+1)\frac{d^{m}}{dx^{m}}U^+_{\nu, \theta}(x), \quad \textrm{ at }x=-1.
\end{split}
\]
Since $2c_1=P_c(-1)\ge \delta>0$ and $U^+_{\nu, \theta}(-1)=2\nu+2\sqrt{\nu^2+c_1}$, we have $1/C\le U^+_{\nu, \theta}(-1)\le C$. 
 Using this and the fact that the right hand side of the above equation involves only $\left\{\frac{d^i}{dx^i}U^+_{\nu, \theta}(-1)\right\}_{0\le i \le m}$, we can easily prove (\ref{eq3_1_1}) by induction. 
By (\ref{eq3_1_1}) and the fact that $1/C\le U^+_{\nu, \theta}(-1)\le C$, take $m-$th derivatives of (\ref{eq_1}), we have that
\begin{equation}\label{eq3_1_2}
   |\frac{d^{m}}{dx^{m}}(\frac{1}{2}(U^+_{\nu, \theta})^2-P_c)(-1)|=\nu|\frac{d^{m}}{dx^{m}}((1-x^2)(U^+_{\nu, \theta})'+2xU^+_{\nu, \theta})|\big|_{x=-1}\le C\nu.
\end{equation}
Since $U^+_{\nu, \theta}-\sqrt{2P_c}=\frac{2}{U^+_{\nu, \theta}+\sqrt{2P_c}}\left[\frac{1}{2}(U^+_{\nu, \theta})^2-P_c\right]$, $1/C\le U^+_{\nu, \theta}(-1)\le C$, and $P_c(-1)\ge 1/C$, the lemma follows from (\ref{eq3_1_1}) and (\ref{eq3_1_2}).
\end{proof}
Now we continue to prove Theorem \ref{thm1_1}. Apply Lemma \ref{lem2_7} with $a=-3/4$ and $b=1$, we have, for all $m\ge 0$,
\[
   ||\frac{d^m}{dx^m}(U_{\nu_k, \theta}^+-\sqrt{2P_{c_k}})||_{L^{\infty}(-\frac{1}{2},1-\epsilon)}\le C\nu_k.
\]
By Lemma \ref{lem3_1} and Lemma \ref{lem2_12} with $a=-1$, $b=-1/2$, we have
\[
   ||\frac{d^m}{dx^m}(U_{\nu_k, \theta}^+-\sqrt{2P_{c_k}})||_{L^{\infty}(-1,-\frac{1}{2})}\le C\nu_k,
\]
for some $C$ depending only on $\delta$, $m$ and an upper bound of $|c|$. Theorem \ref{thm1_1} is proved.
\qed

%%%%%%%%%%%%%%%%%%%
%Next, we study the behavior of solutions $U_{\nu_k, \theta}$ of (\ref{eq:NSE}) other than $U_{\nu_k, \theta}^{\pm}$. We have
%\begin{thm}\label{thm3_2}
%Let $\nu_k\to 0^+$, $c_k\in J_{\nu_k}$, $c_k\to c\in \mathring{J}_0$, $m$ be a positive integer,  and $U_{\nu_k, \theta}\in C^m(-1,1)$ is a solution of (\ref{eq:NSE}) with $\nu=\nu_k$ and $c=c_k$, other than $U_{\nu_k, \theta}^\pm$, then there exists a unique $-1<x_k<1$ such that $U_{\nu_k, \theta}(x_k)=0$. Moreover, for any  $\epsilon>0$, there exits some constant $C$, depending only on $\epsilon, m$, and $c$, such that for large $k$, 
%\begin{equation}\label{eq1_2_1}
%||U_{\nu_k, \theta}+\sqrt{2P_{c_k}}||_{L^{\infty}(-1, x_k-\epsilon)}\le C\nu_k, \quad ||U_{\nu_k, \theta}-\sqrt{2P_{c_k}}||_{L^{\infty}(x_k+\epsilon, 1)}\le C\nu_k.
%\end{equation}
%and 
%\begin{equation}\label{eq1_2_2}
%||U_{\nu_k, \theta}+\sqrt{2P_{c_k}}||_{C^m(-1+\epsilon, x_k-\epsilon)}\le C\nu_k, ||U_{\nu_k, \theta}-\sqrt{2P_{c_k}}||_{C^m(x_k+\epsilon, 1-\epsilon)}\le C\nu_k.
%\end{equation}
%\end{thm}
%%%%%%%%%%%%%%%%%%%

\medskip

\noindent{\textit{Proof of Theorem \ref{thm1_3} Started}}:
%\begin{proof}
In this part, we prove the first paragraph of Theorem \ref{thm1_3} and part (iii).
Let $C$ denote a positive constant, having the same dependence as specified in the theorem, which may vary from line to line. %For convenience write $f_k=U_{\nu_k, \theta}$ and $P_k=P_{c_k}$. Throughout the proof $k$ is large.
 By Lemma \ref{lem2_1},
\begin{equation}\label{eqthm3_3_1}
|U_{\nu_k, \theta}|\le C.
\end{equation}
Since $U_{\nu_k, \theta}$ is not $U_{\nu_k, \theta}^\pm$, we know from Theorem A that $U_{\nu_k, \theta}(-1)=\tau_1(\nu_k, c_{k1})$ and $U_{\nu_k, \theta}(1)=\tau_2'(\nu_k, c_{k2})$. 
Since $c_k\in J_0\setminus\{0\}$, we have $P_c\ge 0$ on $[-1,1]$. By Lemma \ref{lem2_4}, there exists at most one $x_k$ in $(-1,1)$ such that $U_{\nu_k, \theta}(x_k)=0$.

Now we prove part (iii). Since $c\in \mathring{J}_0$, we have $c_1,c_2>0$ and $\min_{[-1,1]}P_c>0$.  By the convergence of $\{c_k\}$ to $c$, we deduce, using (\ref{eq_2}), that
\begin{equation}\label{eqthm3_3_2}
  U_{\nu_k, \theta}(-1)\le -1/C, \quad U_{\nu_k, \theta}(1)\ge 1/C, \quad \min_{[-1,1]}P_{c_k}\ge 1/C,
\end{equation}
and
\begin{equation}\label{eqthm3_3_3}
   |U_{\nu_k, \theta}(-1)+\sqrt{2P_{c_k}(-1)}|+|U_{\nu_k, \theta}(1)-\sqrt{2P_{c_k}(1)}|\le C\nu_k.
\end{equation}
Clearly there exists $x_k\in (-1,1)$ such that $U_{\nu_k, \theta}(x_k)=0$. %So  %$x_k$ must exist. %there exists a unique $x_k\in (-1,1)$, such that $U_{\nu_k, \theta}(x_k)=0$. 
By Lemma \ref{lem2_4} and (\ref{eqthm3_3_2}),
\begin{equation}\label{eqthm3_3_4}
   U_{\nu_k, \theta}(x)<0, -1\le x< x_k,  \textrm{ and }U_{\nu_k, \theta}(x)>0,  x_k<x\le 1.
\end{equation}
%From the above $x_k$ must exist. 
By Corollary \ref{cor2_3} and Corollary \ref{cor2_3'}, using (\ref{eqthm3_3_1}) and (\ref{eqthm3_3_4}), we have that
\begin{equation}\label{eqthm3_3_5}
   -C\le U_{\nu_k, \theta}(x)\le -1/C, x\in (-1,x_k-\epsilon/2),  \textrm{ and }1/C\le U_{\nu_k, \theta}\le C, x\in (x_k+\epsilon/2, 1).
\end{equation}
With (\ref{eqthm3_3_5}) we deduce (\ref{eq1_2_2}) by applying Lemma \ref{lem2_7} and Lemma \ref{lem2_7'}. With (\ref{eqthm3_3_3}), (\ref{eqthm3_3_5}) and (\ref{eq1_2_2}), we deduce (\ref{eq1_2_1}) by applying Corollary \ref{cor2_2} and Corollary \ref{cor2_2'}.
\qed

\section{$c\in \partial J_0\setminus\{0\}$ and $c_3=c_3^*(c_1,c_2)$}\label{sec4}

In this section, we study a sequence of solutions $U_{\nu_k, \theta}$ of (\ref{eq:NSE}) with $c_k\to c$ and $\nu_k\to 0$, where $c\in \partial J_0\setminus\{0\}$ and $c_3=c_3^*(c_1,c_2)$. We first study the behaviors of $U_{\nu_k, \theta}^{\pm}$.

\noindent{\emph{Proof of Theorem \ref{thm1_2}}}:

Let $C$ denote a constant depending only on $c$ which may vary from line to line. We only prove the result for the case $c_3=c_3^*(c_1,c_2)$ and $c_2>0$. The result for the case $c_3=c_3^*(c_1,c_2)$ and $c_1>0$ can be proved similarly. 
%that if $c_3=c_3^*(c_1,c_2)$ and $c_2>0$,  then there exist some sequences $c_k\in J_{\nu_k}$, $c_k\to c$, $\nu_k\to 0$, and some constant $\epsilon>0$, such that $||\frac{1}{2}(U_{\nu_k, \theta}^+)^2-P_{c_k}||_{L^{\infty}(-1,1)}>\epsilon$ for all $k$.  The second half of the theorem can be proved similarly. 

Since $c_3=c_3^*(c_1,c_2)$ and $c_2>0$, we have $c_3<0$ and $P_c=-c_3(x-\bar{x})^2$ with $\bar{x}:=\frac{\sqrt{c_1}-\sqrt{c_2}}{\sqrt{c_1}+\sqrt{c_2}}\in [-1,1)$. Then for any $\epsilon<(1-\bar{x})/8$,  we have 
\begin{equation}\label{eqlem4_1_1}
   P_{c}(x)\ge2P_c(1)/3=4c_2/3>0, \quad 1-2\epsilon\le x\le 1. 
   \end{equation}
Choose sequences  $c_{k1}\to c_1$, $c_{k2}\to c_2$,  let, as in (\ref{eqdef_1}), 
\[
   \bar{c}_{k3}:=\bar{c}_{3}(c_{k1},c_{k2};\nu_k)=-\frac{1}{2}(\sqrt{\nu_k^2+c_{k1}}+\sqrt{\nu_k^2+c_{k2}})(\sqrt{\nu_k^2+c_{k1}}+\sqrt{\nu_k^2+c_{k2}}+2\nu_k).
\]
Let $c_k=(c_{k1},c_{k2},c_{k3})$, where $c_{k3}>\bar{c}_{k3}$ will be chosen later. It is easy to see that $c_k\in J_{\nu_k}$. Let $U_{\nu_k, \theta}^+$ be the solution of (\ref{eq_1}) with the right hand side  $P_{c_k}$. For convenience, write $f_k=U_{\nu_k, \theta}^+$, $P_k=P_{c_k}$, and $h_k:=\frac{1}{2}f^2_k-P_k$. 
Let $\bar{P}_{k}:=P_{(c_{k1},c_{k2},\bar{c}_{k3})}$.  %$\bar{U}_{\nu_k, \theta}^+$
By Theorem A, there exists a unique solution $\bar{f}_k$  of (\ref{eq_1}) with the right hand side $\bar{P}_k$, and
\begin{equation}\label{eqlem4_1_2}
   \bar{f}_k(x)=(\nu_k+\sqrt{\nu_k^2+c_{k1}})(1-x)-(\nu_k+\sqrt{\nu_k^2+c_{k2}})(1+x).
\end{equation}
By Theorem A again, for any integer $i>0$, 
% $\epsilon>0$ and $\delta>0$, 
there exists $\delta_{ik}>0$, $\delta_{ik}\to 0$, such that for  $|c_{k3}-\bar{c}_{k3}|\le \delta_{ik}$, we have
%\begin{equation*}%\label{eq5_2}
  $||f_k-\bar{f}_k||_{L^{\infty}(-1,1-1/i)}<1/i$.
%\end{equation*}
Choose $c_{k3}=\bar{c}_{k3}+\delta_{kk}$. Then  
%\begin{equation*}%\label{eq5_2}
   $||f_k-\bar{f}_k||_{L^{\infty}(-1,1-1/k)}<1/k$.
%\end{equation*}
By computation, for any $\epsilon>0$,
%\begin{equation*}
  $\bar{f}_k(1-\epsilon)=\epsilon \sqrt{c_1}-(2-\epsilon)\sqrt{c_2}+o(1)$,
%\end{equation*}
where $o(1)\to 0$ as $k\to \infty$.
Since $c_2>0$,  we have that for $k>1/\epsilon$,  % for small $\epsilon$ and $\delta$, we have that for all large $k$,
  $f_k(1-\epsilon)<0$.

On the other hand, by Theorem A, using $c_{k3}>\bar{c}_{k3}$, we have $f_k(1)=-2\nu_k+2\sqrt{\nu_k^2+c_{k2}}\ge \sqrt{c_2}>0$ for sufficiently large $k$. So there is some $x_k\in (1-\epsilon, 1)$ such that $f_k(x_k)=0$.
By (\ref{eqlem4_1_1}), $P_k(x_k)\ge P_c(x_k)+o(1)\ge c_2>0$ for large $k$. 
 So we have $|\frac{1}{2}f^2_k(x_k)-P_k(x_k)|\ge c_2$ for large $k$.  
 Theorem \ref{thm1_2} is proved.
%\end{proof}
\qed

%Similarly we have
%
%\addtocounter{lem}{-1}
%\renewcommand{\thelem}{\thesection.\arabic{lem}'}%
%\begin{lem}\label{lem4_1'}
%If $c_3=c_3^*(c_1,c_2)$ and $c_1>0$,  then there exist some sequences $c_k\in J_{\nu_k}$, $c_k\to c$, $\nu_k\to 0$, and some constant $\epsilon>0$, such that $||\frac{1}{2}(U_{\nu_k, \theta}^-)^2-P_{c_k}||_{L^{\infty}(-1,1)}>\epsilon$ for all $k$. %, where $U_{\nu_k, \theta}^-$ is the lower solution of (\ref{eq_1}) with $\nu=\nu_k$ and $c=c_k$. %, satisfying $f_k(-1)=\tau_2(\nu_k,c_{k1})$.
%\end{lem}
%\renewcommand{\thelem}{\thesection.\arabic{lem}}%

\begin{rmk}
  If $c_k\in J_0$, i.e. $P_{c_k}\ge 0$ on $[-1,1]$, then we have, by Theorem \ref{thm1_2_1},  that $||\frac{1}{2}(U_{\nu_k, \theta}^{\pm}(c_k))^2-P_{c_k}||_{L^{\infty}[-1,1]}\to 0$ as $k\to 0$. So the $\{P_{c_k}\}$ constructed in Theorem \ref{thm1_2} has the property that $\min_{[-1,1]}P_{c_k}<0$ for large $k$. % Actually the latter can be seen from the fact that $\min_{[-1,1]} \bar{P}_k<0$.% as shown in the proof of Lemma \ref{lem2_9}.
  %So in the above proof, for $\delta_k$ small enough, we have $\min_{[-1,1]}P_k=P_k(\bar{x}_k)<0$. In Theorem \ref{thm4_1}, we have assumed $P_k(x)\ge 0$ on $[-1,1]$, so there is no sequence $c_{k3}$ such that $|c_{k3}-\bar{c}_{k3}|<\delta_k$. The counter example does not exist there.
\end{rmk}

\noindent{\textit{Proof of Theorem \ref{thm1_2_1}}}:

Let $C$ denote a positive constant depending only on $c$ which may vary from line to line. For convenience, write $f_k=U_{\nu_k, \theta}^+(c_k)$, $P_k=P_{c_k}$, and $h_k:=\frac{1}{2}f^2_k-P_k$. In the following we always assume that $k$ is large.

(i) We only prove the results for $U_{\nu_k,\theta}^+$, the proof for $U_{\nu_k,\theta}^-$ is similar. 
Since $P_k\ge 0$ in $[-1,1]$ and $f_k(-1)=\tau_2(\nu_k,c_{k1})>0$, we have,  by Lemma \ref{lem2_1} and Lemma \ref{lem2_4}, 
\begin{equation}\label{eq4_2_9}
     0<f_k(x)\le C, \quad -1\le x<1.
\end{equation}
By (\ref{eqP_1}), $P_c(x)=-c_3(x-\bar{x})^2$ with $\bar{x}=\frac{\sqrt{c_1}-\sqrt{c_2}}{\sqrt{c_1}+\sqrt{c_2}}\in [-1,1]$. %when $c_3=c_3^*(c_1,c_2)$, there exists some $\bar{x}\in [-1,1]$, such that
%\[
%   P_c(x)=-c_3(x-\bar{x})^2.
%\]
Since $c_k\to c$ and $c_3<0$, we know $c_{k3}<\frac{1}{2}c_3<0$ for large $k$. Let $\bar{x}_k$ be the unique minimum point of $P_k$, then $\bar{x}_k\to \bar{x}$, %and there exist a sequence $\bar{x}_k\to \bar{x}$, such that
\begin{equation*}%\label{eq4_2_2}
   P_k(x)=P_k(\bar{x}_k)-c_{k3}(x-\bar{x}_k)^2,
\end{equation*}
and for large $k$, that
\begin{equation}\label{eq4_2_3}
   \frac{|c_3|}{2}(x-\bar{x}_k)^2\le P_k(x)-P_k(\bar{x}_k)\le 2|c_3|(x-\bar{x}_k)^2, \quad -1\le x\le 1.
\end{equation}

We first prove 
\begin{equation}\label{eq4_1_1}
   ||\frac{1}{2}(U^{+}_{\nu_k, \theta})^2-P_{c_k}||_{L^{\infty}(-1,1)}\le C\nu_k^{2/3} .
\end{equation}

   \textbf{Case 1}: $c_1,c_2>0$. %, $c_3=c_3^*(c_1,c_2)$.

   In this case $\bar{x}\in (-1,1)$. Let $a_k=\nu_k^{1/3}/\alpha$ for some positive $k$-independent constant $\alpha$ to be determined. 
 By Lemma \ref{lem2_5}, there exists some $x_{k}\in (\bar{x}_k+a_k,\bar{x}_k+2a_k)$, such that $|h_k(x_{k})|\le C\nu_k/a_k=C\alpha^3 a_k^2$. It follows from (\ref{eq4_2_3}), using the fact that $P_k(\bar{x}_k)\ge 0$,  that
 \begin{equation}\label{eq4_2_4}
    P_c(x)\ge |c_3|a_k^2/2, \quad \forall \bar{x}_k+a_k\le x\le 1.
 \end{equation}
 Thus $f^2_k(x_k)/2\ge P_{k}(x_k)-|h_k(x_k)|\ge \left(|c_3|/2-C\alpha^3\right)a_k^2$. 
 Fix $\alpha^3=|c_3|/(4C)$.  By (\ref{eq4_2_9}) we have $f_{k}(x_{k})\ge a_k/C$.
 Applying  Lemma \ref{lem2_2} on $[x_k,1]$, using (\ref{eq4_2_4}), we have $f_k\ge  a_k/C$ on $[x_{k}, 1]$.
   Since $|h_{k}(x_{k})|\le C\nu_k^{2/3}$ and $|h_k(1)|=|\frac{1}{2}f^2_k(1)-P_k(1)|=|\frac{1}{2}(\tau_2'(\nu,c_{k2}))^2-2c_2|\le C\nu_k$, we have, by applying Lemma \ref{lem2_6} on $[x_k,1]$, that
       \[
          \max_{[\bar{x}_k+2a_k,1]}|h_k|\le \max_{[x_k,1]}|h_k|\le C\nu_k^{2/3}.% \quad x_{\nu}<x<1.
       \]
       Similarly, by Lemma \ref{lem2_5}, there exists some $x'_k\in [\bar{x}_k-2a_k, \bar{x}_k-a_k]$, such that $|h_k(x'_k)|\le C\nu_k^{2/3}$.
       We also have $|h_k(-1)|=|\frac{1}{2}f^2_k(-1)-P_k(-1)|=|\frac{1}{2}(\tau_2(\nu,c_{k1}))^2-2c_1|\le C\nu_k$. Similar to (\ref{eq4_2_4}), we have $P_k(x)\ge a_k/C^2$ for $x\in [-1,x_k']$. Recall that $f_k(-1)\ge \sqrt{c_1}$. Using Lemma \ref{lem2_2} we have $f_k\ge a_k/C$ on $[-1,x_k']$. 
       Then by similar arguments as on $[\bar{x}_k+2a_k, 1]$, we have
       \[
           \max_{[-1, \bar{x}_k-2a_k]}|h_k|\le C\nu_k^{2/3}.
       \]

       Now we have that $|h_k(\bar{x}_k-2a_k)|\le Ca_k^2$ and $|h_k(\bar{x}_k+2a_k)|\le Ca_k^2$. Notice that $f_k>0$ on $(-1,1)$, $P_k\ge 0$ in $(-1,1)$, 
   using (\ref{eq4_2_3}) we have $P_k(x)\le P_k(\bar{x}_k)+ 2|c_3|(x-\bar{x}_k)^2$ on $[-1,1]$. Applying Lemma \ref{lem2_10} on $[\bar{x}_k-2a_k, \bar{x}_k+2a_k]$ with $\alpha=2$ and $\bar{x}=\bar{x}_k$, there we have that
       \[
          \max_{[\bar{x}_k-2a_k,\bar{x}_k+2a_k]}|h_k|\le Ca_k^2+C\nu_k/a_k\le C\nu_k^{2/3}.
       \]
       Estimate (\ref{eq4_1_1}) is proved in this case.

      \textbf{ Case 2}: $c_1=0$ and $c_2>0$. %, $c_3=c_3^*(c_1,c_2)$.

       In this case $P_c(x)=\frac{1}{2}c_2(x+1)^2$, $\bar{x}_k\to -1$. Let $a_k=\nu_k^{1/3}/\alpha$ for some $\alpha>0$ to be determined.  Let $b_k=\max\{-1,\bar{x}_k-2a_k\}$ and $d_k=\max\{-1+2a_k, \bar{x}_k+2a_k\}$. It is clear that $-1\le b_k<d_k<1$. We prove estimate (\ref{eq4_1_1}) separately on $[d_k,1]$, $[-1,b_k]$ and $[b_k,d_k]$.

        We first prove the estimate on $[d_k,1]$. Since $d_k-2a_k\ge \bar{x}_k$, we have $x-\bar{x}_k\ge a_k$ for $x$ in $[d_k-a_k, d_k]$. By (\ref{eq4_2_3}), we have $P_k\ge \frac{1}{2}|c_3|a_k^2$ in $[d_k-a_k,d_k]$. We also have $[d_k-a_k,d_k]\subset [-1,1]$. Applying Lemma \ref{lem2_5} on $[d_k-a_k,d_k]$, using (\ref{eq4_2_9}), there exists some $x_k\in [d_k-a_k,d_k]$, such that  $|h_k(x_{k})|\le C\nu_k/a_k=C\alpha^3 a_k^2$.
 Thus
           $\frac{1}{2}f^2_k(x_k)\ge P_{k}(x_k)-|h_k(x_k)|\ge \left(\frac{1}{2}|c_3|-C\alpha^3\right)a_k^2$.
 Fix $\alpha^3=|c_3|/(4C)$. By (\ref{eq4_2_9}) we have $f_{k}(x_{k})\ge a_k/C$.  Applying Lemma \ref{lem2_2} on $[x_{k}, 1]$, and using $P_k\ge \frac{1}{2}|c_3|a_k^2$ on the interval, we have $f_k(x)\ge a_k/C$ on $[x_{k}, 1]$. 
Since $|h_{k}(x_{k})|\le C a_k^2$ and $|h_k(1)|=|\frac{1}{2}f^2_k(1)-P_k(1)|=|\frac{1}{2}(\tau_2'(\nu,c_{k2}))^2-2c_2|\le C\nu_k$, notice $x_k\le d_k$, we have, by applying Lemma  \ref{lem2_6} on $[x_{k}, 1]$, that
       \begin{equation}\label{eq4_2_5}
           \max_{[d_k, 1]}|h_k|\le \max_{[x_k,1]}|h_k|\le C\nu_k^{2/3}.
       \end{equation}

       Next, we prove the estimate on $[-1,b_k]$. If $\bar{x}_k-2a_k>-1$, by (\ref{eq4_2_3}) we have $P_k\ge \frac{1}{2}|c_3|a_k^2$ on $[-1,\bar{x}_k-a_k]$. In particular, $2c_{k1}=P_k(-1)\ge \frac{1}{2}|c_3|a_k^2$. So $c_{k1}\ge \frac{1}{4}|c_3|a_k^2$, and consequently $f_k(-1)=\tau_2(\nu_k,c_{k1})\ge \sqrt{c_{k1}}\ge \frac{1}{2}\sqrt{|c_3|}a_k$. Applying Lemma \ref{lem2_2} on $[-1,\bar{x}_k-a_k]$, and using $P_k\ge \frac{1}{2}|c_3|a_k^2$ on the interval, we have $f_k(x)\ge a_k/C$ on $[-1,\bar{x}_k-a_k]$. 
       Applying Lemma \ref{lem2_5} on $[\bar{x}_k-2a_k, \bar{x}_k-a_k]$, using (\ref{eq4_2_9}), there exists some $y_k\in [\bar{x}_k-2a_k, \bar{x}_k-a_k]$, such that  $|h_k(y_{k})|\le C\nu_k/a_k=C\alpha^3 a_k^2$.
       We also have $|h_k(-1)|\le  C\nu_k$. Notice $b_k=\bar{x}_k-2a_k\le y_k\le \bar{x}_k-a_k$, applying Lemma  \ref{lem2_6} on $[-1,y_k]$, we have that
       \begin{equation}\label{eq4_2_7}
          \max_{[-1,b_k]}|h_k|\le \max_{[-1,y_k]}|h_k|\le C\nu_k^{2/3}.
       \end{equation}
     If $\bar{x}_k-2a_k\le -1$, $b_k=-1$, $\max_{[-1,b_k]}|h_k|=|h_k(-1)|\le C\nu_k\le C\nu_k^{\frac{2}{3}}$.

      Now we prove the estimate on $[b_k,d_k]$. We have proved in the above that $|h_{k}(b_k)|\le Ca_k^2$ and $|h_{k}(d_k)|\le Ca_k^2$. 
      If $\bar{x}_k<-1-a_k$, then $[b_k,d_k]=[-1,-1+2a_k]$, and for any $x\in [-1,-1+2a_k]$, $x-\bar{x}_k\ge -1-\bar{x}_k>a_k$. By (\ref{eq4_2_3}), we have $P_k\ge a^2_k/C$ on $[-1,-1+2a_k]$. In particular, $2c_{k1}=P_k(-1)\ge \frac{1}{2}|c_3|a_k^2$. So $c_{k1}\ge \frac{1}{4}|c_3|a_k^2$, and consequently $f_k(-1)=\tau_2(\nu_k,c_{k1})\ge \sqrt{c_{k1}}\ge \frac{1}{2}\sqrt{|c_3|}a_k$.
        Applying Lemma \ref{lem2_2} on $[-1,-1+2a_k]$, we have $f_k\ge a_k/C$ on $[-1,-1+2a_k]$. Notice we also know $|h_{k}(-1)|\le Ca_k^2$ and $|h_{k}(-1+2a_k)|\le Ca_k^2$. Applying Lemma \ref{lem2_6} on $[b_k,d_k]=[-1,-1+2a_k]$, we have that in this case
         \begin{equation}\label{eq4_2_12}
            \max_{[b_k,d_k]}|h_k|\le Ca_k^2. 
         \end{equation}

      Next, we consider the case $\bar{x}_k\ge -1-a_k$. 
  If $\bar{x}_k-2a_k\ge -1$, then $[b_k,d_k]=[\bar{x}_k-2a_k,\bar{x}_k+2a_k]$. If $\bar{x}_k-2a_k<-1$, then $[b_k,d_k]=[-1,-1+2a_k]$ when $\bar{x}_k<-1$, and $[b_k,d_k]=[-1,\bar{x}_k+2a_k]$ when $\bar{x}_k\ge -1$. So we have $\mathrm{dist}(\bar{x}_k, [b_k,d_k])\le Ca_k$, and $2a_k\le d_k-b_k\le 4a_k$.
     Notice that $|h_{k}(b_k)|\le Ca_k^2$, $|h_{k}(d_k)|\le Ca_k^2$,
     $f_k>0$ on $(-1,1)$ and $P_k\ge 0$ in $(-1,1)$,  and using (\ref{eq4_2_3}), $P_k(x)\le P_k(\bar{x}_k)+ 2|c_3|(x-\bar{x}_k)^2$ on $[-1,1]$. Applying Lemma \ref{lem2_10} on $[b_k,d_k]$ with $\alpha=2$, we have that
       \begin{equation}\label{eq4_2_8}
          \max_{[b_k,d_k]}|h_k|\le Ca_k^2+C\nu_k/a_k\le C\nu_k^{2/3}.
       \end{equation}
       By (\ref{eq4_2_5}), (\ref{eq4_2_7}), (\ref{eq4_2_12}) and (\ref{eq4_2_8}), we have $\max_{[-1,1]}|h_k|\le C\nu_k^{2/3}$. So estimate (\ref{eq4_1_1}) is proved in Case 2.

     %   ************************

       \textbf{Case 3}: $c_2=0,c_1>0$. The proof of (\ref{eq4_1_1}) is similar to that of Case 2.

       We have by now proved (\ref{eq4_1_1}). By (\ref{eq4_1_1}), we have $\lim_{k\to \infty}|||f_k|-\sqrt{2P_{c}}||_{L^{\infty}(-1,1)}=0$. Using this and (\ref{eq4_2_9}), we have $\lim_{k\to \infty}||U^{+}_{\nu_k, \theta}-\sqrt{2P_{c}}||_{L^{\infty}(-1,1)}=0$.
       Next, we prove %(\ref{eq4_1_3}). 
       \begin{equation}\label{eq4_1_3}
 ||U^{+}_{\nu_k, \theta}-\sqrt{2P_{c_k}}||_{C^{m}([-1,1-\epsilon]\setminus[\bar{x}-\epsilon, \bar{x}+\epsilon])}\le C\nu_k,
  \end{equation}
       If $\bar{x}=-1$, then  by (\ref{eq4_2_3}) and the fact that $\bar{x}_k\to \bar{x}$, we have $P_k\ge \epsilon^2/C$ on $[-1+\epsilon/2, 1-\epsilon/2]$ for large $k$. Applying Lemma \ref{lem2_7} on $[-1+\epsilon/2, 1-\epsilon/2]$, using (\ref{eq4_2_9}), we have (\ref{eq4_1_3}) in this case.
       
       If $\bar{x}>-1$, without loss of generality we assume $\epsilon$ is small such that $\bar{x}-2\epsilon>-1$. In this case, by (\ref{eq4_2_3}) and the fact that $\bar{x}_k\to \bar{x}$, we have $P_k\ge \epsilon^2/C$ on $[-1, 1-\epsilon/2]\setminus [\bar{x}-\epsilon/2, \bar{x}+\epsilon/2]$ for large $k$. Applying Lemma \ref{lem2_7} on $[-1+\epsilon/2, \bar{x}-\epsilon/2]$ and $[\bar{x}+\epsilon/2, 1-\epsilon/2]$ separately, we have
       \begin{equation}\label{eq4_2_10}
          ||f_k-\sqrt{2P_k}||_{C^{m}([-1+\epsilon,\bar{x}-\epsilon]\cup [\bar{x}+\epsilon, 1-\epsilon])}\le C\nu_k,
       \end{equation}
       for some constant $C$ depending only on $\epsilon$, $m$ and an upper bound of $|c|$.     
 By Lemma \ref{lem3_1} and (\ref{eq4_2_10}), we have
   $ \left|\frac{d^i}{dx^i}(f_k-\sqrt{2P_k})(-1)\right|\le C\nu_k$ and $\left|\frac{d^i}{dx^i}(f_k-\sqrt{2P_k})(\bar{x}-\epsilon)\right|\le C\nu_k$, $0\le i\le m$,
       where $C$ depending only on $m$ and an upper bound of $|c|$.
%By (\ref{eq4_2_10}) we have
    %      $\left|\frac{d^i}{dx^i}(f_k-\sqrt{2P_k})(\bar{x}-\epsilon)\right|\le C\nu_k$, for $0\le i\le m$.
       Applying Lemma \ref{lem2_12} on $[-1,\bar{x}-\epsilon]$, using (\ref{eq4_2_9}), we have
       \begin{equation}\label{eq4_2_11}
           ||f_k-\sqrt{2P_k}||_{C^{m}([-1,\bar{x}-\epsilon])}\le C\nu_k.
       \end{equation}
Estimate (\ref{eq4_1_3}) in this case follows from  (\ref{eq4_2_10}) and (\ref{eq4_2_11}).

      % Now we prove (\ref{eq4_1_4}). By (\ref{eq4_1_1}) and (\ref{eq_1}),  f
      Next, for any $\epsilon>0$, there exists some constant $C>0$, depending only on $\epsilon$ and an upper bound of $|c|$, such that $|f'_k|\le C\nu_k^{-1/3}$ on $[-1+\epsilon, 1-\epsilon]$, so $|h'_k|=|f_kf'_k-P'_k|\le C\nu_k^{-1/3}$. %So (\ref{eq4_1_4}) is proved.
 By interpolation %Lemma \ref{lem6_1},
        for any $x, y\in (-1+\epsilon, 1-\epsilon)$ and $0<\beta<1$,
       \[
          \frac{|h_k(x)-h_k(y)|}{|x-y|^{\beta}}\le 2||h_||^{1-\beta}_{L^{\infty}(-1+\epsilon, 1-\epsilon)}||h_k'||^{\beta}_{L^{\infty}(-1+\epsilon, 1-\epsilon)}\le C\nu_k^{2(1-\beta)/3}\nu_k^{-\beta/3}\le C\nu_k^{\frac{2}{3}-\beta}.
       \]
        So we have %(\ref{eq4_1_2}) is proved.
\begin{equation}\label{eq4_1_2}
  ||\frac{1}{2}(U^{+}_{\nu_k, \theta})^2-P_{c_k}||_{C^{\beta}(-1+\epsilon,1-\epsilon)}\le C\nu_k^{\frac{2}{3}-\beta},
   \end{equation}
       Part (i)  follows from (\ref{eq4_1_1}), (\ref{eq4_1_3}) and (\ref{eq4_1_2}). 
              
       \bigskip
       
       (ii) If $P_k\ge 0$ on $[-1,1]$, then the conclusion of the lemma follows from part(i). So below we assume that $\min_{[-1,1]}P_k(x)<0$. 
Let $\min_{[-1,1]}P_k(x)=P_k(\bar{x}_k)$. Since $P_k(x)=c_{k1}(1-x)+c_{k2}(1+x)+c_{k3}(1-x^2)$, we have $\bar{x}_k=\frac{c_{k2}-c_{k1}}{2c_{k3}}$, and
\begin{equation}\label{eqlem4_2_5}
   P_k(x)=P_k(\bar{x}_k)-c_{k3}(x-\bar{x}_k)^2.
\end{equation}
Then
\begin{equation}\label{eq5_2_2}
   1-\bar{x}_k=\frac{-c_{k2}+2c_{k3}+c_{k1}}{2c_{k3}}\le C(|c_{k2}|+|2c_{k3}+c_{k1}|).
\end{equation}
By Lemma \ref{lem2_9} and the assumption that $\min_{[-1,1]}P_k(x)<0$, we have
\begin{equation}\label{eqlem4_2_4}
   -C\nu_k\le P_k(\bar{x}_k)<0.
\end{equation} 
Let $\bar{P}_{k}:=P_{(c_{k1},c_{k2},\bar{c}_{k3})}$
 and $\bar{f}_k$ be the same as in (\ref{eqlem4_1_2}). %From the expression of $\bar{f}_k$
Denote %By (\ref{eqlem4_1_2}) we have that $\bar{f}_k(\tilde{x})=0$, where
\[
    \tilde{x}_k=\frac{\sqrt{\nu_k^2+c_{k1}}-\sqrt{\nu_k^2+c_{k2}}}{2\nu_k+\sqrt{\nu_k^2+c_{k1}}+\sqrt{\nu_k^2+c_{k2}}}\in (-1,1).
\]
By (\ref{eqlem4_1_2}) we have that
\begin{equation}\label{eqlem4_2_2}
   \bar{f}_k(\tilde{x}_k)=0, \quad \bar{f}_k>0 , -1 \le x< \tilde{x}_k,   \textrm{ and } \bar{f}_k<0,  \tilde{x}_k<x\le  1.
\end{equation}
By computation
\begin{equation}\label{eq5_2_1}
   1-\tilde{x}_k=\frac{2\nu_k+2\sqrt{\nu_k^2+c_{k2}}}{2\nu_k+\sqrt{\nu_k^2+c_{k1}}+\sqrt{\nu_k^2+c_{k2}}}\le C(\nu_k+\sqrt{|c_{k2}|}).
\end{equation}
Since $2c_3=-c_1$, by (\ref{eq5_2_2}) and (\ref{eq5_2_1}) we see that $\bar{x}_k\to 1$ and $\tilde{x}_k\to 1$. 
Notice that $P_k\ge \bar{P}_k$ and $f_k(-1)=\bar{f}_k(-1)>2\nu_k$. By Lemma 2.4 in \cite{LLY2}, we have
\begin{equation}\label{eqlem4_2_3}
    f_k\ge \bar{f}_k, \quad -1<x<1.
\end{equation}
%So $f_k>0$ on $(-1, \tilde{x}_k)$.
Let $y_k=\min\{\bar{x}_k, \tilde{x}_k\}\to 1$. By (\ref{eqlem4_2_2}) and (\ref{eqlem4_2_3}) we have $f_k>0$ for $-1\le x<y_k$.
As in Case 3 in the proof of part (i), we have
\begin{equation}\label{eq5_2_5}
   \lim_{k\to \infty}||f_k-\sqrt{2P_c}||_{L^{\infty}(-1,y_k-2a_k)}=0,
      \end{equation}
    and that there is some $a_k=\nu_k^{1/3}/\alpha$ with $\alpha>0$ fixed, such that
    \begin{equation}\label{eq5_2_6}
       \max_{[-1,y_k-2a_k]}|h_k|\le C\nu_k^{2/3}.
       \end{equation} 
       
By (\ref{eq5_2_1}) and (\ref{eq5_2_2}) and the fact $a_k=\nu_k^{1/3}/\alpha$ we have
\begin{equation}\label{eq5_2_2_1}
   |y_k-2a_k-1|\le C(\sqrt{|c_{k2}|}+|2c_{k3}+c_{k1}|+\nu_k^{1/3}).
\end{equation}
On the interval $[y_k-2a_k, 1]$, by (\ref{eq5_2_2_1}), (\ref{eqlem4_2_5}), (\ref{eqlem4_2_4}) and the fact $\bar{x}_k\in [y_k-2a_k,1]$, we have that for large $k$,
\[
   |P_k(x)|\le %C\nu_k+C(1-\bar{x}_k+2a_k)^2+C(1-\tilde{x}_k+2a_k)^2\le
   C(|c_{k2}|+|2c_{k3}+c_{k1}|^2+\nu_k^{2/3}), \quad y_k-2a_k\le x\le 1, 
\]
and
\[
   P_k(x)\ge -C\nu_k+|c_3|a_k^2/2>0, \quad -1\le x<y_k-2a_k.
\]
%By (\ref{eqlem4_2_4}), there exists some constant $C>0$ such that
 Let $\hat{P}_k(x)=P_{\hat{c}_k}(x):=P_k(x)+C\nu_k(1+x)$. It can be seen that the corresponding $\hat{c}_k$ belongs to $J_{\nu_k}$. We have $\hat{P}_k\ge P_k>0$ for $-1\le x<y_k-2a_k$. By (\ref{eqlem4_2_4}), $\hat{P}_k>0$ for $y_k-2a_k\le x\le 1$. So $\hat{P}_k>0$ on $[-1,1]$. Let $\hat{f}_k$ be the upper solution of (\ref{eq_1}) with the right hand side to be $\hat{P}_k$. Then by part (i), we have $||\frac{1}{2}\hat{f}_k^2-\hat{P}_k||_{L^{\infty}(-1,1)}\le C\nu_k^{2/3}$. Notice that
\[
   |\hat{P}_k|\le C(|c_{k2}|+|2c_{k3}+c_{k1}|^2+\nu_k^{2/3}), \quad y_k-2a_k<x< 1 .
   \]
   So
   \begin{equation}\label{eq5_2_3}
      \hat{f}_k^2<C(|c_{k2}|+|2c_{k3}+c_{k1}|^2+\nu_k^{2/3}), \quad y_k-2a_k<x< 1 .
   \end{equation}
Since $\hat{c}_{k1}=c_{k1}$, we have $f_{k}(-1)=\hat{f}_k(-1)>2\nu_k$. Using this and the fact $\hat{P_k}\ge P_k$, by Lemma 2.4 in \cite{LLY2}, we have $f_k\le \hat{f}_k$ on $(-1,1)$. 
So on the interval $[y_k-2a_k, 1]$, we have $\bar{f}_k\le f_k\le \hat{f}_k$. Using the expression of $\bar{f}_k$, (\ref{eqlem4_1_2}) and (\ref{eq5_2_2_1}), we have
\[
   |\bar{f}_k|\le C(\sqrt{|c_{k2}|}+|2c_{k3}+c_{k1}|+a_k), \quad y_k-2a_k<x< 1 .
\]
%So $|\bar{f}_k|^2\le C(|c_{k2}|+|2c_{k3}+c_{k1}|^2+\nu_k^{\frac{2}{3}})$ for $y_k-2a_k<x< 1$.
By this estimate and (\ref{eq5_2_3}), we have
\begin{equation}\label{eq5_2_4}
    |f_k|\le C(\sqrt{|c_{k2}|}+|2c_{k3}+c_{k1}|+\nu_k^{1/3}), \quad y_k-2a_k<x< 1 .
\end{equation}
So we have
\[
   |\frac{1}{2}f_k^2-P_k|\le C(|c_{k2}|+|2c_{k3}+c_{k1}|^2+\nu_k^{2/3}), \quad y_k-2a_k<x< 1 .
\]
By this and (\ref{eq5_2_6}) we have (\ref{eqthm1_2_1_1}). %(\ref{eqlem4_2_1}).
Moreover, by (\ref{eq5_2_2_1}) and (\ref{eq5_2_4}), we have
\[
  |f_k(x)-\sqrt{2P_c(x)}|\le |f_k|+|c_3||y_k-2a_k-1|\le C(\sqrt{|c_{k2}|}+|2c_{k3}+c_{k1}|+\nu_k^{1/3}), \quad y_k-2a_k<x< 1.
  \]
  By this and (\ref{eq5_2_5}), we have $\lim_{k\to \infty}||U^{+}_{\nu_k, \theta}-\sqrt{2P_{c}}||_{L^{\infty}(-1,1)}=0$.  Part (ii) is proved. %we have (\ref{eqlem4_2_0}).
  
  \bigskip
  
  (iii) The proof is similar as that of part (ii).
  Theorem \ref{thm1_2_1} is proved.
\qed

\bigskip

Now we study sequence of solutions $U_{\nu_k, \theta}$ of (\ref{eq:NSE}) with $\nu_k$ and $c_k$ other than $U_{\nu_k, \theta}^\pm$.

\noindent{\emph{Proof of Theorem \ref{thm1_3} continued}}:

We will prove  Theorem \ref{thm1_3} (i) and (ii) in the case $c_3=c_3^*(c_1,c_2)$. Let $C$ denote a positive constant, having the same dependence as specified in the theorem, which may vary from line to line. For convenience write $f_k=U_{\nu_k, \theta}$, $P_k=P_{c_k}$ and $h_k=\frac{1}{2}f_k^2-P_k$. Throughout the proof $k$ is large. Let $\bar{x}=\frac{\sqrt{c_1}-\sqrt{c_2}}{\sqrt{c_1}+\sqrt{c_2}}$.  %$\bar{x}$ be defined as in the theorem. 
By the assumption, $c_1,c_2\ge 0$, %$c_1$ and $c_2$ are not both zero,
 $c_3=c_3^*(c_1,c_2)=-\frac{1}{2}(c_1+2\sqrt{c_1c_2}+c_2)<0$,  $-1\le \bar{x}\le 1$, and $P_c(x)=-c_3(x-\bar{x})^2$.

Since $c_k\in J_0$, we have $P_k\ge 0$ on $[-1,1]$. %In particular, $P_k(-1)=2c_{k1}\ge 0$ and $P_k(1)=2c_{k2}\ge 0$. So $f_k(-1)=\tau_2(\nu_k,c_{k1})\le 0$ and $f_k(1)=\tau_1'(\nu_k,c_{k2})\ge 0$. So there exists some $x_k\in [-1,1]$ such that $f_k(x_k)=0$.
By Lemma \ref{lem2_4}, there exists at most one $x_k\in (-1,1)$ such that $f_k(x_k)=0$, and if such $x_k$ exists we have
\begin{equation}\label{eq4_3_8}
   f_k(x)<0, -1< x<x_k, \textrm{ and }f_k(x)>0, x_k<x< 1.
\end{equation}
%Since $c_1, c_2\ge 0$ and $c_3=c_3^*(c_1,c_2)$, we have $P_c(x)=-c_3(x-\bar{x})^2$.
Since $c_k\to c$ and $c_3<0$, we know $c_{k3}<\frac{1}{2}c_3<0$ for large $k$. Let $\bar{x}_k$ be the unique minimum point of $P_k$, then $\bar{x}_k\to \bar{x}$,
%\begin{equation*}
   $P_k(x)=P_k(\bar{x}_k)-c_{k3}(x-\bar{x}_k)^2$, 
%\end{equation*}
and for large $k$, 
\begin{equation}\label{eq4_3_0}
   \frac{1}{2}|c_3|(x-\bar{x}_k)^2\le P_k(x)-P_k(\bar{x}_k)\le 2|c_3|(x-\bar{x}_k)^2, \quad -1\le x\le 1.
\end{equation}
By Lemma \ref{lem2_1},
\begin{equation}\label{eq4_3_5}
|f_k|\le C.
\end{equation}
Since $P_c(x)=-c_3(x-\bar{x})^2$, we have, for every $\epsilon>0$, $\min_{[-1,1]\setminus (\bar{x}-\epsilon/2, \bar{x}+\epsilon/2)}P_c>0$. By the convergence of $\{c_k\}$ to $c$, we deduce that
\begin{equation}\label{eq4_3_6}
 \min_{[-1,1]\setminus (\bar{x}-\epsilon/2, \bar{x}+\epsilon/2)}P_k\ge 1/C. %f_k(-1)\le -1/C, \quad f_k(1)\ge 1/C, \quad \min_{[-1,1]\setminus [\bar{x}-\epsilon, \bar{x}+\epsilon]}P_k\ge 1/C,
\end{equation}
Using (\ref{eq4_3_8}) and (\ref{eq4_3_6}), by applying Lemma \ref{lem2_7} and Lemma \ref{lem2_7'} on each interval of $[-1, x_k-\epsilon/2]\setminus(\bar{x}-\epsilon/2, \bar{x}+\epsilon/2)$ and $[x_k+\epsilon/2, 1]\setminus(\bar{x}-\epsilon/2, \bar{x}+\epsilon/2)$ separately, we have 
\begin{equation*}%\label{eq4_3_2}
   ||U_{\nu_k, \theta}+\sqrt{2P_{c_k}}||_{C^m([-1+\epsilon, x_k-\epsilon]\setminus[\bar{x}-\epsilon, \bar{x}+\epsilon])}+||U_{\nu_k, \theta}-\sqrt{2P_{c_k}}||_{C^m([x_k+\epsilon, 1-\epsilon]\setminus[\bar{x}-\epsilon, \bar{x}+\epsilon])}\le C\nu_k,
\end{equation*}

%We first prove (\ref{eq4_3_2}).
 %Using (\ref{eq4_3_8}) and (\ref{eq4_3_6}), we deduce (\ref{eq4_3_2}) by applying Lemma \ref{lem2_7} and Lemma \ref{lem2_7'} on each interval of $[-1, x_k-\epsilon/2]\setminus(\bar{x}-\epsilon/2, \bar{x}+\epsilon/2)$ and $[x_k+\epsilon/2, 1]\setminus(\bar{x}-\epsilon/2, \bar{x}+\epsilon/2)$ separately.

%and
%\begin{equation}\label{eq4_3_7}
%   |f_{k}(-1)+\sqrt{2P_k(-1)}|+|f_{k}(1)-\sqrt{2P_k(1)}|\le C\nu_k.
%\end{equation}
%By Lemma \ref{lem2_4}, there exists a unique $x_k\in (-1,1)$, such that $f_k(x_k)=0$, and, in view of (\ref{eq4_3_6}),
%\begin{equation}\label{eq4_3_8}
%   f_k(x)<0 \textrm{ for }-1\le x<x_k, \quad \textrm{ and }f_k(x)>0 \textrm{ for }x_k<x\le 1.
%\end{equation}
%By Corollary \ref{cor2_3} and Corollary \ref{cor2_3'}, using (\ref{eq4_3_5}) and (\ref{eq4_3_8}), we have that
%\begin{equation}\label{eq4_3_9}
%   -C\le f_k\le -1/C \textrm{ on }\left(-1,x_k-\epsilon/2\right), \quad  \textrm{ and }1/C\le f_k\le C \textrm{ on }\left(x_k+\epsilon/2, 1\right).
%\end{equation}
%With (\ref{eq4_3_9}) we deduce (\ref{eq4_3_2}) by applying Lemma \ref{lem2_7} and Lemma \ref{lem2_7'} on $[-1+\epsilon, x_k-\epsilon]\setminus[\bar{x}-\epsilon, \bar{x}+\epsilon]$ and $[x_k+\epsilon, 1-\epsilon]\setminus[\bar{x}-\epsilon, \bar{x}+\epsilon]$.

%With (\ref{eqthm3_3_3}), (\ref{eqthm3_3_5}) and (\ref{eq1_2_2}), we deduce (\ref{eq1_2_1}) by applying Corollary \ref{cor2_2} and Corollary \ref{cor2_2'}.

\medskip

  Next, we prove 
  \begin{equation}\label{eq4_3_1}
    ||\frac{1}{2}U^2_{\nu_k, \theta}-P_{c_k}||_{L^{\infty}((-1, x_k-\epsilon)\cup (x_k+\epsilon, 1))}\le C\nu_k^{2/3}.
\end{equation}
 Suppose $x_k\to \hat{x}\in [-1,1]$ as $k\to \infty$.  Since $f_k$ is not $U_{\nu_k, \theta}^\pm$, we know from Theorem A that $f_k(-1)=\tau_1(\nu_k, c_{k1})$, $f_k(1)=\tau_2'(\nu_k, c_{k2})$. %\marginpar{Margin Note: We should add this to Theorem A}
  and therefore,  in view of (\ref{eq_2}), we have
\begin{equation}\label{eq4_3_7}
   |f_{k}(-1)+\sqrt{2P_k(-1)}|+|f_{k}(1)-\sqrt{2P_k(1)}|\le C\nu_k.
\end{equation}
   We have $\min_{[-1,1]}P_c=P_c(\bar{x})$. Assume $\min_{[-1,1]}P_k=P_k(\bar{x}_k)$. Then $\bar{x}\in [-1,1]$ and $\bar{x}_k\to \bar{x}$.  We also have $P_k$ satisfy (\ref{eq4_2_3}) for large $k$.

\medskip

   \textbf{Case 1:} $c_1>0$, $c_2>0$, $c_3=c_3^*(c_1,c_2)<0$.

   In this case $\bar{x}\in (-1,1)$. We discuss the cases when $|x_k-\bar{x}_k|\ge \epsilon/4$ and $|x_k-\bar{x}_k|<\epsilon/4$ separately. %If $\hat{x}\ne \bar{x}$, without loss of generality assume $\hat{x}>\bar{x}$. Then for sufficiently large $k$ we have $x_k>\bar{x}_k$, consider $\epsilon$ small such that $x_k-\bar{x}_k>10\epsilon$ for large $k$.

   If $|x_k-\bar{x}_k|\ge  \epsilon/4$, we prove the case when $x_k\ge \bar{x}_k+\epsilon/4$, the other case can be proved similarly. In view of (\ref{eq_2}), we have $f_k(-1)\le -1/C$ and $f_k(1)\ge 1/C$. We first estimate $|h_k|$ on $[x_k+\epsilon,1]$. We have $P_k\ge 1/C$ on $[x_k+\epsilon/2,1]$ for large $k$. By Corollary \ref{cor2_3} and Corollary \ref{cor2_3'}, using (\ref{eq4_3_5}) and (\ref{eq4_3_8}), we have that
   \begin{equation}\label{eq4_3_9}
    1/C\le f_k\le C, x\in \left(x_k+\epsilon/2, 1\right).%-C\le f_k\le -1/C \textrm{ on }\left(-1,x_k-\epsilon/2\right)\setminus [\bar{x}-\epsilon/2, \bar{x}+\epsilon/2], \quad  1/C\le f_k\le C \textrm{ on }\left(x_k+\epsilon/2, 1\right)\setminus [\bar{x}-\epsilon/2, \bar{x}+\epsilon/2].
   \end{equation}
Using (\ref{eq4_3_7}) and (\ref{eq4_3_9}), applying Lemma \ref{lem2_6} on $(x_k+\epsilon/2, 1)$, we have
   \begin{equation}\label{eq4_3_10}
      \max_{[x_k+\epsilon, 1]}|h_k|\le C\nu_k.
      \end{equation}
   Next, we prove  estimate (\ref{eq4_3_1}) on $[-1,x_k-\epsilon]$. The proof is similar to Case 1 in the proof of Theorem \ref{thm1_2_1} (i). %for  $x<\hat{x}-\epsilon/2$ we have $f_k<0$ for large $k$.
 Let $a_k=\nu_k^{1/3}/\alpha$ for some positive constant $\alpha$ to be determined. Since $P_k(\bar{x}_k)\ge 0$, it follows from (\ref{eq4_3_0}) that
 \begin{equation}\label{eq4_3_14}
    P_k(x)\ge |c_3|a_k^2/2, \quad \forall x\in [-1,\bar{x}_k-a_k]\cup [\bar{x}_k+a_k,1].
 \end{equation}
 By Lemma \ref{lem2_5}, there exists some $y_{k}\in (\bar{x}_k-2a_k,\bar{x}_k-a_k)$, $s_k\in (\bar{x}_k+a_k, \bar{x}_k+2a_k)$ and $t_k\in (x_k-\epsilon/8, x_k-\epsilon/16)$, such that $|h_k(y_{k})|+|h_k(s_k)|\le C\nu_k/a_k=C\alpha^3 a_k^2$ and $|h_k(t_k)|\le C\nu_k$. It follows from (\ref{eq4_3_14}) that
       \[
           f^2_k(y_k)/2\ge P_{c_k}(y_k)-|h_k(y_k)|\ge \left(|c_3|/2-C\alpha^3\right)a_k^2.
       \]
 Fix $\alpha^3=\frac{|c_3|}{4C}$.  We have $f_{k}(y_{k})<- a_k/\sqrt{C}$. Similarly we have $f_k(t_k)<-a_k/\sqrt{C}$.   Using (\ref{eq4_3_8}), applying Lemma \ref{lem2_2'} on $[-1,y_k]$ and $[s_k,t_k]$ separately, we have $f_k(x)\le -a_k/C$ on $[-1, y_{k}]$ and $[s_k,t_k]$. Since $|h_{k}(y_{k})|\le C\nu_k^{2/3}$, $|h_k(-1)|\le C\nu_k$, $|h_k(s_k)|\le C\nu_k^{2/3}$ and $|h_k(t_k)|\le C\nu_k$,  applying Lemma \ref{lem2_6} on $[-1, y_k]$ and $[s_k, t_k]$ separately, we have
       \begin{equation}\label{eq4_3_11}
          \max_{[-1, \bar{x}_k-2a_k]}|h_k|\le \max_{[-1, y_k]}|h_k|\le C\nu_k^{2/3},% \quad x_{\nu}<x<1.
       \end{equation}
       and
      \begin{equation}\label{eq4_3_12}
         \max_{[\bar{x}_k+2a_k, x_k-\epsilon]}|h_k|\le \max_{[s_k,t_k]}|h_k| \le  C\nu_k^{2/3}.
      \end{equation}
     % Since $P_k\ge Ca_k^2$ on $(\bar{x}_k+a_k, x_k)$, %by Lemma \ref{lem2_1}, and Corollary \ref{cor2_3'}, we have $-C\le  f_k\le -Ca_k$ on $(\bar{x}_k+a_k, \hat{x}-\epsilon/2)$.
     % By Lemma \ref{lem2_5}, there is some $s_k\in [\bar{x}_k+a_k, \bar{x}_k+2a_k]$ such that $|h_k(s_k)|\le Ca_k^2$,% and there is some $t_k\in (x_k-\epsilon, x_k-\epsilon/2)$, such that $|h_k(t_k)|\le C\nu_k$. Then
%      Using this and (\ref{eq4_3_9}), applying Lemma \ref{lem2_6} on $[s_k,x_k-\epsilon/2]$, we have that
%      \begin{equation}\label{eq4_3_12}
%         \max_{[\bar{x}_k+2a_k, x_k-\epsilon]}|h_k|\le \max_{[s_k,x_k-\epsilon]}|h_k| \le  C\nu_k^{\frac{2}{3}}.
%      \end{equation}
      Now we have that $h_k(\bar{x}_k-2a_k)\le Ca_k^2$ and $h_k(\bar{x}_k+2a_k)\le Ca_k^2$. Notice that $f_k<0$ on $[\bar{x}_k-2a_k,\bar{x}_k+2a_k]$, $P_k(\bar{x}_k)\ge 0$ and $P_k(x)=P_k(\bar{x}_k)-c_{k3}(x-\bar{x}_k)^2$. %$\min_{[\bar{x}_k-2a_k,\bar{x}_k+2a_k]} P_k=P_k(\bar{x}_k)$, $P'_k(\bar{x}_k)=0$.
     Applying Lemma \ref{lem2_10} on $[\bar{x}_k-2a_k,\bar{x}_k+2a_k]$ with $\alpha=2$, we have that
       \begin{equation}\label{eq4_3_13}
          \max_{[\bar{x}_k-2a_k,\bar{x}_k+2a_k]}|h_k|\le Ca_k^2+C\nu_k/a_k\le C\nu_k^{2/3}.
       \end{equation}
       By (\ref{eq4_3_10}), (\ref{eq4_3_11}), (\ref{eq4_3_12}) and (\ref{eq4_3_13}), we have proved (\ref{eq4_3_1}) when $x_k\ge \bar{x}_k+\epsilon/4$.

    Next, %if $\hat{x}=\bar{x}$, then $\bar{x}\in (x_k-\epsilon/2, x_k+\epsilon/2)$,
     if $|x_k-\bar{x}_k|<\epsilon/4$,  similar as (\ref{eq4_3_9}) we have
    \[
    -C\le f_k\le -1/C, x\in \left(-1,x_k-\epsilon/2\right), \textrm{ and } 1/C\le f_k\le C, x\in \left(x_k+\epsilon/2, 1\right).
   \]
    Using this and (\ref{eq4_3_7}), applying Lemma \ref{lem2_6} on $(-1,x_k-\epsilon/2)$ and $(x_k+\epsilon/2, 1)$, (\ref{eq4_3_1}) is proved.

     \medskip

     \textbf{Case 2}: $c_1=0$, $c_2>0$, $c_3=c_3^*(c_1,c_2)=-c_2/2<0$.

     In this case $P_c(x)=\frac{1}{2}c_2(x+1)^2$, $\bar{x}_k\to -1$. we have the estimate (\ref{eq4_3_14}). We discuss the cases when $x_k-\bar{x}_k\ge \epsilon/4$ and $x_k-\bar{x}_k<\epsilon/4$ separately. %If $\hat{x}>-1$, for sufficiently large $k$ we have $x_k>\bar{x}_k$, consider $\epsilon$ small such that $x_k-\bar{x}_k>10\epsilon$ for large $k$.

   If $x_k-\bar{x}_k\ge \epsilon/4$, in view of (\ref{eq_2}), we have  $f_k(1)\ge 1/C$. We first estimate $|h_k|$ on $[x_k+\epsilon,1]$. We have $P_k\ge 1/C$ on $[x_k+\epsilon/2,1]$ for large $k$. By Corollary \ref{cor2_3} and Corollary \ref{cor2_3'}, using (\ref{eq4_3_5}) and (\ref{eq4_3_8}), we have (\ref{eq4_3_9}).
   %\begin{equation}\label{eq4_3_14}
%    1/C\le f_k\le C \textrm{ on }\left(x_k+\epsilon/2, 1\right).
%   \end{equation}
 Using (\ref{eq4_3_7}) and (\ref{eq4_3_9}), applying Lemma \ref{lem2_6} on $(x_k+\epsilon/2, 1)$, we have
   \begin{equation*}%\label{eq4_3_15}
      \max_{[x_k+\epsilon, 1]}|h_k|\le C\nu_k.
      \end{equation*}
      
Next, we prove the estimate (\ref{eq4_3_1}) on $[-1,x_k-\epsilon]$. The proof is similar to Case 2 in the proof of Theorem \ref{thm1_2_1}(ii).
     Let $a_k=\nu_k^{1/3}/\alpha$ for some $\alpha>0$ to be determined, $b_k=\max\{-1,\bar{x}_k-2a_k\}$ and $d_k=\max\{-1+2a_k, \bar{x}_k+2a_k\}$. It is clear that $-1\le b_k<d_k<1$. We prove estimate (\ref{eq4_3_1}) separately on $[d_k,x_k-\epsilon]$, $[-1,b_k]$ and $[b_k,d_k]$.

        We first prove the estimate on  $[d_k,x_k-\epsilon]$. Since $d_k-2a_k\ge \bar{x}_k$, we have $x-\bar{x}_k\ge a_k$ for $x$ in $[d_k-a_k, d_k]$. %By (\ref{eq4_3_14}), we have $P_k\ge \frac{1}{2}|c_3|a_k^2$ in $[d_k-a_k,d_k]$. We also have $[d_k-a_k,d_k]\subset [-1,1]$.
        Applying Lemma \ref{lem2_5} on $[d_k-a_k,d_k]$, using (\ref{eq4_3_5}), there exists some %$y_k\in (b_k-2a_k, b_k-a_k)$,
         $s_k\in (d_k-a_k,d_k)$ and $t_k\in (x_k-\epsilon/8, x_k-\epsilon/16)$, such that  $|h_k(s_{k})|\le C\nu_k/a_k=C\alpha^3 a_k^2$ and $|h_k(t_k)|\le C\nu_k$.
 Thus
       \[
           \frac{1}{2}f^2_k(t_k)\ge P_{k}(t_k)-|h_k(t_k)|\ge \frac{1}{2}|c_3|a_k^2-C\nu_k .%\left(\frac{1}{2}|c_3|-C\alpha^3\right)a_k^2.
       \]
 %Fix $\alpha^3=\frac{|c_3|}{4C}$. %By (\ref{eq4_2_9}), $f_k>0$ on $[-1,1)$, from the above
 By (\ref{eq4_3_8}) and (\ref{eq4_3_14}), we have $f_{k}(t_{k})\le -a_k/C$.  Applying Lemma \ref{lem2_2'} on $[s_{k}, t_k]$, and using $P_k\ge \frac{1}{2}|c_3|a_k^2$ on the interval, we have
 \begin{equation}\label{eq4_3_17}
   f_k(x)\le -a_k/C, \quad s_k\le x\le t_k.
   \end{equation}
 Notice $s_k\le d_k$ and $|h_k(s_{k})|+|h_k(t_k)|\le C\alpha^3 a_k^2$, applying Lemma  \ref{lem2_6} on $[s_{k}, t_k]$, we have
       \begin{equation}\label{eq4_3_16}
           \max_{[d_k, x_k-\epsilon]}|h_k|\le \max_{[s_k,t_k]}|h_k|\le C\nu_k^{2/3}.
       \end{equation}
       %In particular, $|h_{k}(d_k)|\le Ca_k^2$.

       Next, we prove the estimate on $[-1,b_k]$.  If $\bar{x}_k-2a_k>-1$, applying Lemma \ref{lem2_5} on $[b_k,b_k+a_k]$, using (\ref{eq4_3_5}), there exists some $y_k\in (b_k, b_k+a_k)$,
         such that  $|h_k(y_{k})|\le C\nu_k/a_k=C\alpha^3 a_k^2$.
 Thus
       \[
           \frac{1}{2}f^2_k(y_k)\ge P_{k}(y_k)-|h_k(y_k)|\ge\left(\frac{1}{2}|c_3|-C\alpha^3\right)a_k^2.
       \]
 Fix $\alpha^3=|c_3|/(4C)$. %By (\ref{eq4_2_9}), $f_k>0$ on $[-1,1)$, from the above
 By (\ref{eq4_3_8}) and (\ref{eq4_3_14}), we have $f_{k}(y_{k})\le -a_k/C$.  Applying Lemma \ref{lem2_2'} on $[-1,y_k]$, we have $f_k(x)\le -a_k/C$ on $(-1, y_{k})$.  Using $|h_{k}(y_{k})|\le C\nu_k^{2/3}$ and $|h_k(-1)|\le C\nu_k$,  applying Lemma \ref{lem2_6} on $[-1, y_k]$, we have
       \begin{equation}\label{eq4_3_20}
          \max_{[-1, b_k]}|h_k|\le \max_{[-1, y_k]}|h_k|\le C\nu_k^{2/3}. % \quad x_{\nu}<x<1.
       \end{equation}
       %If $\bar{x}_k-2a_k>-1$, by (\ref{eq4_2_3}) we have $P_k\ge \frac{1}{2}|c_3|a_k^2$ on $[-1,\bar{x}_k-a_k]$. In particular, $2c_{k1}=P_k(-1)\ge \frac{1}{2}|c_3|a_k^2$. So $c_{k1}\ge \frac{1}{4}|c_3|a_k^2$, and consequently $f_k(-1)=\tau_1(\nu_k,c_{k1})\le -\sqrt{c_{k1}}\le -\frac{1}{2}\sqrt{|c_3|}a_k$. Applying Lemma \ref{lem2_2} on $[-1,\bar{x}_k-a_k]$, and using $P_k\ge \frac{1}{2}|c_3|a_k^2$ on the interval, we have $f_k(x)\ge a_k/C$ on $[-1,\bar{x}_k-a_k]$.
%
%       Applying Lemma \ref{lem2_5} on $[\bar{x}_k-2a_k, \bar{x}_k-a_k]$, using (\ref{eq4_2_9}), there exists some $y_k\in [\bar{x}_k-2a_k, \bar{x}_k-a_k]$, such that  $|h_k(y_{k})|\le \frac{C\nu_k}{a_k}=C\alpha^3 a_k^2$.
%
%       We also have $|h_k(-1)|\le  C\nu_k$. Notice $b_k=\bar{x}_k-2a_k\le y_k\le \bar{x}_k-a_k$, applying Lemma  \ref{lem2_6} on $[-1,y_k]$, we have that
%       \begin{equation}\label{eq4_2_7}
%          \max_{[-1,b_k]}|h_k|\le \max_{[-1,y_k]}|h_k|\le C\nu_k^{2/3}.
%       \end{equation}
     If $\bar{x}_k-2a_k\le -1$, $b_k=-1$, $\max_{[-1,b_k]}|h_k|=|h_k(-1)|\le C\nu_k\le C\nu_k^{2/3}$.

      Now we prove the estimate on $[b_k,d_k]$. We have proved in the above that $|h_{k}(b_k)|\le Ca_k^2$ and $|h_{k}(d_k)|\le Ca_k^2$.
  If $\bar{x}_k<-1-a_k$, then $[b_k,d_k]=[-1,-1+2a_k]$, and for any $x\in [-1,-1+2a_k]$, $x-\bar{x}_k\ge -1-\bar{x}_k>a_k$. By (\ref{eq4_3_14}), we have $P_k\ge a_k^2/C$ on $[-1,-1+2a_k]$. By (\ref{eq4_3_17}), $f_k(d_k)\le -a_k/C$. %In particular, $2c_{k1}=P_k(-1)\ge \frac{1}{2}|c_3|a_k^2$. So $c_{k1}\ge \frac{1}{4}|c_3|a_k^2$, and consequently $f_k(-1)=\tau_2(\nu_k,c_{k1})\ge \sqrt{c_{k1}}\ge \frac{1}{2}\sqrt{|c_3|}a_k$.
      %By plugging $x=-1$ in (\ref{eq_1}), we have
%      \[
%         -2\nu_kf_k(-1)+\frac{1}{2}f^2_k(-1)=P_c(-1)\ge a_k/C^2.
%         \]
        % So $f_k(-1)\ge a_k/C$.
        Applying Lemma \ref{lem2_2'} on $[-1,-1+2a_k]$, we have $f_k\le -a_k/C$ on $[-1,-1+2a_k]$. Notice we also know $|h_{k}(-1)|\le Ca_k^2$ and $|h_{k}(-1+2a_k)|\le Ca_k^2$. Applying Lemma \ref{lem2_6} on $[b_k,d_k]=[-1,-1+2a_k]$, we have that in this case
         \begin{equation}\label{eq4_3_18}
            \max_{[b_k,d_k]}|h_k|\le Ca_k^2. %, \textrm{ on }[b_k, d_k].
         \end{equation}

      Next, we consider the case $\bar{x}_k\ge -1-a_k$.
      %We have proved in the above that $|h_{k}(b_k)|\le Ca_k^2$ and $|h_{k}(d_k)|\le Ca_k^2$.
      If $\bar{x}_k-2a_k\ge -1$, then $[b_k,d_k]=[\bar{x}_k-2a_k,\bar{x}_k+2a_k]$. If $\bar{x}_k-2a_k<-1$, then $[b_k,d_k]=[-1,-1+2a_k]$ when $\bar{x}_k<-1$, and $[b_k,d_k]=[-1,\bar{x}_k+2a_k]$ when $\bar{x}_k\ge -1$. So we have $\mathrm{dist}(\bar{x}_k, [b_k,d_k])\le Ca_k$, and $2a_k\le d_k-b_k\le 4a_k$.
     Notice that $|h_{k}(b_k)|\le Ca_k^2$, $|h_{k}(d_k)|\le Ca_k^2$,
     $f_k<0$ on $[b_k,d_k]$ and $P_k\ge 0$ in $(-1,1)$,  and using (\ref{eq4_3_0}), $P_k(x)\le P_k(\bar{x}_k)+ 2|c_3|(x-\bar{x}_k)^2$ on $[-1,1]$. Applying Lemma \ref{lem2_10} on $[b_k,d_k]$ with $\alpha=2$, we have that
       \begin{equation}\label{eq4_3_19}
          \max_{[b_k,d_k]}|h_k|\le Ca_k^2+C\nu_k/a_k\le C\nu_k^{2/3}.
       \end{equation}
       By (\ref{eq4_3_16}), (\ref{eq4_3_20}), (\ref{eq4_3_18}) and (\ref{eq4_3_19}), we have $\max_{[-1,1]}|h_k|\le C\nu_k^{2/3}$. So estimate (\ref{eq4_3_1}) is proved when $x_k-\bar{x}_k\ge \epsilon/4$.

        Next, if $x_k-\bar{x}_k<\epsilon/4$. Since $x_k>-1$ and $\bar{x}_k\to -1$, we have $x_k+\epsilon/2>\bar{x}_k+\epsilon/4$, and therefore  we have $P_k\ge 1/C$ on $[x_k+\epsilon/2, 1]$. %$\bar{x}\in (x_k-\epsilon/2, x_k+\epsilon/2)$,
         similar as (\ref{eq4_3_9}) we have

    \[
   1/C\le f_k\le C \textrm{ on }\left(x_k+\epsilon/2, 1\right).
   \]
    Using this and (\ref{eq4_3_7}), applying Lemma \ref{lem2_6} on $(x_k+\epsilon/2, 1)$, (\ref{eq4_3_1}) is proved.

    \textbf{Case 3}: $c_2=0,c_1>0, c_3=c_3^*(c_1,c_2)$. The proof of (\ref{eq4_3_1}) is similar to that of Case 2. 
    We have by now proved (\ref{eq4_3_1}).

  By (\ref{eq4_3_1}) and (\ref{eq_1}),  for any $\epsilon>0$, there exists some constant $C>0$, depending only on $\epsilon$ and an upper bound of $|c|$, such that $|f'_k|\le C\nu_k^{-1/3}$ on  $[-1+\epsilon, x_k-\epsilon]\cup [x_k+\epsilon, 1-\epsilon])$, 
  %$[-1+\epsilon, 1-\epsilon]$, 
  so $|h'_k|=|f_kf'_k-P'_k|\le C\nu_k^{-1/3}$. So we have 
  \begin{equation*}%\label{eq4_3_3}
       ||\frac{1}{2}U^2_{\nu_k, \theta}-P_{c_k}||_{C^1(([-1+\epsilon, x_k-\epsilon]\cup [x_k+\epsilon, 1-\epsilon])\cap [\bar{x}-\epsilon, \bar{x}+\epsilon]}\le C\nu_k^{-1/3},
\end{equation*}
%and, consequently, for any $0<\beta<1$
 %  (\ref{eq4_3_3}) is proved.
        By interpolation %Lemma \ref{lem6_1},
        for any $x, y\in (-1+\epsilon, 1-\epsilon)$ and $0<\beta<1$,
       \[
          \frac{|h_k(x)-h_k(y)|}{|x-y|^{\beta}}\le 2||h_k||^{1-\beta}_{L^{\infty}(-1+\epsilon, 1-\epsilon)}||h_k'||^{\beta}_{L^{\infty}(-1+\epsilon, 1-\epsilon)}\le C\nu_k^{2(1-\beta)/3}\nu_k^{-\beta/3}\le C\nu_k^{2/3-\beta}.
       \]
        We have 
    \begin{equation}\label{eq4_3_4}
   ||\frac{1}{2}U^2_{\nu_k, \theta}-P_{c_k}||_{C^{\beta}([-1+\epsilon, x_k-\epsilon]\cup [x_k+\epsilon, 1-\epsilon])}\le C\nu_k^{2/3-\beta}.
   \end{equation}
    Using (\ref{eq4_3_8}), (\ref{eq4_3_6}) and (\ref{eq4_3_1}), we have (\ref{eq1_7_0}) in this case.
 Part (i) in this case follows from (\ref{eq1_7_0}), (\ref{eq4_3_1}) and (\ref{eq4_3_4}).
   
%   Next, using (\ref{eq4_3_8}), (\ref{eq4_3_6}) and (\ref{eq4_3_1_1}), we have 
%   \begin{equation*}%\label{eq4_3_1}
%    \lim_{k\to \infty}\left(||U_{\nu_k, \theta}+\sqrt{2P_{c}}||_{L^{\infty}(-1, x_k-\epsilon)}+||U_{\nu_k, \theta}-\sqrt{2P_{c}}||_{L^{\infty}(x_k+\epsilon, 1)}\right)=0.%\le C\nu_k^{2/3}.
%\end{equation*}
% Part (ii) in this case follows from the above.

 %  ************************** 

   Next, we prove part (ii) in this case. Notice that in this case $c_3=c_3^*(c_1,c_2)$, $c_1,c_2$ cannot both be zero. We first prove that if such $x_k$ exists and $x_k\to -1$ with $c_1=0$ or such $x_k$ does not exist with $c_2>0=c_1$, then%when $c_1=0$, $x_k$ does not exist, or $x_k$ exists with $x_k\to -1$, %(\ref{eq4_3_1_1}). 
    \begin{equation}\label{eq4_3_1_1}
      \lim_{k\to \infty}||f_k-\sqrt{2P_{k}}||_{L^{\infty}(-1, 1)}=0.
   \end{equation}
  % We first prove the case $c_1=0$ and $x_k\to -1$, the case $c_2=0$, $x_k\to 1$ can be proved similarly.%Notice that if $x_k$ does not exist, since $c_k\in \hat{J}_{\nu_k}$, we know that $c_{k1}=0$ or $c_{k2}=0$.  If $c_{k1}=0$, then $f_{k}(-1)=0$. If $c_{k2}=0$ then $f_k(1)=0$. So we only prove the case when $c_1=0$ and $x_k\to -1$, the other cases can be proved similarly.
 In this case $P_c(x)=-\frac{1}{2}c_2(x+1)^2$ where $c_2>0$. %$P_c(-1)=0$. 
 By Theorem \ref{thm1_2_1}(i), we have
\[
   \limsup_{k\to \infty}||U^{\pm}_{\nu_k, \theta}\mp \sqrt{2P_{k}}||_{L^{\infty}(-1, 1)}=0.
\]
So for any $\epsilon_0>0$, there exists some $\epsilon>0$, such that $||P_k||_{L^{\infty}L^{\infty}(-1,-1+2\epsilon)}<\epsilon_0$, and 
$||U^{\pm}_{\nu_k, \theta}||_{L^{\infty}(-1,-1+2\epsilon)}<\epsilon_0$ for large $k$. Notice  $U_{\nu_k, \theta}^-\le f_k\le U^+_{\nu_k, \theta}$, we then have $||f_k-\sqrt{2P_{k}}||_{L^{\infty}(-1,-1+2\epsilon)}<2\epsilon_0$. %This together with  (\ref{eq4_3_1}) give (\ref{eq4_3_1_1}).
%$||f_k||_{L^{\infty}(-1,-1+2\epsilon)}<\epsilon_0$.
Since $P_c(x)=-\frac{1}{2}c_2(x+1)^2$ we also have that $P_c\ge 1/C$ on $[-1+\epsilon, 1]$. Moreover, if such $x_k$ does not exist, then since $f_k(1)=\tau'_2(\nu_k, c_{k2})>0$, we have $f_k>0$ on $(-1,1]$. If such $x_k$ exists and $x_k\to -1$, then for $k$ large we have $-1<x_k<-1+\epsilon$. By (\ref{eq4_3_8}), we also have $f_k>0$ on $[-1+\epsilon, 1]$. %By Lemma , there exists some $y_k\in (-1+\epsilon,-1+2\epsilon)$, such that $|h_k(y_k)|\le C\nu_k$. Then
Then by Corollary \ref{cor2_3}, we have $f_k\ge 1/C$ on $[-1+2\epsilon, 1]$. Notice $|f_k(1)-\sqrt{2P_c(-1)}|\le C\nu_k$ and $|f_k(-1+2\epsilon)-\sqrt{2P_c(-1+2\epsilon)}|\le 2\epsilon_0$,  by Corollary \ref{cor2_2} we have  $||f_k-\sqrt{2P_c}||_{L^{\infty}(-1+2\epsilon,1)}<C\epsilon_0$. So (\ref{eq4_3_1_1}) is proved. %The case $c_2=0$ can be proved similarly.

Similarly, if such $x_k$ exists and $x_k\to 1$ with $c_2=0$ or such $x_k$ does not exist with $c_1>0=c_2$, , we have 
\[
   \lim_{k\to \infty}||U_{\nu_k, \theta}+\sqrt{2P_{c_k}}||_{L^{\infty}(-1, 1)}=0.
\]
%Now we prove (\ref{eq4_3_1_1}) when $x_k$ does not exist. In this case we must have $c_1=0$ or $c_2=0$. If $c_1=0$, then $c_2>0$ and $P_c(-1)=0$. As above we can prove that for any $\epsilon_0>0$, there exists $\epsilon>0$, such that $||\frac{1}{2}f^2_k-P_{k}||_{L^{\infty}(-1,-1+2\epsilon)}<2\epsilon_0$ for large $k$. Then we have $P_c\ge 1/C$ on $[-1+\epsilon, 1]$ and $f_k>0$ on $[-1+\epsilon, 1]$. %By Lemma , there exists some $y_k\in (-1+\epsilon,-1+2\epsilon)$, such that $|h_k(y_k)|\le C\nu_k$. Then
%By Corollary \ref{cor2_3}, we have $f_k\ge 1/C$ on $[-1+2\epsilon, 1]$. Notice $|h_k(1)|\le C\nu_k$ and $|h_k(-1+2\epsilon)|\le 2\epsilon_0$,  by Lemma \ref{lem2_6} we have  $||\frac{1}{2}f^2_k-P_{k}||_{L^{\infty}[-1+2\epsilon, 1]}<C\epsilon_0$. So (\ref{eq4_3_1_1}) is proved. The case $c_2=0$ can be proved similarly.
%***************************
\qed

\section{$c\in \partial J_0$ and $c_3>c_3^*(c_1,c_2)$}\label{sec5}%Proof of Theorem \ref{thm1_5} and \ref{thm1_8}}

%Recall if $c_1=0$, then $c_3^*(c_1,c_2)=-\frac{1}{2}c_2$. If $c_2=0$, then $c_3^*(c_1,c_2)=-\frac{1}{2}c_1$.

\begin{lem}\label{lem5_1}
Let $-1\le a<b\le 1$,  $\nu_k\to 0^+$, $c_k\in J_{\nu_k}$, $c_1c_2=0, c_3>c_3^*(c_1,c_2)$, and $c_k\to c$ as $k\to \infty$. Then
\[
   \min_{[a,b]}P_{c_k}=\min\{P_{c_k}(a), P_{c_k}(b)\},
\]
for sufficiently large $k$.
\end{lem}
\begin{proof}
   When $c_3>0$, we have $P_{c_k}''(x)=-2c_{k3}<0$ for large $k$. Thus $P_{c_k}$ is concave down, $\min_{[a,b]}P_{c_k}=\min\{P_{c_k}(a), P_{c_k}(b)\}$. 
   When $c_3\le 0$, we distinguish to two cases.
   Case $1$ is $c_1=0$ and Case $2$ is $c_2=0$. If $c_1=0$, then $P_c'(x)=c_2-2c_3x\ge c_2+2c_3>0$ in $[-1,1]$. So $P_k'>0$ in $[-1,1]$ for large $k$. Thus $\min_{[a,b]}P_{c_k}=P_{c_k}(a)$. If $c_2=0$, then $P_c'(x)=-c_1-2c_3x\le -c_1-2c_3<0$ in $[-1,1]$. So $P_{c_k}'<0$ in $[-1,1]$ for large $k$. Thus $\min_{[a,b]}P_{c_k}=P_{c_k}(b)$.
   %If $c_1=0$, then we have $P_c(x)=c_2(1+x)+c_3(1-x^2)$ and $c_3>c_3^*(c_1,c_2)=-\frac{1}{2}c_2$. If $c_3\le 0$, then $P_c'(x)=c_2-2c_3x\ge c_2+2c_3>0$ in $[-1,1]$. Since $c_k\to c$, for sufficiently large $k$ we have $P_k'>0$ in $[-1,1]$. Thus $\min_{[a,b]}P_k=P_k(a)$. If $c_3>0$, then $P_c''(x)=-2c_3<0$ in $[-1,1]$. So for large $k$ we have $P_k''<0$ in $[-1,1]$. Thus $P_k$ is concave down, $\min_{[a,b]}P_k=\min\{P_k(a), P_k(b)\}$.
%
%   If $c_2=0$, then we have $P_c(x)=c_1(1-x)+c_3(1-x^2)$ and $c_3>c_3^*(c_1,c_2)=-\frac{1}{2}c_1$. If $c_3\le 0$, then $P_c'(x)=-c_1-2c_3x\le -c_1-2c_3<0$ in $[-1,1]$. Since $c_k\to c$, for sufficiently large $k$ we have $P_k'<0$ in $[-1,1]$. Thus $\min_{[a,b]}P_k=P_k(b)$. If $c_3>0$, then $P_c''(x)=-2c_3<0$ in $[-1,1]$. So for large $k$ we have $P_k''<0$ in $[-1,1]$. Thus $P_k$ is concave down, $\min_{[a,b]}P_k=\min\{P_k(a), P_k(b)\}$.
\end{proof}

\begin{lem}\label{lem5_3}
   Let $\nu_k\to 0^+$, $c_k\in J_{\nu_k}$, $-1<b\le 1$, $c_1c_2=0, c_3>c_3^*(c_1,c_2)$, and $c_k\to c$ as $k\to \infty$. If $P_c(b)> 0$, then $U^+_{\nu_k, \theta}>0$ on $[-1,b]$ for sufficiently large $k$.
\end{lem}
\begin{proof}
   For convenience denote $f_k=U_{\nu_k, \theta}^+$, and $P_k=P_{c_k}$.
   If $c_{k1}\ge 0$, we have $P_k(-1)=2c_{k1}\ge 0$ for large $k$. Since $P_c(b)>0$, we also have $P_k(b)>0$. By Lemma \ref{lem5_1} we have $P_k(x)\ge \min\{P_k(-1), P_k(b)\}>0$. Using this and the fact that $f_k(-1)=\tau_2(c_{k1}, \nu_k)>0$, by Lemma \ref{lem2_4} we have $f_k>0$ on $[-1,b]$.

   %If $c_1=0$, we have $c_3>-\frac{1}{2}c_2$ and $P_c(x)=c_2(1+x)+c_3(1-x^2)$. So $P_c(-1)=0$ and $P_c'(-1)=c_2+2c_3>0$. Since $c_k\to c$ as $k\to \infty$, there exists some $C_0>0$ and $\delta>0$, such that
%   \begin{equation}\label{eq5_3_1}
%       C_0(1+x)\le P_k(x)-P_k(-1)\le 2C_0(1+x), \quad -1<x<-1+\delta
%   \end{equation}
%   for $k$ sufficiently large.

    If $c_{k1}<0$, since $c_{k1}\to c_1\ge 0$, we must have $c_1=0$, and then $c_3>-c_2/2$ and $P_c(x)=c_2(1+x)+c_3(1-x^2)$. So $P_c(-1)=0$ and $P_c'(-1)=c_2+2c_3>0$. Since $c_k\to c$ as $k\to \infty$, there exists some $C_0>0$ and $\delta>0$, such that
   \begin{equation}\label{eq5_3_1}
       C_0(1+x)\le P_k(x)-P_k(-1)\le 2C_0(1+x), \quad -1<x<-1+\delta
   \end{equation}
   for $k$ sufficiently large.
      Notice $P_k(-1)=2c_{k1}\ge -2\nu_k^2$. %If $c_{k1}\ge 0$, then $P_k(-1)=2c_{k1}\ge 0$. Since $P_k(b)>0$, by Lemma \ref{lem5_1} we have $P_k\ge 0$ on $[-1,b]$. Using the fact $f_k(-1)=2\nu_k+2\sqrt{\nu_k^2+c_{k1}}>0$, by Lemma \ref{lem2_4}  we have $f_k>0$ on $[-1,b]$.
   Since $c_{k1}<0$, then since $P_k(-1)=2c_{k1}\ge -2\nu_k^2$, by (\ref{eq5_3_1}), we have $P_k(-1+2\nu^2_k/C_0)\ge 0$. So by Lemma \ref{lem5_1}, we have $P_k\ge \min\{P_k(-1+2\nu^2_k/C_0), P_k(b)\}\ge 0$ on $[-1+2\nu^2_k/C_0, b]$.

   Next, let
   \[
      g_k(x):=f_k(-1)-\frac{C_0}{8\nu_k}(1+x)
   \]
   Since $f_k(-1)\ge 2\nu_k$, it can be checked that $g_k>0$ on $[-1,-1+2\nu^2_k/C_0]$. 
   By computation, using the facts that $\frac{1}{2}f^2_k(-1)-2\nu_kf_k(-1)=2c_{k1}$ and $f_k(-1)\ge 2\nu_k$, we have that for $-1\le x\le -1+2\nu^2_k/C_0$ and $k$ sufficiently large,
   \[
      \begin{split}
        & Q_k:=\nu_k(1-x^2)g_k'+2\nu_k xg_k+\frac{1}{2}g_k^2 \\
       % & =-\frac{C_0}{4\nu_k}\nu_k(1-x^2)+2\nu_k x(f_k(-1)-\frac{C_0}{4\nu_k}(1+x))+\frac{1}{2}\left(f^2_k(-1)-\frac{C_0}{2\nu_k}f_k(-1)(1+x)+\frac{C^2_0}{16\nu^2_k}(1+x)^2\right)\\
       % & =\frac{C_0}{8}(1+x)^2-\frac{C_0}{4}(1+x)+2\nu_kf_k(-1)(1+x)-2\nu_kf_k(-1)-\frac{C_0}{4}(1+x)^2+\frac{C_0}{4}(1+x)+\frac{1}{2}f^2_k(-1)\\
       % & -\frac{C_0}{8\nu_k}f_k(-1)(1+x)
       % +\frac{C^2_0}{128\nu^2_k}(1+x)^2\\
        %& =\left(\frac{C^2_0}{128\nu^2_k}-\frac{C_0}{8}\right)(1+x)^2+\left(2\nu_kf_k(-1)-\frac{C_0}{8\nu_k}f_k(-1)\right)(1+x)+\frac{1}{2}f^2_k(-1)-2\nu_kf_k(-1)\\
        & =\left([\frac{C^2_0}{128\nu^2_k}-\frac{C_0}{8}](1+x)+(2\nu_k-\frac{C_0}{8\nu_k})f_k(-1)\right)(1+x)+\frac{1}{2}f^2_k(-1)-2\nu_kf_k(-1)\\
         & \le \left([\frac{C^2_0}{128\nu^2_k}-\frac{C_0}{8}]\frac{2\nu_k^2}{C_0}+4\nu_k^2-\frac{C_0}{4}\right)(1+x)+2c_{k1}\\
        % & \le -\frac{C_0}{8}(1+x)+2c_{k1}\\
         & \le 2c_{k1}<P_k(x).
      \end{split}
   \]
   By Lemma 2.3 in \cite{LLY2}, we have $\limsup_{x\to -1^+}|x+1|^{-1}|f_k(x)-f_k(-1)|<\infty$. So we have $g_k(-1)=f_k(-1)> 2\nu_k$ or $f_k(-1)=g_k(-1)=2\nu_k$ with
   $\limsup_{x\to -1^+}\int_{-1+2\nu_k^2/C_0}^{x}(1-s^2)^{-1}(-2\nu_k+f_k(s))ds<\infty$.
   It can be checked that $f_k$ is a solution of (\ref{eq_1}) if and only if $\nu_k f_k$ is a solution of (\ref{eqNSE_1}). Similarly, $\nu_k g_k$ is a solution of (\ref{eqNSE_1}) with the right hand side to be $Q_k/\nu_k^2$.  Notice that $Q_k< P_k$, applying Lemma 2.4 in \cite{LLY2}, we have $f_k\ge g_k>0$ on $(-1,-1+2\nu_k^2/C_0]$.
     Since $P_k\ge 0$ on $[-1+2\nu_k^2/C_0, b)$ and $f_k(-1+2\nu_k^2/C_0)>0$, we have, by Lemma \ref{lem2_4}, that $f_k>0$ on $[-1+2\nu_k^2/C_0, b]$. So $f_k>0$ in $[-1,b]$, the lemma is proved.
\end{proof}

\noindent{\textit{Proof of Theorem \ref{thm1_2_2}}}:

We only prove the results for $U_{\nu_k, \theta}^+$, the proof of the results for $U_{\nu_k, \theta}^-$ is similar. 
Let $C$ be a positive constant depending only on $c$
%and a positive lower bound of $-c_3$
which may vary from line to line. For convenience, write $f_k=U_{\nu_k, \theta}^+$, $P_k=P_{c_k}$, and let $h_k:=\frac{1}{2}f^2_k-P_k$. In the following we always assume that $k$ is large.

Since $c_k\to c$ as $k\to \infty$, by Lemma \ref{lem2_1}, we have $f_k\le C$ in $[-1,1]$. 
We first prove 
\begin{equation}\label{eq5_1_1}
    ||\frac{1}{2}(U^{+}_{\nu_k, \theta})^2-P_{c_k}||_{L^{\infty}(-1,1)}\le C\nu_k^{1/2} .
\end{equation}

  \textbf{Case} 1: $c_1=0$, $c_2>0$, $c_3>-c_2/2$.

  In this case, $P_c(x)=c_2(1+x)+c_3(1-x^2)$ in $(-1,1)$. So $P_c(-1)=2c_1=0$, and $P_c'(-1)=c_2+2c_3>0$. Since $c_k\to c$ as $k\to \infty$,  there exists some $\delta>0$, such that for large $k$,
   \begin{equation}\label{eq5_3_5}
        \frac{1}{2}P_{c}'(-1)(1+x)\le P_{k}(x)-P_{k}(-1)\le 2P_{c}'(-1)(1+x), \quad -1<x<-1+\delta.
   \end{equation}
Let $a_k=\nu_k^{1/2}/\alpha$ for some positive constant $\alpha$ to be determined.
 Then by Lemma \ref{lem2_5}, there exists some $x_{k}\in (-1+a_k,-1+2a_k)$, such that $|h_k(x_{k})|\le \frac{C\nu_k}{a_k}=C\alpha^2 a_k$. It follows from (\ref{eq5_3_5}) and the fact that $P_k(-1)=2c_{k1}\ge -\nu_k^2$ and $P_c'(-1)>0$, that
       \[
         \begin{split}
           \frac{1}{2}f^2_k(x_k) & \ge P_k(x_k)-|h_k(x_k)|\ge P_k(-1)+\frac{1}{2}P_c'(-1)(x_k+1)-|h_k(x_k)|\\
           & \ge -2\nu_k^2+\left(\frac{1}{2}P'_c(-1)-C\alpha^2\right)a_k
           \ge \left(\frac{1}{4}P'_c(-1)-C\alpha^2\right)a_k.
           \end{split}
       \]
 Fix $\alpha^2=P'_c(-1)/(8C)$. By (\ref{eq5_3_5}) $P_c(x_k)-2\nu_k^2+\frac{1}{2}P'_c(-1)a_k>0$, by Lemma \ref{lem5_3}, we have $f_k>0$ on $[-1,x_k]$.  So $f_{k}(x_{k})\ge \sqrt{a_k/C}$. 
 Since $P_k(1)>\frac{1}{2}P_c(1)>0$ for large $k$, by (\ref{eq5_3_5}) we have
 \[
     P_k(-1+a_k)\ge 2c_{k1}+a_k/C\ge -2\nu_k^2+a_k/C\ge a_k/C.
      \]
       Then by Lemma \ref{lem5_1}, we have $P_k(x)\ge a_k/C$ in $[-1+a_k, 1]$ for $k$  large. Applying Lemma \ref{lem2_2} on $[x_k,1]$, we have $f_k(x)\ge \sqrt{a_k/C}$ on $[x_{k}, 1]$.
           We also have $|h_{k}(x_{k})|\le C\nu_k^{1/2}$, $|h_k(1)|=|\frac{1}{2}f^2_k(1)-P_k(1)|=|\frac{1}{2}(\tau_2'(\nu,c_{k2}))^2-2c_{k2}|\le C\nu_k$. So by applying Lemma \ref{lem2_6} on $[x_k,1]$,
       \[
          \max_{[-1+2a_k,1]}|h_k|\le \max_{[x_k,1]}|h_k|\le C\nu_k^{1/2}.
                 \]
Now we have $h_k(-1+2a_k)\le Ca_k$ and $|h_k(-1)|=|\frac{1}{2}f^2_k(-1)-P_k(-1)|=|\frac{1}{2}(\tau_2(\nu,c_{k1}))^2-2c_{k1}|\le C\nu_k$. By (\ref{eq5_3_5}) we have $P_k(x)\ge P_k(-1)=2c_{k1}\ge -2\nu_k^2\ge -Ca_k$ in $[-1,-1+2a_k]$ for $k$ large.
%So $\min_{[-1,-1+2a_k]} P_k=P_k(-1)=2c_{k1}\ge -2\nu_k^2\ge -Ca_k$.
By (\ref{eq5_3_5}) we also have that $P_k(x)\le P_k(-1)+C(1+x)$. Notice  $f_k(-1)=2\nu_k+2\sqrt{\nu_k^2+c_{k1}}>0$,  %$P'_k\ne 0$ on $[-1,-1+2a_k)$ for sufficiently large $k$.
applying Lemma \ref{lem2_10} on $[-1,-1+2a_k]$ with $\alpha=1$ and $\bar{x}_k=-1$ there, we have that
       \[
          \max_{(-1,-1+2a_k)}|h_k|\le Ca_k+C\nu_k/\sqrt{a_k}\le C\nu_k^{1/2}.
       \]
       Estimate (\ref{eq5_1_1}) is proved in this case.
%So we conclude that
%\[
%          \max_{(-1,1)}|h_k|\le  C\nu_k^{1/2}.
%       \]
       %Now we estimate $h_k$ in $(-1,- 1+2a_k)$. Suppose $\max_{[-1,- 1+2a_k]}|h_k|=|h_{\nu_k}(\tilde{z}_k)|$ for some $\tilde{z}_k\in (-1,- 1+2a_k)$. We will prove $|h_k(\tilde{z}_k)|\le C a_k$.
%
%        If $P_{c_k}(-1)> 2Ca_k$,  then since $|h_{k}(-1)|\le C\nu_k^{1/2}$,  we have
%       \[
%          f^2_k(-1)\ge P_{c_k}(-1)-C\nu_k^{1/2}\ge P_{c_k}(-1)+a_k/C-Ca_k\ge Ca_k.
%       \]
%       So $f_k(-1)\ge \sqrt{Ca_k}$. Then by Lemma \ref{lem2_2}, we have
%       \[
%           f_k(x)\ge \sqrt{Ca_k}, \quad -1<x<1.
%       \]
%
%       If $P_{c_k}(-1)\le 2Ca_k$, then by (\ref{eq4_3_2}), we have
%       \[
%           \frac{1}{2}f^2_k(\tilde{z}_k)\ge |h_k(\tilde{z}_k)|-|P_{c_k}(\tilde{z}_k)| \ge |h_k(\tilde{z}_k)|-Ca_k.
%       \]
%
%       If $ |h_k(\tilde{z}_k)|\ge 2Ca_k$, then $f_k(\tilde{z}_k)\ge \sqrt{2Ca_k}$.
%
%
%       So we have that
%       \[
%          |f'_k(\tilde{z}_k)|=\frac{|P_c'(\tilde{z}_k)|}{|f_{\nu}(\tilde{z}_k)|}\le \frac{C}{\sqrt{a_k}}.
%       \]
%
%       By (\ref{eq_1}) we have
%       \[
%          k\nu_k^{1/2}\le h_{k}(\tilde{z}_k)\le C\nu_k |f'_k(\tilde{z}_k)|+C\nu_k\le C\nu_k/\sqrt{a_k}+C\nu_k\le  C\nu_k^{1/2}.
%       \]
%       Contradiction. We have also proved $h_{k}(-1)\le C\nu_k$ and $h_k(-1+2a_k)\le C\nu_k^{1/2}$. So
%       \[
%          \max_{(-1, -1+2a_k)}|h_k|\le C\nu_k^{1/2}.
%       \]

  \textbf{Case} 2: $c_1>0$, $c_2=0$, $c_3>-\frac{1}{2}c_1$.

  In this case, $P_c(x)=c_1(1-x)+c_3(1-x^2)$. So $P_c(1)=0$ and $P_c'(1)=-c_1-2c_3<0$. Since $c_k\to c$ as $k\to \infty$, there exists some $\delta>0$, such that %$P_k(1)=0$, and there is some $C_0>0$ and $\delta_1>0$, such that
  \begin{equation}\label{eq5_1_2}
     -\frac{1}{2}P'_c(1)(1-x)<P_k(x)-P_k(1)<-2P'_c(1)(1-x), \quad 1-\delta<x<1,
  \end{equation}
  for  large $k$.
  Let $a_k=\nu_k^\frac{1}{2}$. % for some positive constant $\alpha$ to be determined later. %Then we have $a_k>1-b_k$, therefore $1-a_k<b_k$  for large $k$.
  By Lemma \ref{lem2_5}, there exists some $x_{k}\in (1-2a_k, 1-a_k)$, such that $|h_k(x_{k})|\le C\nu_k/a_k=C a_k$. 
  Since $P_c(-1)=2c_1>0$, we have $P_k(-1)>\frac{1}{2}P_c(-1)>0$ for  large $k$. By (\ref{eq5_1_2}) we have
 \[
     P_k(1-a_k)\ge 2c_{k2}+a_k/C\ge -2\nu_k^2+a_k/C\ge a_k/C.
      \]
          Then by Lemma \ref{lem5_1}, we have $P_k(x)\ge a_k/C$ in $[-1, 1-a_k]$ for $k$  large. Notice $f_k(-1)=\tau_2(\nu_k,c_{k1})\ge \sqrt{c_1}>0$ for large $k$. Applying Lemma \ref{lem2_2} on $[-1, 1-a_k]$, we have $f_k(x)\ge \sqrt{a_k/C}$ on $[-1, 1-a_k]$.
        We also have $|h_{k}(x_{k})|\le C\nu_k^{1/2}$, $|h_k(-1)|=|\frac{1}{2}f^2_k(-1)-P_k(-1)|=|\frac{1}{2}(\tau_2(\nu,c_{k1}))^2-2c_{k1}|\le C\nu_k$. So by applying Lemma \ref{lem2_6} on $[-1,x_k]$,
       \[
          \max_{[-1, 1-2a_k]}|h_k|\le \max_{[-1, x_k]}|h_k|\le C\nu_k^{1/2}.
                 \]
Now we have $h_k(1-2a_k)\le Ca_k$ and $|h_k(1)|=|\frac{1}{2}f^2_k(1)-P_k(1)|=|\frac{1}{2}(\tau'_2(\nu,c_{k2}))^2-2c_{k2}|\le C\nu_k$. By (\ref{eq5_1_2}) we have $P_k(x)\ge P_k(1)=2c_{k1}\ge -2\nu_k^2\ge -Ca_k$ in $[1-2a_k, 1]$ for $k$ large. %So $\min_{[1-2a_k, 1]} P_k=P_k(1)=2c_{k1}\ge -2\nu_k^2\ge -C\nu_k\ge -Ca_k$.
By (\ref{eq5_1_2}) we also have that $P_k(x)\le P_k(1)+C(1-x)$. Notice  $f_k(1-2a_k)\ge \sqrt{a_k/C}>0$,  %$P'_k\ne 0$ on $[-1,-1+2a_k)$ for sufficiently large $k$.
applying Lemma \ref{lem2_10} on $[1-2a_k, 1]$ with $\alpha=1$ and $\bar{x}_k=1$ there, we have that
       \[
          \max_{[1-2a_k, 1]}|h_k|\le Ca_k+C\nu_k/\sqrt{a_k}\le C\nu_k^{1/2}.
       \]
       Estimate (\ref{eq5_1_1}) is proved in this case.

  %By %defined similar as in case 1, by similar proof there we have that
%  \[
%     \max_{[-1,y_k-2a_k]}|h_k|<C\nu_{k}^{1/2}.
%  \]
%
%
%Now we have that $h_k(1)\le Ca_k$ and $h_k(y_k-2a_k)\le Ca_k$. Notice that $f_k(y_k-2a_k)>0$, and $\min_{[y_k-2a_k,1]} P_k=P_k(1)\ge -C\nu_k$, %$P'_k\ne 0$ in $(y_k-2a_k, 1]$ for $k$ sufficiently large.
%By Lemma \ref{lem2_10} we have that
%       \[
%          \max_{(y_k-2a_k,1)}|h_k|\le Ca_k+C\nu_k/\sqrt{a_k}\le C\nu_k^{1/2}.
%       \]
%So we conclude that
%\[
%          \max_{(-1,1)}|h_k|\le  C\nu_k^{1/2}.
%       \]

%On the interval $(y_k-2a_k, 1)$, we have $|P_k|<Ca_k$. Suppose $\max_{[y_k-2a_k, 1]}|h_k|=|h_{\nu_k}(\tilde{z}_k)|$ for some $\tilde{z}_k\in (y_k-2a_k, 1)$. We will prove $|h_k(\tilde{z}_k)|\le C a_k$.
%
%       We have
%       \[
%           \frac{1}{2}f^2_k(\tilde{z}_k)\ge |h_k(\tilde{z}_k)|-|P_{c_k}(\tilde{z}_k)| \ge |h_k(\tilde{z}_k)|-Ca_k.
%       \]
%
%       If $ |h_k(\tilde{z}_k)|\ge 2Ca_k$, then $|f_k(\tilde{z}_k)|\ge \sqrt{Ca_k}$.
%
%
%       So we have that
%       \[
%          |f'_k(\tilde{z}_k)|=\frac{|P_c'(\tilde{z}_k)|}{|f_{\nu}(\tilde{z}_k)|}\le \frac{C}{\sqrt{a_k}}.
%       \]
%
%       By (\ref{eq_1}) we have
%       \[
%           h_{k}(\tilde{z}_k)\le C\nu_k |f'_k(\tilde{z}_k)|+C\nu_k\le C\nu_k/\sqrt{a_k}+C\nu_k\le  C\nu_k^{1/2}.
%       \]
%        We have also proved $h_{k}(-1)\le C\nu_k$ and $h_k(y_k-2a_k)\le C\nu_k^{1/2}$. So
%       \[
%          \max_{(y_k-2a_k, 1)}|h_k|\le C\nu_k^{1/2}.
%       \]

       \textbf{Case} 3: $c_1=c_2=0$, $c_3>0$.

       In this case, $P_c(x)=c_3(1-x^2)$ in $(-1,1)$. So $P_c(\pm 1)=0$, $P_c'(-1)=2c_3>0$ and $P_c'(1)=-2c_3<0$. Since $c_k\to c$ as $k\to \infty$,  there exists some $\delta>0$, such that for large $k$, (\ref{eq5_3_5}) and (\ref{eq5_1_2}) are true.
 Let $a_k=\nu_k^{1/2}/\alpha$ for some positive constant $\alpha$ to be determined.
 Then by Lemma \ref{lem2_5}, there exists some $x_{k}\in (-1+a_k,-1+2a_k)$ and $y_k\in (1-2a_k, 1-a_k)$, such that $|h_k(x_{k})|+|h_k(y_{k})|\le C\nu_k/a_k=C\alpha^2 a_k$. Similar as Case 1, we have $f_{k}(x_{k})\ge \sqrt{a_k/C}$. 
 By (\ref{eq5_3_5}), (\ref{eq5_1_2}), and  Lemma \ref{lem5_1}, we have $P_k(x)\ge a_k/C$ in $[x_k, y_k]$ for $k$ large. %we have
 %\[
%     \min\{P_k(x_k), P_k(y_k)\} \ge a_k/C,  \quad P_k(y_k)a_k/C, %\ge 2c_{k1}+a_k/C\ge -2\nu_k^2+a_k/C\ge a_k/C.
%      \]
     % By (\ref{eq5_1_2}) we have
%\[
%     P_k(1-a_k)\ge 2c_{k2}+a_k/C\ge -2\nu_k^2+a_k/C\ge a_k/C.
%      \]
     % and, by Lemma \ref{lem5_1}, we have $P_k(x)\ge a_k/C$ in $[x_k, y_k]$ for $k$ large.
      Applying Lemma \ref{lem2_2} on $[x_k,y_k]$, we have $f_k(x)\ge \sqrt{a_k/C}$ on $[x_{k}, y_k]$.
  We also have $|h_{k}(x_{k})|\le C\nu_k^{1/2}$, $|h_k(y_k)|\le C\nu_k^{1/2}$. So by applying Lemma \ref{lem2_6} on $[x_k,y_k]$,
       \begin{equation*}%\label{eq5_1_4}
          \max_{[-1+2a_k,1-2a_k]}|h_k|\le \max_{[x_k,y_k]}|h_k|\le C\nu_k^{1/2}.
                 \end{equation*}
As in Case 1 and Case 2, we have
       \begin{equation*}%\label{eq5_1_5}
          \max_{[-1,-1+2a_k]}|h_k|+\max_{[1-2a_k, 1]}|h_k|%\le Ca_k+C\nu_k/\sqrt{a_k}
          \le C\nu_k^{1/2}.
       \end{equation*}
      % We also have $h_k(1-2a_k)\le Ca_k$ and $|h_k(1)|=|\frac{1}{2}f^2_k(1)-P_k(1)|=|\frac{1}{2}(\tau'_2(\nu,c_{k2}))^2-2c_{k2}|\le C\nu_k$. By similar proof as that of Case 2, we have that %By (\ref{eq5_1_2}) we have $P_k(x)\ge P_k(1)$ in $[1-2a_k, 1]$ for $k$ large. So $\min_{[1-2a_k, 1]} P_k=P_k(1)=2c_{k1}\ge -2\nu_k^2\ge -C\nu_k\ge -Ca_k$. By (\ref{eq5_1_2}) we also have that $P_k(x)\le P_k(1)+C(1-x)$. Notice  $f_k(1-2a_k)\ge \sqrt{a_k/C}>0$,  %$P'_k\ne 0$ on $[-1,-1+2a_k)$ for sufficiently large $k$.
%applying Lemma \ref{lem2_10} on $[1-2a_k, 1]$ with $\alpha=1$ and $\bar{x}_k=1$ there, we have that
%As in Case 2, we have
%       \begin{equation}\label{eq5_1_6}
%          \max_{[1-2a_k, 1]}|h_k|%\le Ca_k+C\nu_k/\sqrt{a_k}
%          \le C\nu_k^{1/2}.
%       \end{equation}
     % By (\ref{eq5_1_4})and (\ref{eq5_1_5}),
      By the above, estimate (\ref{eq5_1_1}) is proved in this case.

      From (\ref{eq5_1_1}) we have $\lim_{k\to 0}|||f_k|-\sqrt{2P_c}||_{L^{\infty}(-1,1)}=0$. By Lemma \ref{lem5_3} we have $f_k>0$ on $[-1,1-2a_k]$. Using this and the fact $\max_{[1-2a_k,1]}|P_k|\le Ca_k$,  we have $ \lim_{k\to \infty}||U^{-}_{\nu_k, \theta}+\sqrt{2P_{c}}||_{L^{\infty}(-1,1)}=0$. %(\ref{eq5_1_0}).

       Next, let $\epsilon>0$ be any fixed positive small constant. If $c_1=0$, by (\ref{eq5_3_5}) we have that  $P_k(-1+\frac{1}{2}\epsilon)\ge epsi\epsilon$. If $c_1>0$, $P_k(-1)\ge \sqrt{c_1}$. Similarly, if $c_2=0$, by (\ref{eq5_1_2}), $P_k(1-\epsilon/2)\ge \epsilon/C$. If $c_2>0$, $P_k(1)\ge \sqrt{c_2}>0$. By Lemma \ref{lem5_1} we have $P_k\ge \epsilon/C$ on $[-1+\epsilon/2, 1-\epsilon/2]$. %From the above proof for (\ref{eq5_1_1}),
       As proved above, we also have $f_k>0$ on $[-1+\epsilon/2, 1-\epsilon/2]$ for large $k$. Applying Lemma \ref{lem2_10} on $[-1+\epsilon/2, 1-\epsilon/2]$, we obtain % estimate (\ref{eq5_1_3}).
       \begin{equation*}%\label{eq5_1_3}
 ||U^{+}_{\nu_k, \theta}-\sqrt{2P_{c_k}}||_{C^{m}([-1+\epsilon,1-\epsilon])}\le C\nu_k.%<\infty
   \end{equation*}
   The proof is finished.
   \qed
   
   \begin{rmk}
The assumption of $c$ in Theorem \ref{thm1_2_2} is equivalent to $P_{c}(1)P_c(-1)=0$,  $P^2_{c}(1)+(P'_c(1))^2\ne 0$ and $P^2_{c}(-1)+(P'_c(-1))^2\ne 0$.
\end{rmk}

%
%
%
%
%
%
%
%       %In this case $P_c(x)=c_3(1-x^2)$.  $P_c(\pm 1)=0$. We obtain $\max_{[-1, 0]}|h_k|\le C\nu_k^{1/2}$ with similar arguments as in Case 1, and  $\max_{[0,1]}|h_k|\le C\nu_k^{1/2}$ with similar arguments as in Case 2.
%
%       Next, by (\ref{eq1_5_1}) and (\ref{eq_1}),  for any $\epsilon>0$, there is some $C$, depending only on $\epsilon$ and $c$, such that $|f'_k|\le C\nu^{-\frac{1}{2}}$ on $(-1+\epsilon, 1-\epsilon)$. So $|h'_k|=|f_kf'_k-P'_k|\le C\nu_k^{-\frac{1}{2}}$ on $(-1+\epsilon, 1-\epsilon)$. Then apply Lemma \ref{lem6_1}, we have that
%       \[
%          [h_k]_{\beta}\le C(\nu_k^{1/2}+\nu_k^{\frac{1}{2}(1-\beta)}\nu_k^{-\frac{1}{2}\beta})\le C\nu_k^{\frac{1}{2}-\beta}, \quad -1+\epsilon<x< 1-\epsilon
%       \]
%       for any $0<\beta<\frac{1}{2}$. So (\ref{eq1_5_2}) is proved.

      % The  Theorem is proved.
      % \qed
%\end{proof}
 % \qed

%Theorem \ref{thm1_2_2} follows from Theorem \ref{thm5_1} and  Theorem \ref{thm5_1'}. % and Theorem \ref{thm5_1}.

Next, we study solutions of (\ref{eq:NSE}) which are not $U_{\nu_k, \theta}^\pm$.

%\begin{thm}\label{thm5_4}
%Let $\nu_k\to 0^+$, $c_k\in J_0$, $c_k\to c$, $m$ be a positive integer, assume $c_1c_2=0$, $c_3>c_3^*(c_1,c_2)$, and $U_{\nu_k, \theta}\in C^1(-1,1)$ is a solution of (\ref{eq:NSE}) with $\nu_k$ and $c_k$ other than $U_{\nu_k, \theta}^\pm$, then there exists at most one $-1< x_k< 1$ such that  $U_{\nu_k, \theta}(x_k)=0$ for large $k$. If $x_k$ exists, for any $\epsilon>0$, there exists some constant $C$, depending only on $\epsilon, m$ and $c$, such that for large $k$,%an upper bound of $|c|$, such that for large $k$,
%\begin{equation}\label{eq5_4_1}
%    ||\frac{1}{2}U^2_{\nu_k, \theta}-P_{c_k}||_{L^{\infty}((-1, x_k-\epsilon)\cup (x_k+\epsilon, 1))}\le C\nu_k^{1/2}.
%\end{equation}
%
%Moreover, for any $\epsilon>0$, %$\bar{x}:=\frac{\sqrt{c_1}-\sqrt{c_2}}{\sqrt{c_1}+\sqrt{c_2}}\in [-1,1]$,
%\begin{equation}\label{eq5_4_2}
%   ||U_{\nu_k, \theta}+\sqrt{2P_{c_k}}||_{C^m([-1+\epsilon, x_k-\epsilon])}+||U_{\nu_k, \theta}-\sqrt{2P_{c_k}}||_{C^m([x_k+\epsilon, 1-\epsilon])}\le C\nu_k.
%\end{equation}
%
%In particular, if $c_1=0$ with $x_k\to  -1$, or $c_2=0$ with $x_k\to 1$, or $x_k$ does not exists, then
%\begin{equation}\label{eq5_4_2_1}
%     \lim_{k\to \infty}||\frac{1}{2}U^2_{\nu_k, \theta}-P_{k}||_{L^{\infty}((-1, 1))}=0.
%\end{equation}
%\end{thm}
%\begin{proof}

\noindent{\emph{Proof of Theorem \ref{thm1_3} completed}}:

We will prove  Theorem \ref{thm1_3} (i) and (ii) in the case $c_1c_2=0$ and $c_3>c_3^*(c_1,c_2)$. 
Let $C$ be a positive constant, having the same dependence as specified in the theorem, which may vary from line to line. For convenience write $f_k=U_{\nu_k, \theta}$, $P_k=P_{c_k}$ and $h_k=\frac{1}{2}f_k^2-P_k$. Throughout the proof $k$ is large. %Let $\bar{x}$ be defined as in the theorem. Since $c_1,c_2\ge 0$, $c_3=c_3^*(c_1,c_2)=-\frac{1}{2}(c_1+2\sqrt{c_1c_2}+c_2)<0$, we have $c_1$ and $c_2$ are not both zero, so $-1\le \bar{x}\le 1$.

We first prove part (i) in this case. 
  Since $P_k\in J_0$, we have $P_k\ge 0$ on $[-1,1]$. %In particular, $P_k(-1)=2c_{k1}\ge 0$ and $P_k(1)=2c_{k2}\ge 0$. So $f_k(-1)=\tau_2(\nu_k,c_{k1})\le 0$ and $f_k(1)=\tau_1'(\nu_k,c_{k2})\ge 0$. So there exists some $x_k\in [-1,1]$ such that $f_k(x_k)=0$.
 By Lemma \ref{lem2_4}, there exists at most one $x_k\in (-1,1)$ such that $f_k(x_k)=0$. Moreover, if $x_k$ exists, then we have
\begin{equation}\label{eq5_4_3}
   f_k(x)<0 \textrm{ for }-1< x<x_k, \quad \textrm{ and }f_k(x)>0 \textrm{ for }x_k<x< 1.
\end{equation}
By Lemma \ref{lem2_1},
\begin{equation}\label{eq5_4_4}
|f_k|\le C.
\end{equation}
Since $c_1,c_2\ge 0$, $c_1c_2=0$, $c_3>c_3^*(c_1,c_2)$,  we have $\min_{[-1+\epsilon/2,1-\epsilon/2]}P_c>0$. By the convergence of $\{c_k\}$ to $c$,
\begin{equation}\label{eq5_4_5}
 \min_{[-1+\epsilon/2,1-\epsilon/2]}P_k\ge 1/C. %f_k(-1)\le -1/C, \quad f_k(1)\ge 1/C, \quad \min_{[-1,1]\setminus [\bar{x}-\epsilon, \bar{x}+\epsilon]}P_k\ge 1/C,
\end{equation}
 Using (\ref{eq5_4_3}) and (\ref{eq5_4_5}), by applying Lemma \ref{lem2_7} and Lemma \ref{lem2_7'} on each interval of $[-1+\epsilon/2, x_k-\epsilon/2]$ and $[x_k+\epsilon/2, 1-\epsilon/2]$ separately, we deduce
\begin{equation*}%\label{eq5_4_2}
   ||U_{\nu_k, \theta}+\sqrt{2P_{c_k}}||_{C^m([-1+\epsilon, x_k-\epsilon])}+||U_{\nu_k, \theta}-\sqrt{2P_{c_k}}||_{C^m([x_k+\epsilon, 1-\epsilon])}\le C\nu_k.
\end{equation*}

\medskip

  Next, we prove %(\ref{eq5_4_1}). %Suppose $x_k\to \hat{x}\in [-1,1]$ as $k\to \infty$.
  \begin{equation}\label{eq5_4_1}
    ||\frac{1}{2}U^2_{\nu_k, \theta}-P_{c_k}||_{L^{\infty}((-1, x_k-\epsilon)\cup (x_k+\epsilon, 1))}\le C\nu_k^{1/2}.
\end{equation}
  Since $f_k$ is not $U_{\nu_k, \theta}^\pm$, we know from Theorem A that $f_k(-1)=\tau_1(\nu_k, c_{k1})$ and $f_k(1)=\tau_2'(\nu_k, c_{k2})$. In view of (\ref{eq_2}), we have
\begin{equation}\label{eq5_4_6}
   |f_{k}(-1)+\sqrt{2P_k(-1)}|+|f_{k}(1)-\sqrt{2P_k(1)}|\le C\nu_k.
\end{equation}

\medskip

   \textbf{Case} 1: $c_1=0$, $c_2>0$, $c_3>-\frac{1}{2}c_2$.

   In this case $P_c(x)=c_2(1+x)+c_3(1-x^2)$ in $(-1,1)$. So $P_c(-1)=0$, $P_c(1)=2c_2>0$, and $P_c'(-1)=c_2+2c_3>0$. Since $c_k\to c$,  there exists some $\delta>0$, such that % $P_k$ satisfies
      \begin{equation}\label{eq5_4_10}
        \frac{1}{2}P_{c}'(-1)(1+x)\le P_{k}(x)-P_{k}(-1)\le 2P_{c}'(-1)(1+x), \quad -1<x<-1+\delta.
   \end{equation}
 So $P_k(-1+\epsilon/2)>1/C$ and $P_k(1)>1/C$. By Lemma \ref{lem5_1}, we have
   \begin{equation}\label{eq5_4_7}
      P_k(x)\ge 1/C, \quad -1+\epsilon/4\le x\le 1.
   \end{equation}
  We discuss the cases when $x_k+1\ge  \epsilon/4$ and $x_k+1<\epsilon/4$ separately.

   We first discuss the case when $x_k+1\ge  \epsilon/4$.
   %In view of (\ref{eq_2}), we have $f_k(-1)\le -1/C$ and $f_k(1)\ge 1/C$.
   %If $(x_k, 1)$ is not empty, %we have $P_k\ge 1/C$ on $[x_k+\epsilon/2,1]$ for large $k$.
   %We start from the estimate of $|h_k|$ on $[x_k+\epsilon,1]$. 
   We have $P_k\ge 1/C$ on $[x_k+\epsilon/4,1]$ for large $k$. Applying Corollary \ref{cor2_3} on $[x_k+\epsilon/4, 1]$, using (\ref{eq5_4_3}), (\ref{eq5_4_4}) and (\ref{eq5_4_7}), we have that
   \begin{equation}\label{eq5_4_8}
    1/C\le f_k\le C \textrm{ on }\left(x_k+\epsilon/2, 1\right).
   \end{equation}
  Using (\ref{eq5_4_8}), applying Lemma \ref{lem2_6} on $(x_k+\epsilon/2, 1)$, we have
   \begin{equation}\label{eq5_4_9}%\label{eq4_3_10}
      \max_{[x_k+\epsilon, 1]}|h_k|\le C\nu_k
      \end{equation}% on $(x_k+\epsilon, 1)$. %we have $P_k>C\epsilon^2$ and $f_k>0$ on $(x_k, 1)$ for large $k$. By Lemma \ref{lem2_1}, and Corollary \ref{cor2_3}, we have $\epsilon\le f_k\le C$ on $(\hat{x}+\epsilon/2, 1)$. Then

   Next, %for  $x<\hat{x}-\epsilon/2$ we have $f_k<0$ for large $k$.
 let $a_k=\nu_k^{1/2}/\alpha$ for some positive constant $\alpha$ to be determined. Since $P_k(-1)\ge 0$, it follows from (\ref{eq5_4_10}) and (\ref{eq5_4_7}) that
% \begin{equation*}%\label{eq5_4_11}
    $P_k(x)\ge a_k/C$ for $x$ in $[-1+a_k,1]$.
% \end{equation*}
 By Lemma \ref{lem2_5}, there exists some  $s_k\in (-1+a_k, -1+2a_k)$ and $t_k\in (x_k-\epsilon, x_k-\epsilon/2)$, such that $|h_k(s_k)|\le C\nu_k/a_k=C\alpha^2 a_k$ and $|h_k(t_k)|\le C\nu_k$. It follows from (\ref{eq4_3_14}) that
       \[
           \frac{1}{2}f^2_k(t_k)\ge P_{c_k}(t_k)-|h_k(t_k)|\ge 1/C\alpha^2 a_k-C\nu_k%\left(\frac{1}{2}|c_3|-C\alpha^3\right)a_k^2.
       \]
%Fix $\alpha^3=\frac{|c_3|}{4C}$.
 By (\ref{eq5_4_3}), we have $f_{k}(t_{k})<- \sqrt{a_k}/C$.  Using (\ref{eq5_4_4}), applying Lemma \ref{lem2_2'} on $[s_k,t_k]$, we have $f_k(x)\le -\sqrt{a_k}/C$ on $[s_k,t_k]$.  Using  $|h_k(s_k)|\le C\nu_k^{1/2}$ and $|h_k(t_k)|\le C\nu_k$,  applying Lemma \ref{lem2_6} on  $[s_k, t_k]$, we have
      \begin{equation}\label{eq5_4_12}
         \max_{[-1+2a_k, x_k-\epsilon]}|h_k|\le \max_{[s_k,t_k]}|h_k| \le  C\nu_k^{1/2}.
      \end{equation}
      Now we have that $|h_k(-1+2a_k)|\le Ca_k$ and $|h_k(-1)|=|\frac{1}{2}\tau^2_1(\nu_k,c_{k1})-2c_{k1}|\le Ca_k$. Notice that $f_k<0$ on $[-1,-1+2a_k]$, $P_k(-1)\ge 0$.  % and $P_k(x)=P_k(-1)-c_{k3}(x-\bar{x}_k)^2$. %$\min_{[\bar{x}_k-2a_k,\bar{x}_k+2a_k]} P_k=P_k(\bar{x}_k)$, $P'_k(\bar{x}_k)=0$.
     Using (\ref{eq5_4_10}), applying Lemma \ref{lem2_10} on $[-1,-1+2a_k]$ with $\alpha=1$, we have that
       \begin{equation}\label{eq5_4_13}
          \max_{[-1,-1+2a_k]}|h_k|\le Ca_k+C\nu_k/a_k\le C\nu_k^{1/2}.
       \end{equation}
       By (\ref{eq5_4_9}), (\ref{eq5_4_12}) and (\ref{eq5_4_13}), we have proved (\ref{eq5_4_1}) when $\hat{x}>-1$.

    Next, if $x_k+1<\epsilon/4$,
    %if $\hat{x}=-1$, then $-1\in (x_k-\epsilon/2, x_k+\epsilon/2)$,
    similar as (\ref{eq5_4_8}) we have
     $1/C\le f_k\le C \textrm{ on }\left(x_k+\epsilon/2, 1\right)$. 
    Using this and (\ref{eq5_4_6}), applying Lemma \ref{lem2_6} on  $(x_k+\epsilon/2, 1)$, (\ref{eq5_4_1}) is proved.

    % \medskip

     \textbf{Case} 2: $c_1>0$, $c_2=0$, $c_3>-\frac{1}{2}c_1$. The proof is similar as Case 1.

     \textbf{Case} 3: $c_1=c_2=0$, $c_3>0$. Similar as Case 1 we have $|h_k|\le C\nu_k^{1/2}$ on $[-1,0]\setminus[x_k-\epsilon, x_k+\epsilon]$, and similar as Case 2 we have $|h_k|\le C\nu_k^{1/2}$ on $[0,1]\setminus[x_k-\epsilon, x_k+\epsilon]$. 
      We have by now proved (\ref{eq5_4_1}).

%%%%%%%%%%%%%%%%%%%%

 By (\ref{eq5_4_1}) and (\ref{eq_1}),  for any $\epsilon>0$, there exists some constant $C>0$, depending only on $\epsilon$ and an upper bound of $|c|$, such that $|f'_k|\le C\nu_k^{-\frac{1}{2}}$ on $[-1+\epsilon, x_k-\epsilon]\cup [x_k+\epsilon, 1-\epsilon]$, so $|h'_k|=|f_kf'_k-P'_k|\le C\nu_k^{-\frac{1}{2}}$. So we have 
  \begin{equation*}%\label{eq4_3_3}
       ||\frac{1}{2}U^2_{\nu_k, \theta}-P_{c_k}||_{C^1(([-1+\epsilon, x_k-\epsilon]\cup [x_k+\epsilon, 1-\epsilon])}\le C\nu_k^{-\frac{1}{2}},
\end{equation*}
%and, consequently, for any $0<\beta<1$
 %  (\ref{eq4_3_3}) is proved.
        By interpolation %Lemma \ref{lem6_1},
        for any $x, y\in (-1+\epsilon, 1-\epsilon)$ and $0<\beta<1$,
       \[
          \frac{|h_k(x)-h_k(y)|}{|x-y|^{\beta}}\le 2||h_||^{1-\beta}_{L^{\infty}(-1+\epsilon, 1-\epsilon)}||h_k'||^{\beta}_{L^{\infty}(-1+\epsilon, 1-\epsilon)}\le C\nu_k^{\frac{1}{2}(1-\beta)}\nu_k^{-\frac{1}{2}\beta}\le C\nu_k^{\frac{1}{2}-\beta}.
       \]
        We have 
    \begin{equation}\label{eq5_4_4_1}
   ||\frac{1}{2}U^2_{\nu_k, \theta}-P_{c_k}||_{C^{\beta}([-1+\epsilon, x_k-\epsilon]\cup [x_k+\epsilon, 1-\epsilon])}\le C\nu_k^{\frac{1}{2}-\beta}.
   \end{equation}
%%%%%%%%%%%%%%%%%%%%
Next, using (\ref{eq5_4_3}), (\ref{eq5_4_5}) and (\ref{eq5_4_1}), we then have (\ref{eq1_7_0}).
 Part (i) in this case follows in view of (\ref{eq5_4_1}) and (\ref{eq5_4_4_1}).

%*****************\marginpar{replace the below proof by "Similar as in section4"?}

 Now we prove part (ii) in this case. 
 If such $x_k$ exists and $x_k\to -1$ with  $c_1=0$, or such $x_k$ does not exist with $c_2>0=c_1$,  we can prove (\ref{eq1_7_2})
 using similar arguments as that for part (ii)  in   \emph{''Proof of Theorem \ref{thm1_3} continued"} in Section 4. If such $x_k$ exists and $x_k\to 1$ with $c_2=0$, or such $x_k$ does not exist with $c_1>0=c_2$, we can prove similarly (\ref{eq1_7_3}).  
 If such $x_k$ does not exist with $c_1=c_2=0$, we prove either  (\ref{eq1_7_2}) or (\ref{eq1_7_3}).

In this case, $f_k$ does not change sign on $(-1,1)$ and $P_c(-1)=P_c(1)=0$. If $f_k>0$ on $(-1,1)$ after passing to a subsequence,  we have, by Theorem \ref{thm1_2_2},
   $\limsup_{k\to \infty}||\frac{1}{2}(U^{\pm}_{\nu_k, \theta})^2-P_{k}||_{L^{\infty}(-1, 1)}=0$.
 So for any $\epsilon_0>0$, there exists some $\epsilon>0$, such that $||P_k||_{L^{\infty}(-1,-1+2\epsilon)}+||P_k||_{L^{\infty}[1-2\epsilon, 1]}<\epsilon_0$, and $||(U^{\pm}_{\nu_k, \theta})^2||_{L^{\infty}(-1,-1+2\epsilon)}+||(U^{\pm}_{\nu_k, \theta})^2||_{L^{\infty}[1-2\epsilon, 1]}<\epsilon_0$. Notice  $U_{\nu_k, \theta}^-\le f_k\le U^+_{\nu_k, \theta}$, we then have $||f_k-\sqrt{2P_k}||_{L^{\infty}(-1,-1+2\epsilon)}+||f_k-\sqrt{2P_k}||_{L^{\infty}[1-2\epsilon, 1]}<2\epsilon_0$.
  We also have $P_c\ge 1/C$ on $[-1+\epsilon, 1-\epsilon]$ and $f_k>0$  on $[-1+\epsilon, 1-\epsilon]$. 
By Corollary \ref{cor2_3}, we have $f_k\ge 1/C$ on $[-1+2\epsilon, 1-2\epsilon]$. Notice $|f_k(-1+2\epsilon)-\sqrt{2P_k(-1+2\epsilon)}|+|f_k(1-\epsilon)-\sqrt{2P_k(1-2\epsilon)}|\le 2\epsilon_0$,  by Corollary \ref{cor2_2} we have  $||f_k-\sqrt{2P_k}||_{L^{\infty}[-1+2\epsilon, 1-2\epsilon]}<C\epsilon_0$. 

If $f_k<0$ on $(-1,1)$ after passing to a subsequence, similar as the above we have (\ref{eq1_7_3}). 
   Part (ii) in this case is proved. The proof of Theorem \ref{thm1_3} is completed now. Part (iii) is proved in Section 3, part (i) and (ii) follows from (iii), \emph{"Proof of Theorem \ref{thm1_3} continued"} in Section 4, and the above.
\qed

\noindent{\emph{Proof of Theorem \ref{thm1_0}}}: 
For $0<\nu\le 1$ and $c\in \mathring{J}_{\nu}$, let 
\[
 u^{\pm}_{\nu,\theta}(c)=\frac{1}{\sin\theta}U_{\nu,\theta}^{\pm}(c),\quad
   u^{\pm}_{\nu,r}(c) =-\frac{d u_{\nu,\theta}^{\pm}}{d \theta} -  u_{\nu,\theta}^{\pm} \ctthe.
\]
%\[
%\begin{split}
%  & u^{\pm}_{\nu,\theta}(c)=\frac{1}{\sin\theta}U_{\nu,\theta}^{\pm}(c),\\
%  & u^{\pm}_{\nu,r}(c) =-\frac{d u_{\nu,\theta}^{\pm}}{d \theta} -  u_{\nu,\theta}^{\pm} \ctthe.
%   \end{split}
%\]
and 
\[
  p^{\pm}_{\nu}(c)=-\frac{1}{2}\left(\frac{d^2 u^{\pm}_{\nu,r}(c)}{d\theta^2} + (\cot\theta - u^{\pm}_{\nu,\theta}(c)) \frac{d u^{\pm}_{\nu,r}(c)}{d\theta} + (u^{\pm}_{\nu,r}(c))^2 +(u^{\pm}_{\nu,\theta}(c))^2\right).
\]
By Theorem 1.1 of \cite{LLY2}, $\{(u_{\nu}^{\pm}(c), p_{\nu}^{\pm}(c))\}_{0<\nu\le 1}$ belong to $C^{0}(\mathring{J}_{\nu}\times (0,1], C^m(\mathbb{S}^2\setminus(B_{\epsilon}(S)\cup B_{\epsilon}(N))))$  for every integer $m\ge 0$.  
  By Theorem \ref{thm1_1}, there exists some constant $C$, which depends only on $K, \epsilon$ and $m$, such that 
\[
  ||U^{+}_{\nu, \theta}-\sqrt{2P_{c}}||_{L^{\infty}(-1,1)}+||U^-_{\nu, \theta}+\sqrt{2P_{c}}||_{L^{\infty}(-1,1)}\le C\nu,
  \]
  and 
  \[
  ||U^+_{\nu, \theta}-\sqrt{2P_{c}}||_{C^m(-1,1-\epsilon)}+||U^-_{\nu, \theta}+\sqrt{2P_{c}}||_{C^m(-1+\epsilon,1)}\le C\nu.
\]
Theorem \ref{thm1_0}(i) follows from the above.

Now we prove part (ii). By Theorem A, there exist a unique $U_{\theta}:=U_{\nu, \theta}(c,\theta_0)$ of (\ref{eq:NSE}) satisfying, with $x_0=\cos \theta_0$, that 
\[
   U_{\theta}(-1)=\tau_1(\nu,c_1)<0, \quad U_{\theta}(1)=\tau_2(\nu,c_2)>0, \quad U_{\theta}(x_0)=0.
\]
For every $\epsilon>0$, we have, by Theorem \ref{thm1_3}, that
 \[
  ||U_{\nu, \theta}-\sqrt{2P_{c}}||_{C^m(x_0+\epsilon,1-\epsilon)}+||U_{\nu, \theta}+\sqrt{2P_{c}}||_{C^m(-1+\epsilon,x_0-\epsilon)}\le C\nu.
\]
The estimate in part (ii) follows from the above.
\qed

\section{Proof of Theorem \ref{thm:BL:1}}\label{sec6}

In this section, we give the %proof of Theorem \ref{thm:BL:1}.

\noindent{\emph{Proof of Theorem \ref{thm:BL:1}}}:
Define
\begin{equation}\label{eq6_0}
			w_k(x):=\sqrt{2P_{c_k}(x_k)} \tanh \Big( \frac{\sqrt{2P_{c_k}(x_k)} \cdot (x-x_k)}{2 (1-x_k^2)\nu_k} \Big).
	\end{equation}
By computation, we know that $w_k(x_k)=0$ and
\begin{equation}\label{eq6_0_0}
   \nu_k(1-x_k^2)w_k'+\frac{1}{2}w_k^2=P_k(x_k).
\end{equation}

{\it Step 1.} We prove
\begin{equation}\label{eq6_2_1}
	|U_{\nu_k, \theta} - w_k| \le C\nu_k |\ln \nu_k|^{2}, \quad x_k-K\nu_k|\ln\nu_k|(1-x_k^2)<x< x_k+K\nu_k|\ln\nu_k|(1-x_k^2)).
\end{equation}
Let $C$ denote a constant depending only on $c$, $K$ and $\hat{x}$ which may vary from line to line. For convenience denote $f_k:=U_{\nu_k, \theta}$ and  $P_k:=P_{c_k}$.
  By Lemma \ref{lem2_1} and Lemma \ref{lem2_4}, we have that
\begin{equation}\label{eq6_2_2}
   0<f_k<C, \textrm{ in }(x_k, 1), \textrm{ and } -C<f_k<0, \textrm{ in }(-1,x_k).
\end{equation}
	Let
	\begin{equation}\label{eq6_2_4}
		y := \frac{x-x_k}{\nu_k(1-x_k^2)}, \qquad \tilde{f}_k(y):=f_k(x), \quad \tilde{w}_k(y):=w_k(x).
	\end{equation}
	Then for $x_k-K\nu_k|\ln\nu_k|(1-x_k^2)\le x\le x_k+K\nu_k|\ln\nu_k|(1-x_k^2)$, we have $-K|\ln\nu_k|\le y\le K|\ln\nu_k|$.
	By $f_k(x_k)=0$ and (\ref{eq_1}), we know that $\tilde{f}_k(0)=0$ and
	\begin{equation}\label{eq6_2_5}
	\begin{aligned}
		& (1-2x_k\nu_ky-\nu_k^2y^2(1-x_k^2)) \tilde{f}_k'(y) + 2 \nu_k (x_k + y \nu_k(1-x_k^2)) \tilde{f}_k(y) + \frac{1}{2} \tilde{f}_k^2(y) \\
		= & P_{c_k} = P_{c_k}(x_k) + P_{c_k}'(x_k)\nu_k y(1-x_k^2)+ \frac{1}{2} P_{c_k}''(x_k) \nu_k^2 y^2(1-x_k^2)^2.
	\end{aligned}
	\end{equation}
 By (\ref{eq6_0}) and (\ref{eq6_0_0}), we have $\tilde{w}(0)=0$ and
	\begin{equation}\label{eq6_2_6}
		\tilde{w}_k'(y) + \frac{1}{2}  \tilde{w}_k^2(y) = P_{c_k}(x_k).
	\end{equation}
	Set $g_k(y) := \tilde{f}_k(y) - \tilde{w}_k(y)$, then by (\ref{eq6_2_5}) and (\ref{eq6_2_6}), we have $g_k(0)=0$ and
	\begin{equation}\label{eq6_2_7}
		g_k'(y) + h_k(y) g_k(y) = H_k(y),
	\end{equation}
	where $h_k(y) = \frac{1}{2} (\tilde{f}_k(y) + \tilde{w}_k(y))$ and
	\[
	\begin{aligned}
		H_k(y) = & P_{c_k}'(x_k) \nu_ky + \frac{1}{2} P_{c_k}''(x_k) \nu^2_k y^2 - 2\nu_k (x_k + y \nu_k(1-x_k^2))\tilde{f}_k(y) \\
		& + \tilde{f}'_k(y)(2x_k\nu_ky + \nu^2_k y^2(1-x_k^2)).
	\end{aligned}
	\]
By (\ref{eq6_0})
and (\ref{eq6_2_2}), we have $|h_k(y)|\le C$ for $|y|\le K|\ln\nu_k|$. By (\ref{eq6_2_5}) and (\ref{eq6_2_2}), we have $|\tilde{f}_k'(y)|\le C$ for $|y|\le K|\ln\nu_k|$ and $k>>1$. So
 $|H_k|\le C \nu_k|\ln \nu_k|$ for $|y|\le K|\ln \nu_k|$ and $k>>1$.
	Hence, from the estimates of $h_k$, $H_k$, (\ref{eq6_2_7}) 
	and the fact that $g_k(0)=0$, we have
	\[
		|g_k(y)| = e^{- \int_{0}^{y} h(s)ds} \bigg| \int_0^y e^{\int_{0}^{s} h(t)dt} H(s) ds \bigg|
		\le C  \nu_k |\ln \nu_k|^2, \qquad |y|\le K|\ln \nu_k|.
	\]
	Therefore, the estimate (\ref{eq6_2_1}) is proved.
	
{\it Step 2. } We prove that there exists some $K>0$ and small $\epsilon>0$, independent of $k$, such that
\begin{equation}\label{eq6_1_01}
	|U_{\nu_k, \theta} + \sqrt{2P_{c_k}} | \le C\nu_k^{\alpha(c)} |\ln \nu_k|^{2\kappa(c)}, \quad b_k\le x\le x_k-K\nu_k|\ln\nu_k|(1-x_k^2),
\end{equation}
\begin{equation}\label{eq6_1_02}
	|U_{\nu_k, \theta} - \sqrt{2P_{c_k}} | \le C\nu_k^{\alpha(c)} |\ln \nu_k|^{2\kappa(c)}, \quad x_k+K\nu_k|\ln\nu_k|(1-x_k^2)\le x\le d_k,
\end{equation}
where $b_k=\max\{-1,x_k-\epsilon\}$ and $d_k=\min\{1,x_k+\epsilon\}$.

It is sufficient to prove (\ref{eq6_1_01}) since the other estimate can be obtained similarly. We first prove that (\ref{eq6_1_01}) holds at the endpoints $x=b_k$ and $x=b_k':=x_k-K\nu_k|\ln\nu_k|(1-x_k^2)$. For convenience denote $f_k:=U_{\nu_k, \theta}$, $P_k=P_{c_k}$ and $h_k=\frac{1}{2}f_k^2-P_k$.
  Since $x_k\to\hat{x}$, $P_c(\hat{x})>0$, we can chose $\epsilon>0$ small, such that
\begin{equation}\label{eq6_1_2}
   P_k\ge 1/C, \qquad x\in  (x_k-2\epsilon, x_k+2\epsilon).
\end{equation}
	By Theorem \ref{thm1_3} (i), we have
   $|h_k| \le C \nu_k^{\alpha(c)}$ for $-1\le x\le x_k-\epsilon$
where $\alpha(c)$ is given by (\ref{eq_alpha}).
 Using this, (\ref{eq6_2_2}) and (\ref{eq6_1_2}), we have that
	\begin{equation}\label{eq6_1_3_2}
   \Big| f_k (b_k)+ \sqrt{2P_k(b_k)} \Big| \le C \nu_k^{\alpha(c)}. 	\end{equation}
Let $K$ be a positive constant to be determined later.
It is easy to see that
\[
\begin{split}
	|w_k(b_k') + \sqrt{2P_{k}(x_k)}|
         \le & C e^{ -K |\ln \nu_k| \frac{\sqrt{2P_{k}(x_k)} }{2}}=C \nu_k^{\frac{K}{2}\sqrt{2P_{k}(x_k)} } \le C\nu_k,
\end{split}
\]
as long as $K\sqrt{2P_{k}(x)} \ge 2$ for any $x\in[b_k,d_k]$ and $k$ sufficiently large.
   Thus by Step 1, we have
 \begin{equation}\label{eq6_1_15}
 \begin{split}
 & \Big| f_k (b_k') + \sqrt{2P_{k}(b_k')} \Big|  \\
		\le & |f_k (b_k')- w_k(b_k')| + |w_k(b_k')+ \sqrt{2P_{k}(x_k)}|  + \Big| \sqrt{2P_{k}(x_k)} - \sqrt{2P_{k}(b_k')} \Big|\\
		\le & C\nu_k|\ln\nu_k|^2.
\end{split}
 \end{equation}
By (\ref{eq6_1_2}) and (\ref{eq6_1_15}), $f_k (b_k') \ge 1/C$. Then using this, (\ref{eq6_1_2}) and (\ref{eq6_2_2}), applying Lemma \ref{lem2_2'} on $[b_k, b_k']$, we have
\begin{equation}\label{eq6_1_16}
   -C\le f_k\le -1/C, \quad b_k \le x\le b_k'.  %x_k+\epsilon
\end{equation}
By (\ref{eq6_1_3_2}), (\ref{eq6_1_15}), (\ref{eq6_1_16}), applying Corollary \ref{cor2_2} on $[b_k,b_k']$, we know that (\ref{eq6_1_01}) holds on $[b_k,b_k']$. Similar argument implies (\ref{eq6_1_02}).
   Theorem \ref{thm:BL:1} follows from the above two steps and Theorem \ref{thm1_3}.
\qed

%By Theorem 1.3 in [2], for any fixed $\nu$ and $c$, the solutions $\{ U_{\nu, \theta}^{c,\gamma} \}_{\gamma^-(c)<\gamma<\gamma^-(c)}$ foliate the region
%	$
%		\left\{ (x,y)| -1<x<1, U_{\nu, \theta}^{c,\gamma^-(c)}(x) < y < U_{\nu, \theta}^{c,\gamma^+(c)}(x) \right\}
%	$.
%	This fact implies that the zero of $U_{\nu, \theta}^{c,\gamma}$ depends on $\gamma$.
%	Actually, for any $\delta_k\to 0$, there exists $\{U_{\nu_k, \theta}\}$ such that $U_{\nu_k, \theta}(x_k)=0$ and $|x_k-\hat{x}|=\delta_k$.
%   By (\ref{eq6_2_4}), the length scale of the layers is $\nu_k(1-x_k^2)$. 	
%	From the above foliation result, if $\hat{x}=\pm 1$, we know that for any $\epsilon_k=o(\nu_k)$, there exists $\{U_{\nu_k, \theta}\}$ such that $U_{\nu_k, \theta}(x_k)=0$ and $\epsilon_k=\nu_k(1-x_k^2)$.

\FloatBarrier
\section{Illustrations on the interior transition layer}\label{sec:illu}

Similar to Jeffery-Hamel flows, in the two-dimensional plane flows between two non-parallel walls, a boundary layer occurs if we impose no-slip boundary conditions in a nozzle. 

For axisymmetric, $(-1)$-homogeneous solutions in three dimensional case, we consider a cone region $\Omega = \{(r,\theta,\phi)\in \mathbb{R}^+\times[0,\pi]\times[0,2\pi): 0\le\theta<\arccos x_0<\pi\}$ for some $x_0\in(-1,1)$, then $\partial \Omega$ corresponds to $x=x_0$. We consider solutions in $\Omega$ and impose no-slip boundary conditions on $\partial \Omega$. Since we are considering a first order differential equation of $U_\theta$, only one no-slip boundary condition can be imposed: $U_\theta|_{\partial \Omega}=0$. 

As shown in \cite{LLY2}, there exist solutions $U=(U_{\theta,i},0)$ of the Navier-Stokes equations, with viscosity $\nu_i\to 0$ satisfying
\begin{equation}\label{eq:ex:w:bc}
\left\{
\begin{aligned}
	\nu_i (1-x^2) U_\theta' + 2 \nu_i x U_\theta + \frac{1}{2} U_\theta^2 & = P_c(x) = \mu (x-\xi)^2, \\
	U_\theta (x_0) & = 0, 
\end{aligned}
\right.
\end{equation}
where $P_c(x)$ represents a quadratic polynomial, and $\xi, x_0\in(-1,1)$ are given constants such that $P_c(\xi)=P_c'(\xi)=0, \xi\not=x_0$. 

To illustrate the behavior of homogeneous solutions, we set $P_c(x)=2(x-2/3)^2$, $x_0=0$, $\xi=2/3$, and $\nu_i\to 0$, for example. Then $U_{\theta,i}$ are illustrated in Figure \ref{fig:1} for $\nu=1,1/8,1/20,1/50$, respectively. The boundary conditions in (\ref{eq:ex:w:bc}) and Theorem \ref{thm1_3} implies that an interior layer happens in a neighborhood of $x_0$. We can see in the figure that 
\begin{equation}\label{eq:limit}
	U_\theta\to -\sqrt{2P_c(x)}, \; x\in (-1,x_0-\epsilon),\qquad 
	U_\theta\to +\sqrt{2P_c(x)}, \; x\in (x_0+\epsilon,1). 
\end{equation}
Since $c_3=c_3^*(c_1,c_2)$ in the example, we know by Theorem 1.8 that in a neighborhood of $x=\xi$, where $P_c$ vanishes, the convergence rate $|U_{\theta,i}-\sqrt{2P_c}|\sim \nu_i^{2/3}$. 

The streamlines of $U_{\theta,i}$ 
% in the upper half-space 
are illustrated in Figure \ref{fig:2} for $\nu=1,1/8,1/20,1/50$, respectively. By (\ref{eq:limit}), $U_{\theta,i}\to \pm 2|x-2/3|=V_\theta$, whose streamlines are illustrated in Figure \ref{fig:3}. Notice that $\pm\sqrt{P_c(x)}$ are not smooth solutions of the Euler equations, since they are singular at $x=\xi$. On the other hand, $\pm 2(x-2/3)\in C^\infty(-1,1)$ are smooth solutions of the Euler equations, whose streamlines are illustrated in Figure \ref{fig:4}. 
%\FloatBarrier

%From Figure \ref{fig:1}-\ref{fig:4}, we can see that as $\nu_i\to 0$, the $(-1)$-homogeneous, axisymmetric solutions of Navier-Stokes equations (except $U_{\theta,i}^{\pm}$) converges to solutions of Euler equations, while 
%\FloatBarrier
\begin{figure}
\begin{center}
	\captionsetup{width=.87\linewidth}
	\includegraphics[width=12cm]{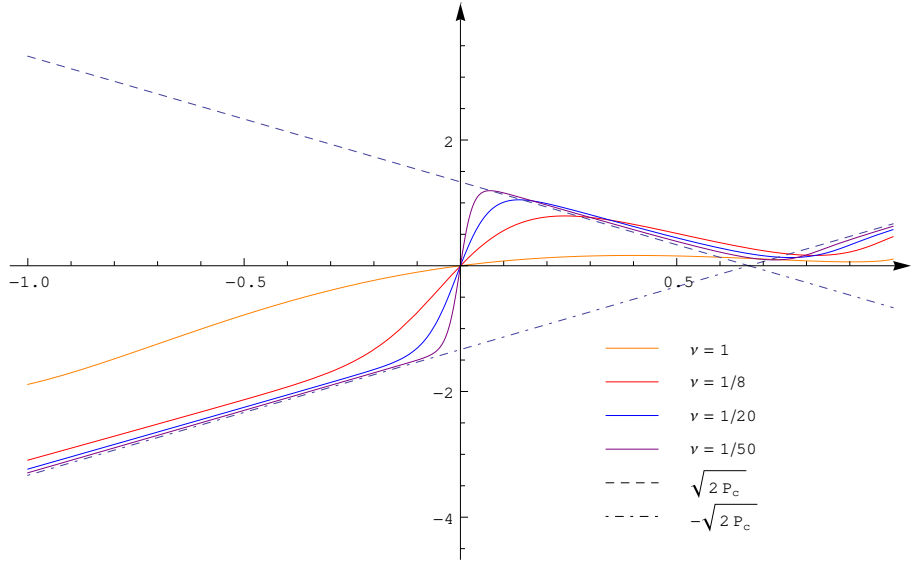}
	\caption{When $P_c=2(x-2/3)^2$, $U_\theta(0)=0$, the graph of $U_\theta$ for $\nu=1, 1/8, 1/20, 1/50$, respectively. }
	\label{fig:1}
\end{center}
\end{figure}

\begin{figure}
\begin{center}
	\begin{subfigure}[t]{0.5\textwidth}
		\centering
		\includegraphics[width=7cm]{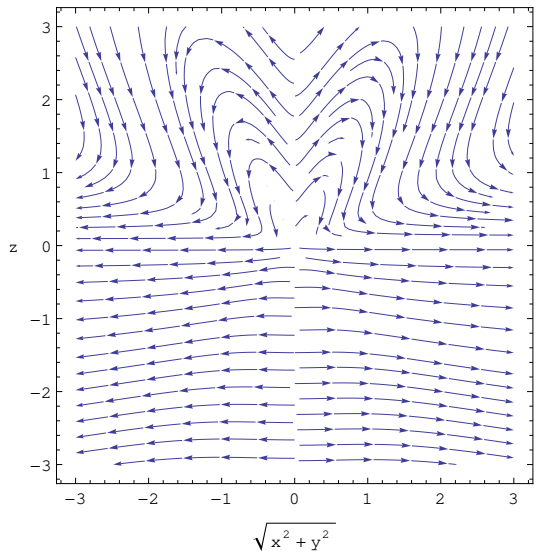}
		\caption{$\nu=1$}
	\end{subfigure}%
	~
	\begin{subfigure}[t]{0.5\textwidth}
		\centering
		\includegraphics[width=7cm]{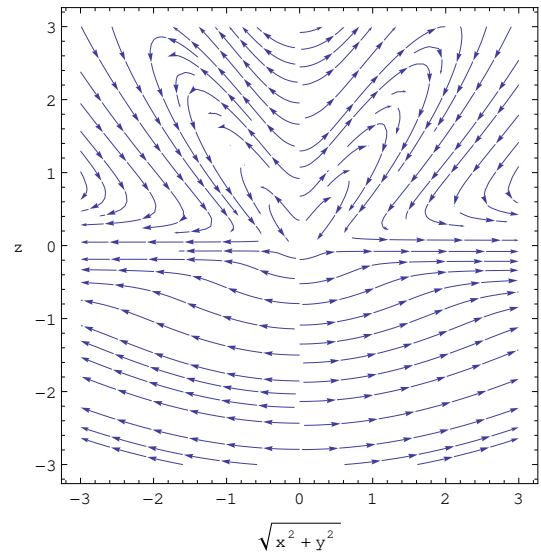}
		\caption{$\nu=1/8$}
	\end{subfigure}%
	\vspace{10pt}
	\newline
	\begin{subfigure}[t]{0.5\textwidth}
		\centering
		\includegraphics[width=7cm]{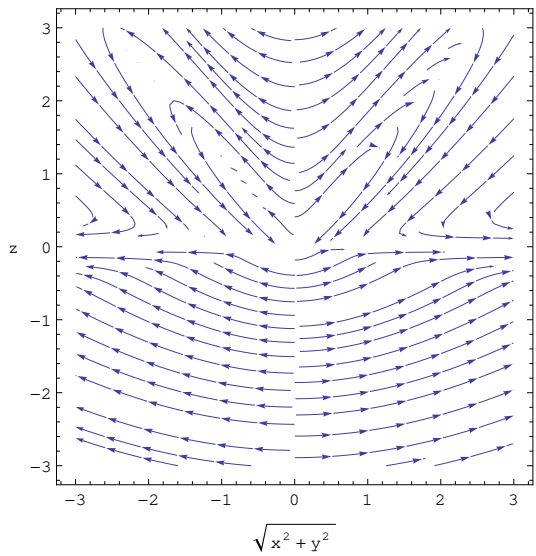}
		\caption{$\nu=1/20$}
	\end{subfigure}%
	~
	\begin{subfigure}[t]{0.5\textwidth}
		\centering
		\includegraphics[width=7cm]{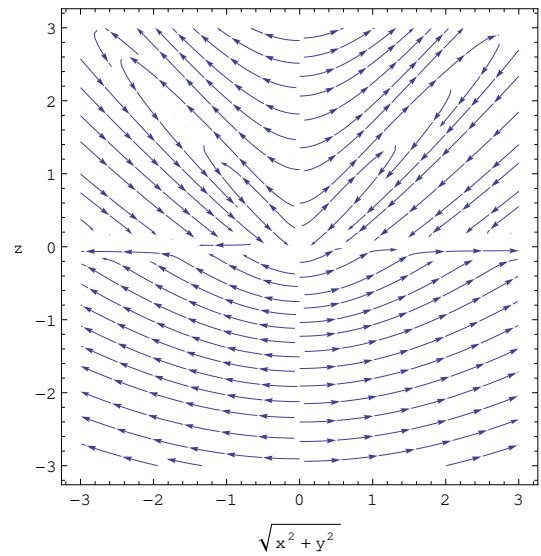}
		\caption{$\nu=1/50$}
	\end{subfigure}%
	\captionsetup{width=.87\linewidth}
	\caption{The streamlines of solutions of Navier-Stokes equation for $\nu=1, 1/8, 1/20, 1/50$, respectively. }
	\label{fig:2}
\end{center}
\end{figure}

\begin{figure}
\begin{center}
	\begin{subfigure}[t]{0.5\textwidth}
		\centering
		\includegraphics[width=7cm]{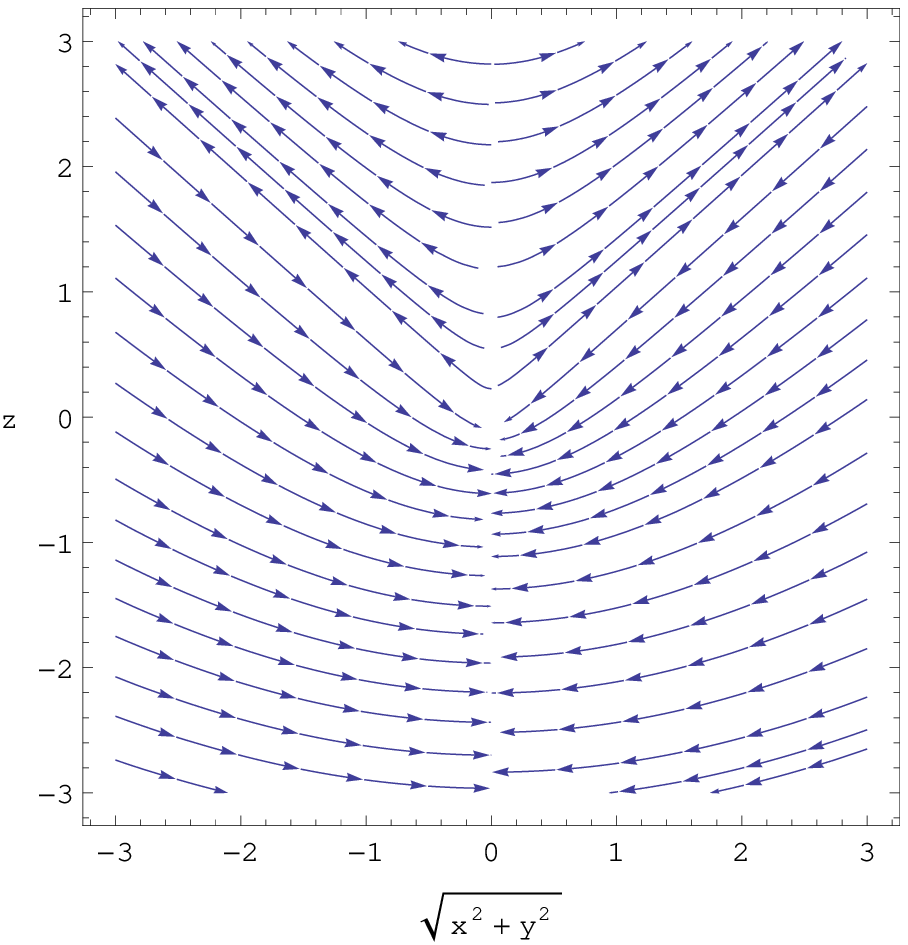}
		\caption{$V_\theta=\sqrt{P_c}$}
	\end{subfigure}%
	~
	\begin{subfigure}[t]{0.5\textwidth}
		\centering
		\includegraphics[width=7cm]{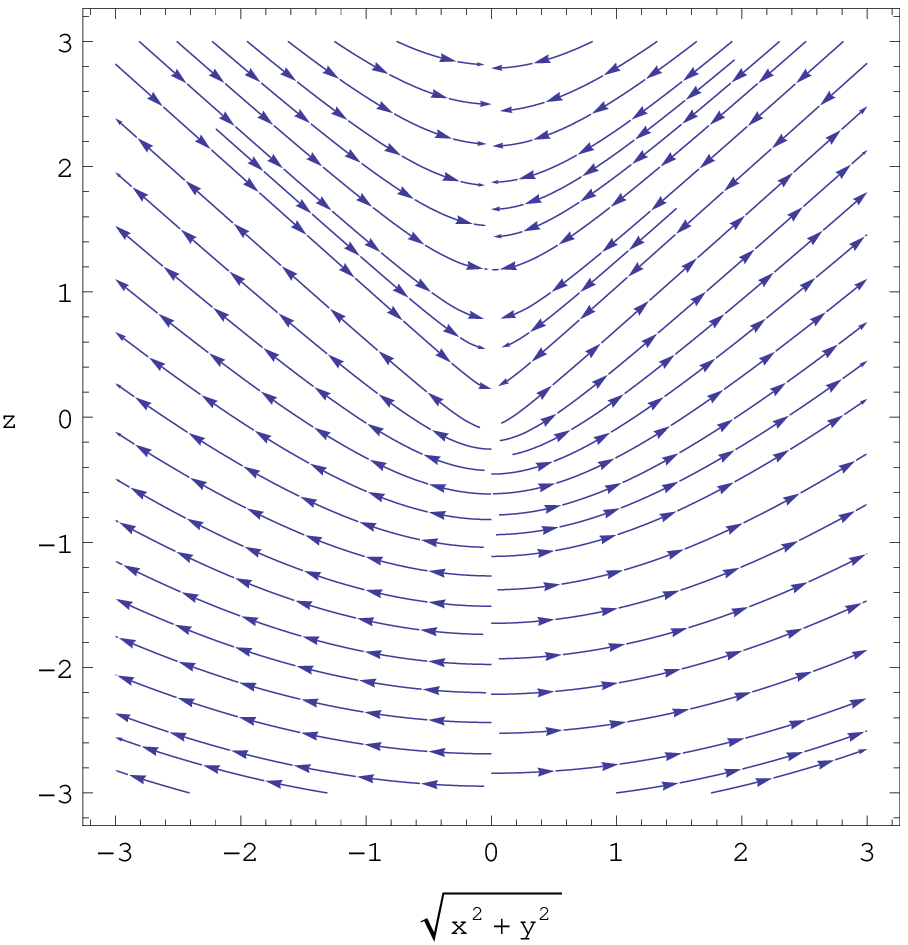}
		\caption{$V_\theta=-\sqrt{P_c}$}
	\end{subfigure}%
	\caption{The streamlines of $\pm\sqrt{P_c}$, the limit of $U_{\theta,i}$ as $\nu_i\to 0$. }
%	\caption{The streamlines of $\sqrt{P_c}$, the limit of $U_{\theta,i}$ as $\nu_i\to 0$. }
	\label{fig:3}
\end{center}
\end{figure}

\begin{figure}
\begin{center}
	\begin{subfigure}[t]{0.5\textwidth}
		\centering
		\includegraphics[width=7cm]{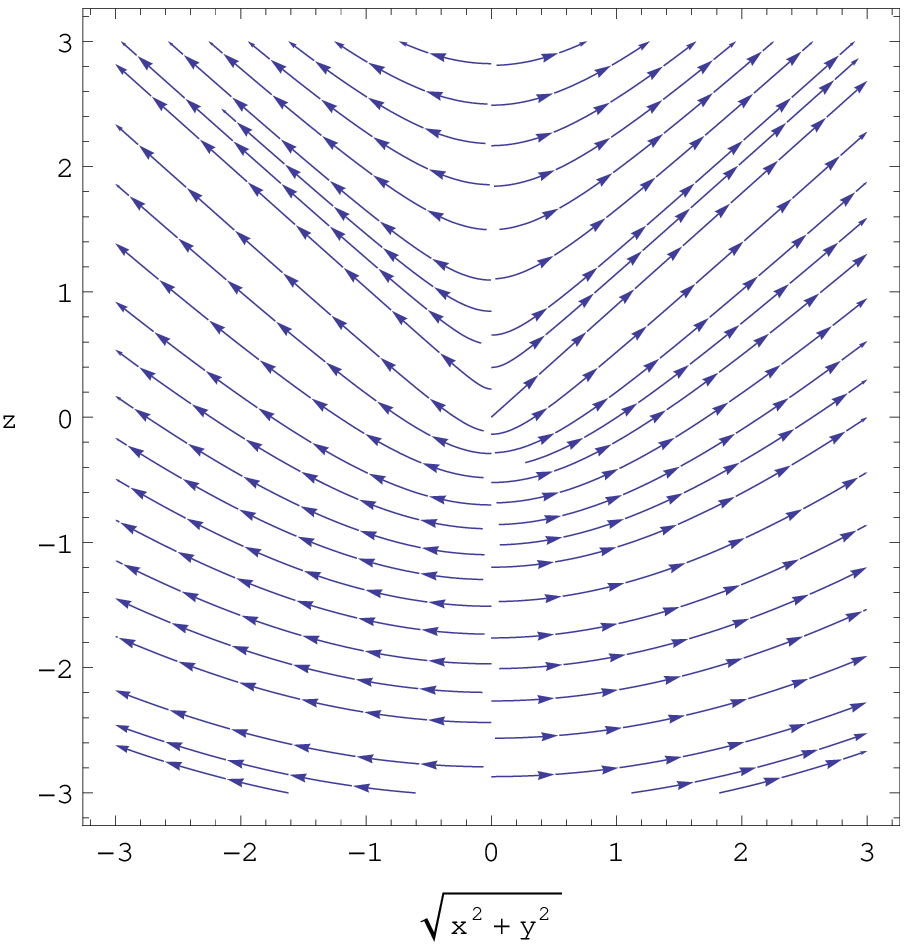}
		\caption{$V_\theta=2(x-2/3)$}
	\end{subfigure}%
	~
	\begin{subfigure}[t]{0.5\textwidth}
		\centering
		\includegraphics[width=7cm]{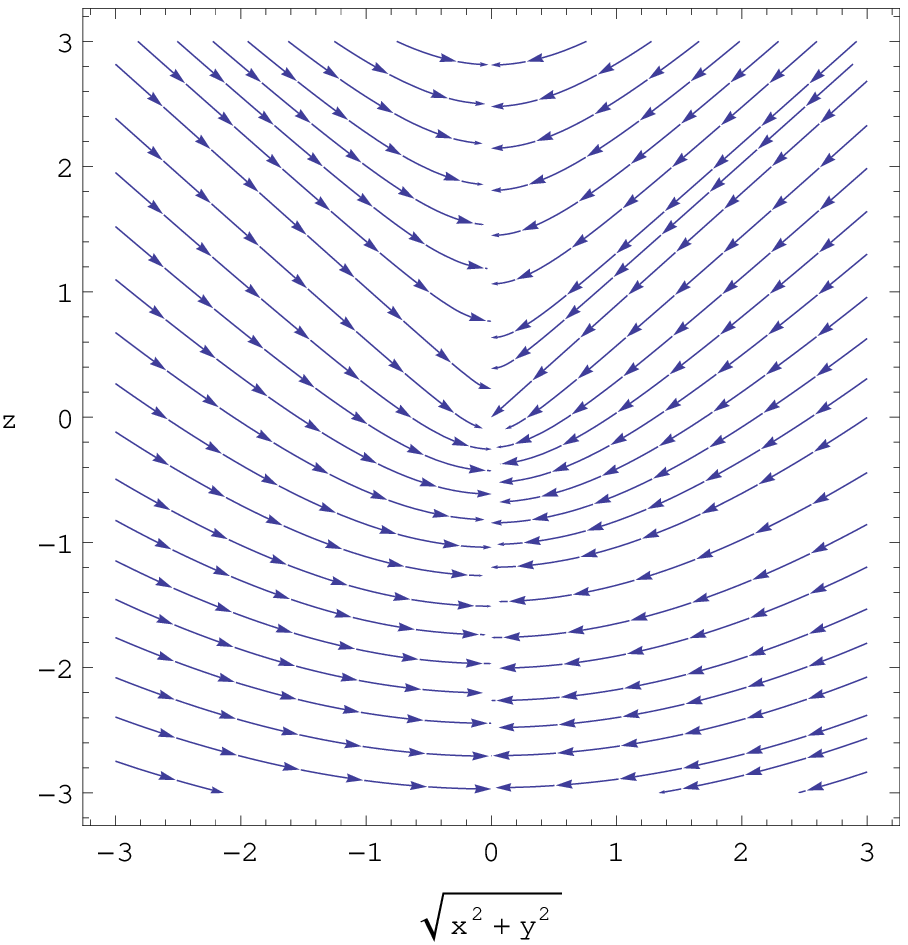}
		\caption{$V_\theta=-2(x-2/3)$}
	\end{subfigure}%
	\caption{The streamlines of the smooth solutions of Euler equation. }
	\label{fig:4}
\end{center}
\end{figure}

\FloatBarrier

\section{Appendix}\label{sec7}

     \begin{lem}\label{lem7_1}
        Let $c\in \mathbb{R}^3$, then 
        
        (i) $P_c\ge 0$ on $[-1,1]$ if and only if $c\in J_0$.
        
        (ii) $P_c>0$ on $[-1,1]$ if and only if $c\in \mathring{J}_0$.
        
        (iii) $\min_{[-1,1]}P_c=0$ if and only if $c\in \partial J_0$.
        
        (iv)$P_c>0$ in $(-1,1)$ if and only if $c\in \mathring{J}_0\cup \partial'J_0$.
     \end{lem}
     \begin{proof}
     
     (i) For $c\in J_0$, we have $c_1,c_2\ge 0$, $c_3\ge c_3^*(c_1,c_2)$ and therefore using (\ref{eqP_1}) and (\ref{eqP_2}) $P_c\ge P_{(c_1,c_2)}^*\ge 0$ on $[-1,1]$. On the other hand, if $P_c\ge 0$ on $[-1,1]$, we have $c_1=\frac{1}{2}P_c(-1)\ge 0$ and $c_2=\frac{1}{2}P_c(1)\ge 0$. If $c_1,c_2>0$, then $\bar{x}:=\frac{\sqrt{c_1}-\sqrt{c_2}}{\sqrt{c_1}+\sqrt{c_2}} \in (-1,1)$ and, using  (\ref{eqP_1}), $0\le P_c(\bar{x})=(c_3-c_3^*(c_1,c_2))(1-\bar{x}^2)$. Thus $c_3\ge c_3^*(c_1,c_2)$ and $c\in J_0$. If $c_1=0$, then $P_c(-1)=0$ and therefore $c_2+2c_3=P'_c(-1)\ge 0$. So $c\in J_0$. If $c_2=0$, then $P_c(1)=0$ and therefore $-c_1-2c_3=P'_c(1)\le 0$. So $c\in J_0$. Part (i) is proved.
     
     (ii) If $c\in \mathring{J}_0$, then $c_1,c_2>0$, $c_3>c_3^*(c_1,c_2)$, and $\bar{x} \in (-1,1)$. The positivity of $P_c$ on $[-1,1]$ then follows from the expression (\ref{eqP_2}). On the other hand, if $P_c>0$ on $[-1,1]$, then $c_1=\frac{1}{2}P_c(-1)> 0$,  $c_2=\frac{1}{2}P_c(1)> 0$, and $\bar{x}\in (-1,1)$. It follows, using (\ref{eqP_2}), that $0< P_c(\bar{x})=(c_3-c_3^*(c_1,c_2))(1-\bar{x}^2)$ and therefore $c_3-c_3^*(c_1,c_2)>0$. We have proved that $c\in \mathring{J}_0$. Part (ii) is proved. 
     
     Part (iii) is a consequence of (i) and (ii).
     
     (iv) If $c\in \mathring{J}_0$, then we know from (ii) that $P_c>0$ on $[-1,1]$. If $c_1=c_2=0$, and $c_3>0$, then $P_c(x)=c_3(1-x^2)>0$ in $(-1,1)$. If $c_1>0$, $c_2=0$, $c_3\ge -c_1/2$, then $P_c(x)\ge c_1(1-x)-\frac{1}{2}c_1(1-x^2)=\frac{c_1}{2}(1-x)^2>0$ in $(-1,1)$. If $c_1=0$, $c_2>0$, $c_3\ge -c_2/2$, then $P_c(x)\ge c_2(1+x)-\frac{1}{2}c_2(1-x^2)=\frac{c_2}{2}(1+x)^2>0$ in $(-1,1)$.  On the other hand, if $P_c>0$ in $(-1,1)$, then $c\in J_0$ by part (i). We only need to prove that $c$ does not belong to $\partial J_0\setminus\partial'J_0$. Indeed, if $c\in \partial J_0\setminus\partial'J_0$, then $c=0$ or $c_1,c_2>0$ and $c_3=c_3^*(c_1,c_2)$. Clearly $c$ cannot be $0$. For the latter, we know from (\ref{eqP_1}) that $P_c=P_{(c_1,c_2)}^*$ has a zero point at $\bar{x} \in (-1,1)$. We have proved (iv).
\end{proof}

We then have
\begin{lem}\label{lem2_9}
   For any $0<\nu\le 1$ and $c\in J_{\nu}$, there exists some constant $C$, depending only on an upper bound of $|c|$, such that
   \begin{equation*}%\label{eq2_9_0}
       P_c(x)\ge -C\nu, \quad \forall -1\le x \le 1.
   \end{equation*}
\end{lem}
\begin{proof}
   For $c\in J_{\nu}$, we have $c_1\ge -\nu^2, c_2\ge -\nu^2$ and $c_3\ge \bar{c}_3:=\bar{c}_3(c_1,c_2;\nu)$. Let $\tilde{c}_1=c_1+\nu^2$, $\tilde{c}_2=c_2+\nu^2$, and $\tilde{c}_3^*=-\frac{1}{2}(\tilde{c}_1+2\sqrt{\tilde{c}_1\tilde{c}_2}+\tilde{c}_2)$. By (\ref{eqP_1}), $P_{(\tilde{c}_1, \tilde{c}_2, \tilde{c}_3^*)}\ge 0$ in $[-1,1]$. Since $c_3\ge \bar{c}_3\ge \tilde{c}_3^*-C\nu$, we have 
   \[
       P_c  \ge P_{(\tilde{c}_1,\tilde{c}_2,c_3)}-C\nu^2
                 \ge P_{(\tilde{c}_1,\tilde{c}_2, \tilde{c}_3^*)}-C\nu
                 \ge -C\nu, \quad \textrm{ in }[-1,1].
   \]
   \end{proof}

\FloatBarrier

\end{document}